\documentclass[11pt,reqno]{amsart}
\usepackage[a4paper, hmargin={2.7cm,2.7cm},vmargin={3.3cm,3.3cm}]{geometry}
\usepackage[english]{babel}
\usepackage{color}
\usepackage[noadjust]{cite}
\usepackage{amsmath}
\usepackage{amssymb}
\usepackage{amsfonts}
\usepackage{amscd}
\usepackage{dsfont}
\usepackage{amsthm}
\usepackage{enumerate}
\usepackage{amsthm}
\usepackage{todonotes}
\usepackage{mathrsfs}
\usepackage{etoolbox}

\usepackage[colorlinks=true, linkcolor=blue, citecolor=blue, urlcolor=blue]{hyperref}

\numberwithin{equation}{section}

\makeatletter
\patchcmd{\ttlh@hang}{\parindent\z@}{\parindent\z@\leavevmode}{}{}
\patchcmd{\ttlh@hang}{\noindent}{}{}{}
\makeatother

\newcommand\numberthis{\addtocounter{equation}{1}\tag{\theequation}}

\theoremstyle{plain}
\newtheorem{theorem}{Theorem}[section]
\newtheorem{lemma}[theorem]{Lemma}
\newtheorem{proposition}[theorem]{Proposition}
\newtheorem{corollary}[theorem]{Corollary}

\theoremstyle{definition}
\newtheorem{definition}[theorem]{Definition}
\newtheorem{examplex}[theorem]{Example}
\newenvironment{example}
  {\pushQED{\qed}\examplex}
  {\popQED\endexamplex}

\theoremstyle{remark}
\newtheorem{remark}[theorem]{Remark}

\newcommand{\supp}{\operatorname{supp}}

\newcommand{\Dom}{\operatorname{Dom}}

\newcommand{\esssup}{\mathop{\mathrm{ess\,sup}}}

\DeclareMathOperator{\Co}{Co}

\newcommand{\NN}{\mathbb{N}} % set of natural numbers
\newcommand{\Z}{\mathbb{Z}} % set of integers
\newcommand{\R}{\mathbb{R}} % set of real ?numbers
\newcommand{\dimN}{d} % topological dimension of the group \N
\newcommand{\hdim}{Q} % homogeneous dimension
\newcommand{\cqn}{\gamma}
\newcommand{\hdeg}{\nu} % homogeneous degree of the Rockland operator in question
\newcommand{\ord}{\sigma} % order of smoothness
\newcommand{\NA}{G} % semi-direct product of $\R^+$ acting on the graded group
\newcommand{\N}{N} % homogeneous group
\newcommand{\n}{\mathfrak{n}} % homogeneous Lie algebra
\renewcommand{\H}{\mathbb{H}_n} % Heisenberg group
\newcommand{\BMO}{\mathop{\mathrm{BMO}}}
\newcommand{\PT}{P^{p,q}_{\PTpar, \ord}} % standard Peetre-type space
\newcommand{\PTi}{P^{\infty,q}_{\PTpar, \ord}} % Peetre-type space with p = \infty
\newcommand{\PTii}{P^{\infty,\infty}_{\PTpar, \ord}} % Peetre-type space with p, q = \infty
\newcommand{\PTv}[3]{P^{#1,#2}_{\PTpar,#3}} % standard Peetre-type space with variable parameters
\newcommand{\MT}{L^{p,q}_{\PTpar, \ord}}
\newcommand{\ST}{Y^{p,q}_{\PTpar, \ord}}% standard solid space
\newcommand{\SC}{\mathcal{S}} % Schwartz class
\newcommand{\TD}{\mathcal{S}'} % space of tempered distributions
\newcommand{\TDP}{\mathcal{S}'/\mathcal{P}} % space of tempered distributions modulo polynomials

\newcommand{\SV}{\mathcal{S}_0} % Schwartz class with all moments VANISHING
\newcommand{\sL}{\mathrm{L}} % sub-Laplacian; in accordance with [FS82]
\renewcommand{\P}{\mathrm{P}} % operator given by convolution with the principal value distribution P

\newcommand{\crk}{\phi} % Calderón reproducing convolution kernel
\newcommand{\drk}{\psi} % dual reproducing convolution kernel
\newcommand{\crkk}{\eta} % ALTERNATIVE Calderón reproducing convolution kernel
\newcommand{\drkk}{\zeta} % ALTERNATIVE dual reproducing convolution kernel
\newcommand{\F}{\dot{\mathbf{F}}} % homogeneous Triebel-Lizorkin space
\newcommand{\B}{\dot{\mathbf{B}}} % homogeneous Besov space

\newcommand{\SIP}{\mathcal{S}'(N) / \mathcal{P}}
\newcommand{\PTpar}{a}
\newcommand{\loc}{\operatorname{loc}}
\newcommand{\RLO}{\mathrm{R}} % in accordance with the fonts chosen for sub-Laplacians and the operator P
\newcommand{\TLptwo}{\F^{0}_{p,2}}
\newcommand{\TLS}{\F^{\ord}_{p,q}}
\newcommand{\BS}{\B^{\ord}_{p,q}}
\newcommand{\op}{\operatorname{op}}

\def\Xint#1{\mathchoice
{\XXint\displaystyle\textstyle{#1}}%
{\XXint\textstyle\scriptstyle{#1}}%
{\XXint\scriptstyle\scriptscriptstyle{#1}}%
{\XXint\scriptscriptstyle\scriptscriptstyle{#1}}%
\!\int}
\def\XXint#1#2#3{{\setbox0=\hbox{$#1{#2#3}{\int}$ }
\vcenter{\hbox{$#2#3$ }}\kern-.6\wd0}}

\def\dashint{\Xint-}

\title[]{Besov and Triebel-Lizorkin spaces\\  on homogeneous groups}

\author[G. Hu]{Guorong Hu}

\address{School of Mathematics and Statistics, Jiangxi Normal University, Nanchang, Jiangxi 330022, China}

\email{hugr@mail.ustc.edu.cn}

\author[D. Rottensteiner]{David Rottensteiner}

\address{Department of Mathematics: Analysis, Logic and Discrete Mathematics, Ghent University, Ghent, Belgium}
\email{david.rottensteiner@ugent.be}

\author[M. Ruzhansky]{Michael Ruzhansky}

\address{Department of Mathematics: Analysis, Logic and Discrete Mathematics, Ghent University, Ghent, Belgium
and
School of Mathematical Sciences, Queen Marry University of London, United Kingdom}

\email{michael.ruzhansky@ugent.be}

\author[J.T. van Velthoven]{Jordy Timo van Velthoven}

\address{Faculty of Mathematics, University of Vienna, Oskar-Morgenstern-Platz 1, 1090 Vienna, Austria}

\email{jordy-timo.van-velthoven@univie.ac.at}



\begin{document}

\begin{abstract}
This paper develops a theory of Besov spaces $\dot{\mathbf{B}}^{\sigma}_{p,q} (N)$ and Triebel-Lizorkin spaces $\dot{\mathbf{F}}^{\sigma}_{p,q} (N)$ on an arbitrary homogeneous group $N$ for the full range of parameters $p, q \in (0, \infty]$ and $\sigma \in \mathbb{R}$. Among others, it is shown that these spaces are independent of the choice of the Littlewood-Paley decomposition and that they admit characterizations in terms of continuous maximal functions and molecular frame decompositions. The defined spaces include as special cases various
classical function spaces, such as Hardy spaces on homogeneous groups and homogeneous Sobolev spaces
and Lipschitz spaces associated to sub-Laplacians on stratified groups. 
\end{abstract}

\maketitle

\tableofcontents

\allowdisplaybreaks

\section{Introduction}
Besov and Triebel-Lizorkin spaces on Euclidean spaces provide a unified framework for various classical function spaces in harmonic analysis  encompassing, among others, Hardy spaces, Lebesgue spaces, Lipschitz spaces and Sobolev spaces. A common feature of these function spaces is that they admit a characterization in terms of a Littlewood-Paley decomposition, which is precisely the property that allows them to be treated within the scale of Besov and Triebel-Lizorkin spaces \cite{frazier1991littlewood, triebel1983theory}. 
Over the years, function spaces defined by Littlewood-Paley-type decompositions have been studied in great generalities beyond Euclidean spaces, such as on Lie groups and metric measure spaces; see, e.g., \cite{feneuil2018algebra, bruno2021potential, bruno2020besov, furioli2006littlewood, coulhon2001sobolev} and \cite{bruno2022homogeneous, KP, bui2012calderon, han2008theory, BBD}. A more specific setting that allows for a particularly rich theory of function spaces is that of homogeneous groups, which are nilpotent Lie groups equipped with a one-parameter family of dilations. The theory of Hardy spaces on such groups has first been initiated in the influential monograph \cite{FS} and has been further developed in, e.g., \cite{christ1984singular, Sato, glowacki1987inversion}. Furthermore, on the subclass of stratified Lie groups, various classes of smoothness spaces associated to sub-Laplacians, such as Lipschitz and Sobolev spaces, have been studied in, e.g.,  \cite{Folland1979Lipschitz, krantz1982lipschitz, Folland1975subelliptic, hu2019littlewood, FR2}. More generally, Besov and Triebel-Lizorkin spaces defined by Littlewood-Paley-type decompositions associated to sub-Laplacians on stratified Lie groups have been studied in, e.g., \cite{fuehr2012homogeneous, hu2013stratified, giulini1986approximation, saka1979besov, bruno2023pointwise}.

The objective of the present paper is to develop a general theory of homogeneous Besov and Triebel-Lizorkin spaces on arbitrary homogeneous groups. Motivated by the approach in \cite{FS}, we develop such a framework 
under quite general assumptions on the Littlewood-Paley-type decompositions. In particular, we do not assume such Littlewood-Paley decompositions to be associated with a specific self-adjoint operator. 
In doing so, we are able to provide a unified framework for the aforementioned Hardy spaces on homogeneous groups \cite{FS, christ1984singular, glowacki1987inversion},  the various smoothness spaces associated to sub-Laplacians on stratified groups \cite{Folland1979Lipschitz, FR2, krantz1982lipschitz, fuehr2012homogeneous, hu2013stratified}, and the Sobolev spaces associated to general Rockland operators on graded groups \cite{FR2}.  

Although, as mentioned above, Besov and Triebel-Lizorkin spaces have been defined on general doubling metric measure spaces (see, e.g., \cite{bruno2022homogeneous, bui2012calderon, BBD, KP}), the general theory developed in the present paper does not appear to be a special case of such existing theories. On the one hand, in contrast to, e.g., the papers  \cite{bruno2022homogeneous, georgiadis2017homogeneous, liu2016besov, BBD, KP}, the Besov and Triebel-Lizorkin spaces defined here are not assumed to be associated to a fixed self-adjoint operator.
On the other hand, even for Littlewood-Paley decompositions associated to self-adjoint operators, we do not require such operators to generate semigroups satisfying Gaussian estimates as is done in, e.g., \cite{bruno2022homogeneous, BBD, liu2016besov, KP}. As a matter of fact, the self-adjoint operators that we will use for constructing Littlewood-Paley-type decompositions on arbitrary homogeneous groups generate semigroups with only mild decay at infinity, cf. \cite{glowacki1986stable, dziubanski1992schwartz}. Nevertheless, in the case of a stratified group, the spaces defined in the present paper naturally coincide with those studied in \cite{fuehr2012homogeneous, hu2013stratified}, which in turn can be treated as special cases of the general theories \cite{bruno2022homogeneous, BBD, KP}. Indeed, in the case of stratified groups, the sub-Laplacian is a homogeneous operator with respect to the canonical automorphic dilations, and hence the spectral multipliers used in \cite{bruno2022homogeneous, BBD, KP} to define Littlewood-Paley decompositions form dilations of a fixed function. Similarly, in the case of a graded group with canonical dilations, a Rockland operator is a homogeneous operator satisfying Gaussian estimates (possibly of arbitrary high order) \cite{auscher1994on}, and hence the associated Besov and Triebel-Lizorkin spaces can be treated a special cases of, e.g., \cite{BBD, KP}. For arbitrary homogeneous groups, however, the existence of homogeneous operators satisfying Gaussian estimates appears unknown, which yields that the Besov and Triebel-Lizorkin spaces defined in this paper cannot simply be studied as special cases of the general theory \cite{BBD, KP}.

We next describe the setting of the paper and give an outline of its content. For a more detailed discussion on the basic notions of analysis on homogeneous groups, see Section \ref{sec:prelim}.

\subsection{Setup and content} Let $N$ be a homogeneous group equipped with a family of dilations $(\delta_t)_{t > 0}$ and let $\hdim$ be the homogeneous dimension of $N$. As mentioned above, our approach towards defining Besov and Triebel-Lizorkin spaces is based on the use of Littlewood-Paley-type decompositions. For this purpose, we say that a Schwartz function $\crk \in \SV(N)$ with all moments vanishing satisfies the \emph{discrete Calder\'on formula} if, for all $f \in \SV'(N)$,
\begin{align} \label{eq:calderon_intro}
f = \sum_{j \in \mathbb{Z}} f \ast \crk_{2^{-j}} \ast \crk_{2^{-j}}, 
\end{align} 
with convergence in $\SV'(N)$, where we use the notation $\crk_{t} := t^{-Q} \crk(\delta_{t^{-1}} (\cdot))$ for a general $t > 0$. Functions satisfying such a Calder\'on condition on arbitrary homogeneous groups will be constructed in Section \ref{sec:auxiliary} as spectral multipliers of some nondifferential convolution operator introduced in \cite{glowacki1986stable}. For stratified (resp. graded) groups, such functions are already known to exist as a consequence of Hulanicki's theorem \cite{Hulanicki} for sub-Laplacians (resp. Rockland operators), see, e.g., \cite{fuehr2012homogeneous, geller2006continuous}. 

Given a function $\crk \in \SV(N)$ satisfying the discrete Calder\'on condition \eqref{eq:calderon_intro}, a smoothness parameter $\ord \in \mathbb{R}$ and integrability exponents $p, q \in (0, \infty]$, we define the homogeneous Besov space $\BS(N)$ and Triebel-Lizorkin space $\TLS(N)$ as the space consisting of all $f \in \SV'(N)$ satisfying $\| f \|_{\BS} < \infty$ and $\| f \|_{\TLS} < \infty$, respectively, where
\[
\| f \|_{\BS} := \bigg( \sum_{j \in \Z} 2^{j \ord q} \| f \ast \crk_{2^{-j}} \|^q_{L^p} \bigg)^{1/q} \]
whenever $q < \infty$ (with the usual modification for $q = \infty$) and
\[
\| f \|_{\TLS} := \bigg\| \bigg( \sum_{j \in \mathbb{Z}} 2^{j \ord q} |f \ast \phi_{2^{-j}} |^q \bigg)^{1/q} \bigg\|_{L^p},
\]
 whenever $p, q < \infty$ (with the usual modification for $q = \infty$); see Section  \ref{sec:BTL} for the definition of $\TLS(N)$ in the case $p = \infty$. We will prove various basic properties of these Besov and Triebel-Lizorkin spaces in Section \ref{sec:BTL}. Among others, we will show that they are independent of the choice of the defining function satisfying the discrete Calder\'on condition \eqref{eq:calderon_intro}. Besides that such an independence result is a desired basic property, it also provides a simple general
 explanation for the independence of the Besov and Triebel-Lizorkin spaces \cite{fuehr2012homogeneous, hu2013stratified} and the Sobolev spaces \cite{Folland1975subelliptic, FR2} on stratified/graded groups of the choice of sub-Laplacian/Rockland operator defining such spaces; see Corollary \ref{cor:spectral_multiplier} and Section \ref{sec:identification}. We mention that our general study of Triebel-Lizorkin spaces $\TLS (N)$ for $p = \infty$ appears to be new even for the special case of stratified groups.
 
 In Section \ref{sec:continuous}, we provide various continuous and maximal characterizations of the Besov and Triebel-Lizorkin spaces in the spirit of the works \cite{Bui1996maximal, BT, Rychkov, stromberg1989weighted} in the classical setting of Euclidean spaces.
 Here, we assume that $\drk \in \SV(N)$ satisfies a \emph{continuous Calder\'on condition} of the form 
 \[
 f = \int_0^{\infty} f \ast \drk_t \ast \drk_t \; \frac{dt}{t}, \quad f \in \SV'(N);
 \]
 see Section \ref{sec:auxiliary}. For such a function $\psi$ and $a, t > 0$, the associated \emph{Peetre-type maximal function} of a distribution $f \in \SV'(N)$ is defined as
 \begin{align*}
	(\drk_t^* f)_\PTpar(x) =
	\sup_{y \in \N} \frac{|f \ast \drk_t(y)| }{(1+t^{-1}|y^{-1}x|)^\PTpar},\quad x \in \N,
\end{align*}
and satisfies $| f \ast \drk_t|(x) \leq (\drk_t^* f)_\PTpar(x)$ for all $x \in \N$ and $a,t>0$.
Our continuous and maximal characterizations (cf. Section \ref{sec:continuous}) then yield that, for all $f \in \SV'(N)$,
\begin{align} \label{eq:continuous_besov_intro}
 \bigg( \int_0^{\infty} t^{- \ord q} \| (\drk_t^* f)_a \|_{L^p}^q \; \frac{dt}{t} \bigg)^{1/q} \lesssim \| f \|_{\BS} \lesssim \bigg( \int_0^{\infty}  t^{- \ord q} \| f \ast \drk_t \|_{L^p}^q \; \frac{dt}{t} \bigg)^{1/q}
\end{align}
whenever $p, q \in (0, \infty]$, $\ord \in \mathbb{R}$ and $a > \max \{ \frac{Q}{p}, |\sigma|\}$, and 
\begin{align} \label{eq:continuous_tl_intro}
\bigg\| \bigg( \int_0^{\infty} t^{- \ord q} [(\drk_t^* f)_a]^q \; \frac{dt}{t} \bigg)^{1/q} \bigg\|_{L^p} 
\lesssim \| f \|_{\TLS} 
\lesssim 
\bigg\| \bigg( \int_0^{\infty} t^{- \ord q} |f \ast \drk_t|^q \; \frac{dt}{t} \bigg)^{1/q} \bigg\|_{L^p} 
\end{align}
whenever $p \in (0, \infty)$, $q \in (0, \infty]$, $\sigma \in \R$ and $a > \max \{ \frac{Q}{\min\{p, q\}}, |\sigma| \}$.  In addition, we also provide such maximal characterizations for $\TLS(N)$ with $p = \infty$ and show that $\dot{\mathbf{B}}_{\infty, \infty}^{\alpha} (N)= \dot{\mathbf{F}}_{\infty, \infty}^{\alpha} (N)$; see Section \ref{sec:continuous} for the precise details.
Our proofs of the continuous maximal characterizations \eqref{eq:continuous_besov_intro} and \eqref{eq:continuous_tl_intro} are inspired by the approaches in \cite{Bui1996maximal, BT, Rychkov, stromberg1989weighted} for Euclidean spaces and hinge on a sub-mean-value property of the convolution products $f \ast \drk_t$, which we establish as a result of independent interest in Section \ref{sec:auxiliary}. For doubling metric measure spaces, continuous characterizations in terms of Peetre-type maximal functions can be found in, e.g., \cite{BBD,georgiadis2017homogeneous, liu2016besov}.  
The novelty of our characterizations  \eqref{eq:continuous_besov_intro} and \eqref{eq:continuous_tl_intro} beyond Euclidean spaces is that they are valid for general functions satisfying a continuous Calder\'on condition and do not need to be spectral multipliers of a fixed self-adjoint operator, which requires various  additional arguments in the proofs.

The continuous maximal characterizations \eqref{eq:continuous_besov_intro} and \eqref{eq:continuous_tl_intro} will play an essential role in our approach towards obtaining molecular and wavelet-type frame decompositions of the Besov and Triebel-Lizorkin spaces. For obtaining such decompositions, we interpret the characterizations \eqref{eq:continuous_besov_intro} and \eqref{eq:continuous_tl_intro} of the Besov and Triebel-Lizorkin spaces as imposing mixed-norms on a continuous wavelet transform or matrix coefficients on the semi-direct product group $\NA = \N \times (0, \infty)$ defined by the action of the dilations ($\delta_t)_{t > 0}$ (see Section \ref{sec:wavelet}). This interpretation will be used in Section \ref{sec:coorbit} to treat the Besov and Triebel-Lizorkin spaces as (quasi-)Banach spaces \cite{feichtinger1989banach, velthoven2022quasi} associated to the unitary representation $\pi$ of $\NA = \N \times (0, \infty)$ acting on $f \in L^2 (N)$ by
\[
\pi(x,t) f = t^{-\hdim/2} f(\delta_{t^{-1}} (x^{-1} \cdot)), \quad (x, t) \in \NA;
\]
see \cite{rauhut2011generalized, ullrich2012continuous, koppensteiner2023anisotropic1, koppensteiner2023anisotropic2}
for similar interpretations in the classical setting of Euclidean spaces. For such (quasi-)Banach spaces associated to the unitary representation $\pi$, the general theory \cite{romero2021dual, feichtinger1989banach, grochenig1991describing, velthoven2022quasi} provides series expansions of $f \in L^2 (N)$ of the form
\begin{align} \label{eq:atomic_intro}
f = \sum_{\lambda \in \Lambda} \langle f, \varphi_{\lambda} \rangle \pi (\lambda) \psi = \sum_{\lambda \in \Lambda} \langle f, \pi(\lambda) \drk \rangle \varphi_{\lambda}
\end{align}
for a discrete set $\Lambda \subseteq G$ and a system $(\varphi_{\lambda})_{\lambda \in \Lambda}$ of vectors $\varphi_{\lambda} \in L^2 (\N)$. In addition, the recent improvements \cite{romero2021dual, velthoven2022quasi} on the classical decompositions of such spaces \cite{feichtinger1989banach, grochenig1991describing} guarantee that the system $(\varphi_{\lambda})_{\lambda \in \Lambda}$ appearing in \eqref{eq:atomic_intro} forms a molecular system, in the sense that it satisfies similar localization estimates as the discrete wavelet system $\{ \pi(\lambda) \drk : \lambda \in \Lambda \}$; see Section \ref{sec:coorbit} for precise definitions. In particular, this yields that the molecular decompositions \eqref{eq:atomic_intro} can be extended to all $f \in \BS(N)$ and $f \in \TLS(N)$, cf. Section \ref{sec:coorbit}. These molecular decompositions of Besov and Triebel-Lizorkin spaces are similar to the state-of-the-art in the classical setting of Euclidean spaces obtained in \cite{gilbert2002smooth}. Beyond Euclidean spaces, general frame decompositions of Besov spaces and Hardy spaces on stratified groups have been obtained earlier in \cite{fuehr2012homogeneous, geller2006continuous}. However, the decompositions on stratified groups \cite{fuehr2012homogeneous, geller2006continuous} do not provide a similar localization of the dual system, and as such our results are new even in these specific settings. We mention that explicit frame constructions on metric measure spaces have been studied in, e.g., \cite{coulhon2012heat, KP, georgiadis2017homogeneous}.

Lastly, in Section \ref{sec:identification}, we provide identifications of the Besov and Triebel-Lizorkin spaces with various function spaces studied earlier in the literature. Among others, we identify the Hardy spaces \cite{FS}, homogeneous Sobolev spaces \cite{FR2, Folland1975subelliptic} and Lipschitz spaces \cite{krantz1982lipschitz, Folland1979Lipschitz, hu2019littlewood} as special cases of Besov or Triebel-Lizorkin spaces. 

\subsection{General notation}
We denote by $\mathbb{N}_0$ the set of all nonnegative integers, and by $\mathbb{N}$ the set of all strictly positive integers. The open and closed half-lines in the real line $\mathbb{R}$ are denoted by $\mathbb{R}^+ = (0, \infty)$ and $\mathbb{R}^+_0 = [0, \infty)$, respectively. 

For $a,b \in \mathbb{R}$, we write $a \land b := \min\{a,b\}$ and $a \lor b := \max\{a,b\}$. The conjugate exponent $r'$ of $r \in (0,\infty]$ is defined to satisfy $1/r + 1/r' = 1$ if $1 < r \leq \infty$ and $r' = \infty$ if $0< r \leq 1$, with the convention that $\frac{r}{\infty} = 0$ if $r \in (0, \infty)$ and $\frac{r}{\infty} = 1$ if $r = \infty$. 

Given two functions $f, g : X \to [0,\infty)$ on a set $X$, we write $f \lesssim g$ whenever there exists $C \geq 1$ such that $f(x) \leq C g(x)$ for all $x \in X$, and $f \asymp g$ if $f \lesssim g$ and $g \lesssim f$. In order to indicate a specific dependence of the implicit constants in $\asymp$ or $\lesssim$, we will use a subscript, e.g., $\lesssim_\delta$ if it depends on some parameter $\delta$ whose meaning is clear from the context.

Given a Lie group $G$ with left Haar measure $\mu_\NA$, the Lebesgue spaces are denoted by $L^r (\NA) = L^r(\NA, \mu_\NA)$ for $r \in (0,\infty)$. The space of locally integrable functions on $\NA$ is denoted by $L^1_{\loc} (G)$. For a function $f \in L^1_{\loc} (G)$ and a set $\Omega$ of positive, but finite measure, we write $\dashint_{\Omega} |f(x)| \; d\mu_G (x) = \frac{1}{\mu_G(\Omega)} \int_{\Omega} |f(x)| \; d\mu_G(x)$. The indicator function of a measurable set $\Omega \subseteq G$ is denoted by $\chi_{\Omega}$.
For a function $f : \NA \to \mathbb{C}$, we write $f^{\vee}(x) = f(x^{-1})$ for $x \in \NA$. The space of smooth functions on $G$ is denoted by $C^{\infty}(\NA)$, and $C_c^{\infty}(\NA)$ denotes the subspace of $C^{\infty}(\NA)$ consisting of functions with compact support.

\section{Analysis on homogeneous groups} \label{sec:prelim}
This section contains preliminaries on homogeneous groups and functions and operators on such groups that will be used freely throughout the article. References for the theory provided here are the books \cite{FS, FR}. We refer the reader to these books for more details and proofs.

\subsection{Dilations}
Let $\n$ be a real Lie algebra of dimension $d$ with Lie bracket $[\cdot, \cdot]$. A one-parameter family $\{D_t \}_{t > 0}$ of automorphisms $D_t : \n \to \n$ is said to consist of \emph{dilations} if each element possesses the form $D_t = \exp(A \ln t)$ for a diagonalizable linear map with positive eigenvalues $v_1, ..., v_d$, which are often called the \emph{dilations' weights}. A Lie algebra admitting dilations is automatically nilpotent, but not every nilpotent Lie algebra admits dilations. Note that if $D_t = \exp(A \ln t)$ are dilations, then so are $D'_t = \exp( \ln t A')$, where $A' = cA$ for some $c > 0$. 

The existence of a family of dilations on $\n$ is intimately related to the existence of a grading on $\mathfrak{n}$. An \emph{$\mathbb{N}$-grading} of a Lie algebra $\mathfrak{n}$ is a family ($\mathfrak{n}_i)_{i \in \mathbb{N}}$ of subspaces $\mathfrak{n}_i \subseteq \mathfrak{n}$ (where all but finitely many $\mathfrak{n}_i$ are zero) such that $\mathfrak{n} = \bigoplus_{i \in \mathbb{N}} \mathfrak{n}_i$ and $[\mathfrak{n}_i, \mathfrak{n}_j ] \subseteq \mathfrak{n}_{i + j}$ for all $i, j \in \mathbb{N}$. A Lie algebra admitting an $\mathbb{N}$-grading is often referred to as a \emph{graded Lie algebra}. If $\mathfrak{n}$ admits an $\mathbb{N}$-grading $(\mathfrak{n}_i)_{i \in \mathbb{N}}$, then a family of dilations $(D_t)_{t > 0}$ can be defined by $D_t = \exp(A \ln t)$, where $A : \mathfrak{n} \to \mathfrak{n}$ is the linear operator defined by $A X = i X$ for $X \in \mathfrak{n}_i$. We will refer to such dilations as the \emph{canonical dilations} associated to an $\mathbb{N}$-grading. Conversely, if $\mathfrak{n}$ admits dilations, then it also admits an $\mathbb{N}$-grading \cite{miller1980parametrices}, but the given dilations need not to be the canonical dilations for such an $\mathbb{N}$-grading.
An $\mathbb{N}$-grading $(\mathfrak{g}_i)_{i \in \mathbb{N}}$ such that $\mathfrak{n}_1$ generates $\n$ as an algebra is said to be a \emph{stratification} of $\n$. A Lie algebra admitting a stratification will be called \emph{stratified}.

A \emph{homogeneous group} is a connected, simply connected nilpotent Lie group whose Lie algebra can be endowed with a family of dilations. 
For a homogeneous group $N$ with Lie algebra $\n$, the exponential map $\exp_N : \n \to N$ is a global diffeomorphism. 
If $(D_t)_{t>0}$ is a family of dilations on $\n$, then the associated automorphisms on $\N = \exp_\N (\n)$ are defined by $\delta_t := \exp_\N \circ D_t \circ \exp_N^{-1}$, and will also be referred to as \emph{dilations}. Given $t > 0$ and $x \in N$, we will often simply write $t x = \delta_t (x)$.

Throughout, unless stated otherwise, $N$ will always denote a homogeneous group with a fixed family of dilations $D_t = \exp(A \ln t)$ on $\n$. It will be assumed that the eigenvalues $v_1, \cdots , v_d$ of $A$ are listed in increasing order and that $v_1 \geq 1$. The number $\hdim := v_1 + ... + v_d \geq \dim(\N)$ is the \emph{homogeneous dimension} of $\N$. 
We fix once and for all a basis $\{X_1, \cdots, X_\dimN\}$ for $\n$ such that $A X_j = v_j X_j$ for $j =1,\cdots, \dimN$. Via the exponential map $\exp_N : \n \to N$, we identify the points $x = (x_1,\cdots, x_\dimN) \in \mathbb{R}^{\dimN} \cong \n$  with points $x =\exp_\N (x_1 X_1 + \cdots + x_\dimN X_d)$ in $\N$. Under this identification, the  abstract Lie group $N$ may be considered as a Lie group with underlying manifold $\mathbb{R}^n$. For each $j \in \{1,\cdots, \dimN\}$, the coordinate function $x = (x_1, \cdots, x_\dimN) \mapsto x_j$ is denoted by $x_j$.

\subsection{Homogeneous functions} \label{subsection:homogeneity}
A function $f : \N \to \mathbb{C}$ is called \emph{homogeneous} of degree $\hdeg \in \mathbb{R}$ if it satisfies $f \circ \delta_t = t^{\hdeg} f$ for all $t > 0$. 
For example, for a multi-index $\alpha =(\alpha_1, \cdots, \alpha_{\dimN}) \in \mathbb{N}_0^\dimN$, the function $x^\alpha = x_1^{\alpha_1}\cdots x_\dimN^{\alpha_{\dimN}}$ is a homogeneous function of degree 
\begin{align} \label{homogeneous degree of alpha}
[\alpha]:=v_1 \alpha_1 +\cdots v_{\dimN} \alpha_{\dimN}.  
\end{align}
Formula \eqref{homogeneous degree of alpha} defines what is often referred to as the \emph{homogeneous degree} of the multi-index $\alpha$. It is usually
different from the length of $\alpha$ given by
$
|\alpha|:= \alpha_1 +\cdots + \alpha_{\dimN}$.  

A \emph{homogeneous quasi-norm} on $N$ is a symmetric continuous function $|\cdot | : \N \to [0,\infty)$ that is homogeneous of degree $1$ and vanishes precisely at the group identity $e_N$. Homogeneous quasi-norms exist on any homogeneous group and any two homogeneous quasi-norms $|\cdot|_1$ and $|\cdot |_2$ are equivalent, in the sense that $| \cdot |_1 \asymp | \cdot |_2$. Throughout, we fix a homogeneous quasi-norm $|\cdot|$.  There exists 
$\cqn \geq 1$ such that
\begin{align} \label{quas}
 |xy| \leq \cqn (|x| + |y|) \quad \mbox{for all } x, y \in \N.
 \end{align}
 In particular, for any $a > 0$, the Peetre-type inequality
 \begin{align} \label{eq:peetretype}
 (1+|x|)^a (1+|y|)^{-a} \leq \cqn^a (1 + |xy^{-1}|)^a
 \end{align}
 holds for all $x, y \in \N$. 
A homogeneous quasi-norm may be chosen such that \eqref{quas} holds with $\gamma = 1$. 
 
 The open ball $B_r (x)$ of center $x \in \N$ and radius $r > 0$ is defined by
\begin{align*}
B_r (x) := \{ y \in \N : |y^{-1} x| < r \}.
\end{align*}
We also sometimes write $B(x, r) = B_r (x)$.  
We have that $B_r(x) = \delta_r \bigl ( B_1(r^{-1}x) \bigr )$ and $B_r (x) = x B_r (e_N)$, and its Haar measure satisfies $\mu_\N(B_t(x)) = t^\hdim \mu_\N(B_1(e_N))$, where $\hdim$ denotes the homogeneous dimension of $\N$. 
Given a homogeneous ball $B \subseteq \N$, we will often use $x_B$ and $r_B$ to denote the center and radius of $B$, respectively.
We will frequently use that the triple $(\N, |\cdot|, \mu_\N)$ forms a space of homogeneous type in the sense of Coifman and Weiss.

There is an analogue of polar coordinates on homogeneous groups with the homogeneous dimension $\hdim$ replacing the topological dimension $d$: for all 
$f \in L^1(\N)$,
\begin{align*}
\int_{\N} f(x)\; d\mu_\N (x) = \int_0^\infty \int_{\mathfrak{G}}
f(\delta_r (y))r^{\hdim -1 }d\nu(y)dr,
\end{align*}
where $d\nu$ is a (unique) Radon measure on the 
unit sphere $\mathfrak{G}:=\{x \in \N: |x| =1\}$. 
This implies that for $0 < r < R < \infty$ and $s \in \mathbb{R}$, 
\begin{align} \label{eq:polar}
\int_{r < |x| <R} |x|^{s-\hdim} \; d\mu_\N (x) 
=\begin{cases}
Cs^{-1}(R^{-s}-r^{-s}) &\text{if } s \neq 0, \\
C \log (R /r) &\text{if } s =0.
\end{cases}
\end{align}
Consequently, if $s > 0$ then $|\cdot|^{s-\hdim}$ is integrable near the group identity $e_\N$, and if $s < 0$ then $|\cdot|^{s-\hdim}$ is integrable near $\infty$.

\subsection{Homogeneous operators}
A linear operator $T : C_c^{\infty}(\N) \to (C_c^{\infty}(\N))'$ is \emph{homogeneous} of degree $\hdeg \in \mathbb{R}$ if $T(f \circ \delta_t) = t^{\hdeg} (Tf) \circ \delta_t$ for all $f \in C_c^{\infty}(\N)$ and $t>0$. If we regard $X_j \in \n$ as a left-invariant vector field on $\N$ in the usual sense, i.e.,
\begin{align} \label{eq:definition of the vector field Xj}
X_j f(x): = \frac{d}{dt} f \big( x \exp(tX_j) \big) |_{t=0},
\end{align} 
then it is homogeneous of degree $v_j$. Similarly, for a multi-index $\alpha = (\alpha_1, ..., \alpha_d) \in \mathbb{N}_0^d$, the differential operator 
\[
X^{\alpha} := X_1^{\alpha_1} \dots X_d^{\alpha_d}\]
is homogeneous of degree $[\alpha]$. 

A left-invariant differential operator that is homogeneous of positive degree and hypoelliptic is called a \emph{Rockland operator}. If $\N$ admits a Rockland operator, then the dilations' weights have a common rational multiple, so that by rescaling the dilations if necessary, it may be assumed that the dilations are canonical dilations associated to an $\mathbb{N}$-grading. Conversely, if $(D_t)_{t >0}$ is a family of canonical dilations associated to an $\mathbb{N}$-grading of $\n$ and $\{X_1, ..., X_d\}$ is a basis for $\n$ satisfying $D_t X_j = t^{v_j} X_j$ for $j = 1, ..., d$ and $t>0$, then the operator
\[
\sum_{j = 1}^d (-1)^{\frac{v_0}{v_j}} c_j X_j^{2 \frac{v_0}{v_j}},
\]
where $c_j >0$ and $v_0$ is any common multiple of $v_1, ..., v_d$, is a Rockland operator of homogeneous degree $2 v_0$. Moreover, if $\n$ admits a stratification $(\n_i)_{i \in \mathbb{N}}$ and $Z_1, ..., Z_{d_1}$ is a basis for $\n_1$, then
the associated \emph{sub-Laplacian}
\[
\sL := \sum_{j = 1}^{d_1} Z^2_j
\]
is a Rockland operator on $\N$ of homogeneous degree $2$.

Lastly, we denote by $\{\widetilde{X}_1,\cdots,\widetilde{X}_d\}$ the basis for right-invariant vector fields corresponding to $\{X_1,\cdots,X_d\}$, i.e., $\widetilde{X}_j$ is defined by 
\begin{align*}
\widetilde{X}_j f(x) :=\frac{d}{dt} f \bigl ( \exp(tX_j) x\bigr ) \big|_{t =0}, \quad j =1,\cdots, d.
\end{align*}
For $\alpha =(\alpha_1,\cdots, \alpha_d) \in \mathbb{N}_0^d$, we set \begin{align}
\widetilde{X}^\alpha := \widetilde{X}_1^{\alpha_1} \cdots \widetilde{X}_d^{\alpha_d}.
\end{align}
Then 
$\widetilde{X}^\alpha$ is a right-invariant differential operator homogeneous of degree $[\alpha]$.

\subsection{Polynomials}

A function $P: \N \rightarrow \mathbb{C}$ is called a \emph{polynomial} if it possesses the form
\begin{align*}
P(x) =\sum_{\alpha \in \mathbb{N}_0^{\dimN}} c_\alpha x^\alpha,
\end{align*}
where all but finitely many of the complex coefficients $c_\alpha$ vanish. The
homogeneous degree
of the polynomial $P$ is defined as 
\begin{align*}
d(P):=\max \{[\alpha]: c_\alpha \neq 0 \} .
\end{align*}
The set of all polynomials is denoted by $\mathcal{P}$ and for any $M \in \mathbb{N}_0$, 
we set
\begin{align*}
\mathcal{P}_{M} := \{P \in \mathcal{P}: d(P)\leq M\},
\end{align*}
the set of polynomials of homogeneous degree at most $M$. 

Following \cite{FS}, we define the notion of Taylor polynomials on homogeneous groups.
Let $x \in \N$, $M \in \mathbb{N}_0$, and 
$f$ be a function whose (distributional) derivatives $X^\alpha f$ are continuous functions in a neighborhood of $x$ for $[\alpha] \leq M$. Then there exists a unique polynomial $P_{x,M}^{(f)} \in \mathcal{P}_{M}$ such that, for all $\alpha \in \mathbb{N}_0^d$ with $[\alpha] \leq M$, 
\begin{align} \label{eq:defintion of taylor polynomial}
 {X}^\alpha \big(P_{x,M}^{(f)}\big)(e_\N) = {X}^\alpha f(x). 
\end{align}
We will refer to this polynomial $P_{x,M}^{(f)}(\cdot)$
the  {\it left Taylor polynomial of $f$ at $x$ of homogeneous degree $\leq M$}. 

\begin{remark}
The classical Taylor polynomial of an adequate function $f: \mathbb{R}^n \rightarrow \mathbb{R}$ at $x_0 \in \mathbb{R}^n$ of order $M \in \mathbb{N}_0$ is defined to be the unique polynomial $P$ such that for all $|\alpha| \leq M$,
\begin{align} \label{eq:classical_taylor}
\partial^\alpha f (x_0) = \partial^\alpha P (x_0).
\end{align}
Note that, in contrast, the left-hand side of \eqref{eq:defintion of taylor polynomial} is  ${X}^\alpha \big(P_{x,M}^{(f)}\big)(e_\N)$ instead of ${X}^\alpha \big(P_{x,M}^{(f)}\big)(x)$ in \eqref{eq:classical_taylor}. The Taylor polynomial defined by \eqref{eq:defintion of taylor polynomial} may be thought of as the composition of the classical Taylor polynomial with a left-translation on $\N$.
\end{remark}

If $\N$ is identified with a Lie group on $\mathbb{R}^n$, 
then we have $X_j f(e_\N) =\partial_{x_j}f(e_\N)$. This means that $\{X_1,\cdots, X_d\}$ forms a Jacobian basis of the Lie algebra of the Lie group with underlying manifold $\mathbb{R}^n$, see, e.g., the proof of \cite[Theorem 2.2.18]{BLU}. Using this observation, we deduce from \cite[Theorem 2]{Bon} that the Taylor polynomial $P_{x,M}^{(f)}$ has the following explicit form.

\begin{lemma} \label{lem:explicit expression of taylor polynomial}
The explicit expression of the left Taylor polynomial of $f$ at $x$ of homogeneous degree $\leq M$ is
\begin{align} \label{eq:explicit taylor polynomial}
P_{x,M}^{(f)}(y) = f(x) + \sum_{k =1}^{\lceil M \rfloor} \sum_{\substack{i_1,\cdots, i_k \in \{1,\cdots, d\}\\ i_1 v_{i_1} +\cdots + i_k v_{i_k} \leq M}} \frac{X_{i_1}\cdots X_{i_k}f(x)}{k!} y_{i_1} \cdots  y_{i_k},
\end{align}
where 
\begin{align}
\label{eq:defintion of M}
\lceil M \rfloor:= \max\{|\alpha|: \alpha \in \mathbb{N}_0^\dimN \text{ with } [\alpha] \leq M\}.
\end{align}
\end{lemma}
 
 A combination of the explicit expression \eqref{eq:explicit taylor polynomial} and the 
Poincar\'{e}-Birkhoff-Witt theorem  implies that
\begin{align} \label{eq:rough taylor polynomial}
P_{x,M}^{(f)}(y)=\sum_{\alpha \in \mathbb{N}_0^d : \ [\alpha] \leq M}\left( \sum_{\beta \in \mathbb{N}_0^d: \ [\beta] = [\alpha]} c_\beta(X^\beta f)(x)  \right)  y^\alpha,
\end{align}
where each $c_\beta$ is a constant depending only on $\beta$.
This rougher expression will be sufficient for our purpose in the sequel.

We will often use the following Taylor inequality.

\begin{lemma}[Left Taylor inequality] \label{lem:taylor inequality}
For any  $M \in \mathbb{N}_0$, there is a constant $C_M >0$ such that for all functions $f \in C^{\lceil M \rfloor +1}(\N)$ and all $x,y\in \N$,
\begin{align} \label{eq:taylor inequality}
|f(xy) - P_{x,M}^{(f)}(y)| \leq C_M \sum_{\substack{|\alpha| \le \lceil M \rfloor +1 \\ [\alpha] > M}} |y|^{[\alpha]}
\sup_{|z| \leq \eta^{\lceil M \rfloor +1}|y|}  |(X^\alpha f)(xz)|,
\end{align}
where $ P_{x,M}^{(f)}(\cdot)$ is the left Taylor polynomial of $f$ at $x$ of homogeneous degree $\leq M$,  $\eta >1$ is a constant, and 
$\lceil M \rfloor$ is defined by \eqref{eq:defintion of M}.
\end{lemma}

\subsection{Schwartz functions and distributions}
A function $f : \N \to \mathbb{C}$ is a \emph{Schwartz function} on $N$ if $f \circ \exp_\N$ is a Schwartz function on $\n$. The space of Schwartz functions is denoted by $\mathcal{S}(\N)$. 
For our purposes, it will be most convenient to use the following family of semi-norms to define the usual topology on $\mathcal{S}(\N)$,
\begin{align} \label{eq:Schwartz_seminorm}
\|f\|_{(k)}:= \sup_{[\alpha] \leq k, \ x \in \N} (1 + |x|)^{k} |X^\alpha f(x)|, \quad k \in \mathbb{N}.
\end{align} 
Consequently, if $f \in \mathcal{S}(\N)$ and $\ell, L \in \mathbb{N}$, then
\begin{align} \label{eq:derivativies of ft}
|X^\alpha f (x)| \leq  (1+|x|)^{-(L+[\alpha])} \|f\|_{(L+\ell)}
\end{align}
holds for all $x \in \N$ and all $\alpha \in \mathbb{N}_0^d$ with $[\alpha] \leq \ell$.

For fixed $k \in \NN$, the semi-norm \eqref{eq:Schwartz_seminorm} can be used to define the Banach space
\begin{align} \label{eq:Sk}
    \mathcal{S}_{k}(\N) := \bigl \{ f \in C^k (\N) : \|f\|_{(k)} < \infty \bigr \}.
\end{align}
Then $\SC(\N) = \bigcap_{k \in \NN} \mathcal{S}^k(\N)$, and $\SC(\N)$ is continuously embedded into $\mathcal{S}_{k}(\N)$ for any $k \in \NN$. For $k > \hdim/2$, we also have the continuous embedding $\mathcal{S}^k (N) \hookrightarrow L^2(\N)$.

The topological dual of $\mathcal{S}(\N)$, denoted by $\mathcal{S}'(\N)$, consists of tempered distributions on $\N$. A distribution $f \in \SC'(\N)$ is \emph{homogeneous} of degree $\nu \in \mathbb{R}$ if $f \circ \delta_t = t^{\nu} f$ for all $t > 0$.

We denote by $\SV(\N)$ the space of all functions $f \in \SC (\N)$ with all moments vanishing, i.e., 
\begin{align*}
\int_{\N} f(x)P(x) \; d\mu_N(x) =0 \quad \text{ for all } P \in \mathcal{P}. 
\end{align*}
Endowed with the topology inherited from $\SC(\N)$,  the space
$\SV(\N)$ is closed in $\SC(\N)$. Its topological dual,
$\SV'(\N)$, can be canonically identified with $\TDP$, see, e.g., \cite[Lemma 3.3]{fuehr2012homogeneous}.
In the following, we will not always explicitly distinguish between $f \in \mathcal{S}'(N)$ and its equivalence class $[f] := f + \mathcal{P}$, and sometimes just write $f \in \mathcal{S}' /\mathcal{P}$.

The duality pairing between $f \in \mathcal{S}'(N)$ and $g \in \mathcal{S}(N)$ will often be denoted by $( f, g ) = f(g)$. The same notation will be used for the duality pairing between $\SV'(\N)$ and $\SV(\N)$.

\subsection{Convolutions} \label{sec:convolution}
The convolution of two measurable functions $f, g : N \to \mathbb{C}$ is defined by
\[
f \ast g (x) = \int_N f(y) g(y^{-1} x) \; d\mu_\N (y),
\]
provided the integral exists. If $f, g \in \mathcal{S}(N)$, then also $f \ast g \in \mathcal{S}(N)$. Moreover, if $f \in \mathcal{S}(\N)$ and $g \in \SV(\N)$, then $f \ast g \in \SV(\N)$.
For any multi-index $\alpha \in \mathbb{N}_0^n$ and functions $f, g \in \mathcal{S}(N)$, 
we have
\begin{align*}
X^{\alpha} (f \ast g) = f \ast (X^\alpha g).
\end{align*}

For a function $f \in L^1(N)$ and $t>0$, we will denote by $f_t$ the $L^1$-normalized dilation of $f$, i.e., 
\begin{align*}
f_t (x) := t^{-\hdim} f (t^{-1}x), \quad x \in N.
\end{align*}
Then $(f \ast g)_t = f_t \ast g_t$ for all $t > 0$.

The convolution $f * g$ of $f \in \TD(\N)$ and $g \in \SC(\N)$ is defined by $f \ast g (x) = ( f, g^{\vee}(x^{-1} \cdot) )$ for $x \in N$, and $f \ast g \in C^{\infty} (N)$. Moreover, for every $f \in \TD(\N)$, there exist $k =k_f\in \mathbb{N}$ and a constant $C =C_f >0$ such that
\begin{align} \label{bounds_of_convolution}
    |f \ast g(x)| \leq C (1 + |x|)^k \|g\|_{(k)},
\end{align}
for all $g \in \mathcal{S}(N)$ and $x \in \N$. In particular, $f \ast g \in \mathcal{S}'(\N)$, and $( f \ast g, h ) = ( f , g \ast h^{\vee} )$ for any $h \in \mathcal{S}(\N)$. For $g \in \SC(\N)$, the map $f \mapsto f \ast g$ is continuous on $\SC'(\N)$.

Lastly, for $g \in \mathcal{S}(\N)$, the map $f \mapsto f \ast g$ is  well-defined and continuous from $\SV'(\N)$ into $\SV'(N)$. If, in addition, $g \in \SV(\N)$, then $f \mapsto f \ast g$ is continuous from $\SV'(\N)$ into $\mathcal{S}'(N)$, cf. \cite[Lemma 3.4]{fuehr2012homogeneous}.
\medskip

\section{Auxiliary results and estimates} \label{sec:auxiliary}
We provide in this section various results on Calder\'on-type reproducing formulae, convolution kernels and vector-valued inequalities. These notions and results will play an essential role in defining Besov and Triebel-Lizorkin spaces and establishing their basic properties in the following sections.

\subsection{Almost orthogonality}

The following lemma provides an almost orthogonality estimate for convolution products of Schwartz functions and will be a central ingredient in various proofs.

\begin{lemma} \label{aoe}
If $f, g \in \SV(\N)$, then for any $M, L >0$, there exist  $C=C_{M,L}>0$ and $k = k_{M,L} \in \mathbb{N}$ such that  
\begin{align} \label{eq:almost orthogonality estimate}
 |f_t \ast g_{s}(x)| \leq C_{M,L} \left(\frac{t}{s} \wedge \frac{s}{t}\right)^{M} \frac{(s^{-1}\wedge t^{-1})^\hdim}{\big(1 + (s^{-1}\wedge t^{-1})|x|\big)^{L}} \|f\|_{(k)}\|g\|_{(k)}
\end{align}
for all $t,s >0$ and $x \in \N$.
\end{lemma}
\begin{proof}

Since both of the functions
\begin{align*}
 M \mapsto \left(\frac{t}{s} \wedge \frac{s}{t}\right)^{M} \quad \text{and}\quad
 L \mapsto  \big(1 + (s^{-1}\wedge t^{-1})|x|\big)^{-L}
\end{align*}
are monotonically decreasing on $(0, \infty)$, it suffices to prove the assertion of the lemma for $M, L \in \mathbb{N}$. Moreover, we may assume, without loss of generality, that  $L > M + \hdim$.

First assume that 
$t \geq s$. Let $P^{(f_t)}_{x,M}$ be the left Taylor polynomial of
$f_{t}$ at $x$ of homogeneous
degree $\leq M$. Denote by $\eta > 0$ the constant in Taylor's inequality \eqref{eq:taylor inequality}, $\cqn$ the constant in the quasi-triangle
inequality \eqref{quas}, and set $M':= \lceil M \rfloor +1$, where $\lceil M \rfloor$ is defined by \eqref{eq:defintion of M}.
Using that $g \in \SV(\N)$, it follows that
\begin{align*}
\left|(f_{t}  \ast g_{s})(x) \right|
&= \left|\int_{\N}  \left[f_{t}(xy^{-1}) -
P^{(f_t)}_{x,M} (y^{-1})\right] g_{s}(y) \; d\mu_\N(y)\right|\\
& \leq\int_{|y| \leq (t + |x|)/(2\cqn \eta^{M'})}\left|f_{t}(xy^{-1}) -
P^{(f_t)}_{x,M} (y^{-1})\right||g_{s}(y)| \; d\mu_\N(y) \\
&\quad +\int_{|y| > (t + |x|)/(2\cqn \eta^{M'})}|f_{t}(xy^{-1})| |g_{s}(y)| \; d\mu_\N(y) \\
&\quad +  \int_{|y| > (t + |x|)/(2\cqn \eta^{M'})}
 \left|{P}^{(f_t)}_{x,M}(y^{-1})\right||g_{s}(y)| \; d\mu_\N(y) \\
 & =: I_{1} + I_{2} + I_{3}.
\end{align*}
We will prove the claim by estimating the integrals $I_1$, $I_2$ and $I_3$ separately.
\\~\\
\emph{Estimate of $I_1$}. For estimating $I_1$, first note that by Taylor's inequality \eqref{eq:taylor inequality} and \eqref{eq:derivativies of ft},
\begin{align*}
     \left|f_{t}(xy^{-1}) - P^{(f_t)}_{x,M} (y^{-1})\right| 
&\lesssim_M \sum_{\substack{|\alpha| \le M' \\ [\alpha] > M}} |y|^{[\alpha]}
\sup_{|z| \leq \eta^{M'}|y|}
|(X^{\alpha} f_t) (xz)| \\
&\leq  \|f\|_{(L+M^\ast)}
\sum_{\substack{|\alpha| \le M' \\ [\alpha] > M}}t^{-[\alpha]} |y|^{[\alpha]}
\sup_{|z| \leq \eta^{M'}|y|}
\frac{t^{-Q}}{(1+t^{-1} |xz|)^{L+[\alpha]}},
\end{align*} 
where $M^\ast := \max\{[\alpha] : |\alpha| \leq M'\}$. Therefore, 
\begin{align} \label{eq:estimate of I_1}
I_{1} 
&\lesssim_M \|f\|_{(L+M^\ast)} \|g\|_{(M + \lfloor \hdim \rfloor +1)}   \int_{|y| \leq (t + |x|)/(2\cqn \eta^{M'})}\frac{s^{- \hdim}}
{(1 + s^{-1} |y|)^{M +\hdim}} \nonumber\\
&\hspace{140pt} \times\sum_{\substack{|\alpha| \le M' \\ [\alpha] > M} }t^{-[\alpha]} |y|^{[\alpha]}
\sup_{|z| \le \eta^{M'}|y|} \frac{ t^{-\hdim} }{(1
+ t^{-1} |xz|)^{L+[\alpha]}} \; d\mu_\N(y) \nonumber\\
&= \|f\|_{(L+M^\ast)} \|g\|_{(M + \lfloor \hdim \rfloor +1)} \int_{|y| \leq (t + |x|)/(2\cqn \eta^{M'})}\frac{s^{M}}
{(s + |y|)^{M+\hdim} } \\
&\hspace{140pt}\times \sum_{\substack{|\alpha| \le M'\\ [\alpha] > M }}t^{-[\alpha]} |y|^{[\alpha]}
\sup_{|z| \le \eta^{M'}|y|} \frac{ t^{L+[\alpha]-\hdim} }{(t
+ |xz|)^{L+[\alpha]}} \; d\mu_\N(y)  \nonumber\\
& \lesssim_{M, L} \|f\|_{(L+M^\ast)} \|g\|_{(M + \lfloor \hdim \rfloor +1)}
\sum_{\substack{|\alpha| \le M' \\ [\alpha] > M}}  \frac{ s^{M}t^{L-\hdim}}{(t + |x|)^{L+[\alpha]}} 
\int_{|y| \leq (t+|x|)/(2\cqn \eta^{M'})} \frac{|y|^{[\alpha]}}{(s + |y|)^{M+\hdim}} \; d\mu_\N(y),  \nonumber
\end{align}
where the last inequality used that 
$
 |y| \leq (t+|x|)/(2\cqn \eta^{M'})$
 and $|z|\leq \eta^{M'}|y| $ imply
\begin{align*}
 t + |xz| \geq  t+ \left(\frac{|x|}{\gamma} -|z|\right) \geq t + \frac{|x|}{\gamma} -\frac{t + |x|}{2\gamma} \geq \frac{t + |x|}{2\gamma},
\end{align*}
cf. Equation \eqref{quas}. For each multi-index $\alpha$ with $|\alpha| \leq M'$ and $[\alpha] > M$, the identity \eqref{eq:polar} yields
\begin{align*}
\int_{|y| \leq (t+|x|)/(2\cqn \eta^{M'})} \frac{|y|^{[\alpha]}}{(s + |y|)^{M+\hdim}} \; d\mu_\N(y) & \leq 
\int_{|y| \leq (t+|x|)/(2\cqn \eta^{M'})} \frac{1}{|y|^{M+\hdim -[\alpha]}} \; d\mu_\N(y) \\
& \lesssim_{ M} 
(t+|x|)^{[\alpha]-M}.
\end{align*}
Putting this into \eqref{eq:estimate of I_1} and using that
\begin{align*}
\frac{ s^{M}t^{L-\hdim}}{(t + |x|)^{L+M}} = \left( \frac{s}{t}\right)^M \left( \frac{t}{t+|x|}\right)^M \frac{t^{L-Q}}{(t+|x|)^{L}} \leq
\left( \frac{s}{t}\right)^M \frac{t^{-Q}}{(1+t^{-1}|x|)^{L}},
\end{align*}
we obtain
\begin{align*}
I_1 \lesssim_{ M,L} \|f\|_{(L+M^\ast)} \|g\|_{(M + \lfloor \hdim \rfloor +1)} \left( \frac{s}{t}\right)^M \frac{t^{-Q}}{(1+t^{-1}|x|)^{L}}.
\end{align*}
\\~\\
\emph{Estimate of $I_2$.}
 We use similar arguments for estimating $I_2$. First, note that if $|y| > (t+|x|)/(2\cqn \eta^{M'})$, then
$(s + |y|)^{-1} < |y|^{-1} \lesssim_{M} (t + |x|)^{-1}$. Hence
\begin{align*}
I_{2}
& \leq {\|f\|_{(L)}\|g\|_{(L)}}  \int_{|y| > (t+|x|)/(2\cqn \eta^{M'})} \frac{t^{-\hdim}}{(1+
t^{-1} |xy^{-1}|)^{L}}\frac{s^{L-\hdim}}{(s+ |y|)^{L}}\; d\mu_\N(y) \\
& \lesssim_{M,L}    {\|f\|_{(L)}\|g\|_{(L)}}   \frac{s^{ L-\hdim}}{(t + |x|)^{L}} \int_{|y| > (t+|x|)/(2\cqn \eta^{M'})} \frac{t^{-\hdim}}{(1+
t^{-1} |xy^{-1}|)^{L}}\; d\mu_\N(y) \\
& \leq   {\|f\|_{(L)}\|g\|_{(L)}}  \frac{s^{ L-\hdim}}{(t + |x|)^{L}} \int_{\N} \frac{1}{(1+
  |y|)^{L}} \;d\mu_\N(y)  \\
&\lesssim_{L}    {\|f\|_{(L)}\|g\|_{(L)}}  \left(\frac{s}{t}\right)^{M} \frac{t^{-\hdim}}{(1 + t^{-1}|x|)^{L}},
\end{align*}
where in the last step we used that $L > M + \hdim$.
\\~\\
\emph{Estimate of $I_3$}. For estimating $I_3$, we use the explicit expression of the  
the Taylor polynomial ${P}_{x,M}^{(f_t)}$ given in \eqref{eq:rough taylor polynomial}, 
\begin{align*}
{P}_{x,M}^{(f_t)}(y) = \sum_{\alpha \in \mathbb{N}_0^d : \ [\alpha] \leq M}\left( \sum_{\beta \in \mathbb{N}_0^d: \ [\beta] = [\alpha]} c_\beta(X^\beta f_{t} )(x)  \right)  y^\alpha,
\end{align*}
for constants $c_\beta \in \mathbb{C}$. Since $f\in \SV(N)$, for every multi-index $\beta$ with $[\beta] \leq M$,  we have
\[ |X^\beta f_t (x)|  
\leq {\|f\|_{(M + \lfloor \hdim \rfloor +1)}} 
\frac{t^{-[\beta]}
t^{-\hdim}}{ (1+ t^{-1}|x|)^{\hdim +[\beta]}}.
\]
Hence, using that $|y^{\alpha} | \lesssim_{M} |y|^{[\alpha]}$ {for all $[\alpha] \leq M$}, we see that
\begin{align*}
\left| {P}_{x,M}^{(f_t)}(y^{-1})\right|\lesssim_{M} {\|f\|_{(M + \lfloor \hdim \rfloor +1)}} 
\sum_{[\alpha] \leq M}  \frac{t^{-[\alpha]}t^{-\hdim}|y|^{[\alpha]}}{ (1 + t^{-1}|x|)^{\hdim + [\alpha] }} =   {\|f\|_{(M + \lfloor \hdim \rfloor +1)}} \sum_{[\alpha] \leq M}  \frac{|y|^{[\alpha]}}{ (t +  |x|)^{\hdim +[\alpha] }}.
\end{align*}
This, combined with the fact that also $g \in \SV(N)$, yields
\begin{align*}
I_{3} &\lesssim_{M}
 {\|f\|_{(M + \lfloor \hdim \rfloor +1)} \|g\|_{(L)} }
\int_{|y| > 
(t+|x|)/(2\cqn \eta^{M'})}
\frac{s^{ L-\hdim}}{(s+ |y|)^{L}}\sum_{[\alpha] \leq M}  \frac{|y|^{[\alpha]}}{ (t +  |x|)^{\hdim +[\alpha]}} \;d\mu_\N(y) \\
&\leq  {\|f\|_{(M + \lfloor \hdim \rfloor +1)} \|g\|_{(L)} }  \sum_{[\alpha] \leq M}\frac{s^{ L-\hdim} }{(t+ |x|)^{\hdim+[\alpha]}}
 \int_{|y| > (t+|x|)/(2\cqn \eta^{M'})}
\frac{1}{|y|^{L- [\alpha]}} \;d\mu_\N(y) .
\end{align*}
Since $L -[\alpha] \geq L - M > \hdim$, 
we have that {for every $\alpha$ with $[\alpha] \leq M$}, \[ \int_{|y| > (t+|x|)/(2\cqn \eta^{M'})}
\frac{1}{|y|^{L- [\alpha]}} \;d\mu_\N(y) \lesssim_{M, L} \frac{1}{(t +|x|)^{L - Q -[\alpha]}}, \] and hence 
\begin{align*}
I_3 &\lesssim_{M,L}  {\|f\|_{(M + \lfloor \hdim \rfloor +1)} \|g\|_{(L)} }\sum_{[\alpha] \leq M}\frac{s^{L-\hdim} }{(t+ |x|)^{\hdim+[\alpha]}} \frac{1}{(t+|x|)^{L-\hdim-[\alpha]}}\\
&\asymp_{M} {\|f\|_{(M + \lfloor \hdim \rfloor +1)} \|g\|_{(L)} }\frac{s^{L -\hdim} }{(t+ |x|)^{L}} \\
&\leq {\|f\|_{(M + \lfloor \hdim \rfloor +1)} \|g\|_{(L)} }\left( \frac{s}{t}\right)  ^{M}  \frac{t^{-\hdim} }{(1+t^{-1} |x|)^{L}} .
\end{align*}

Combining the estimates for $I_1,I_2,I_3$, we obtain that for $t \geq s$,
\begin{align} \label{eq:ftgsconvo}
\left|(f_{t}  \ast g_{s})(x) \right| \lesssim_{M,L} \|f\|_{(k_1)} \|g\|_{(k_2)} \left( \frac{s}{t}\right)^M \frac{t^{-Q}}{(1+t^{-1}|x|)^{L}},
\end{align}
where {$k_1 :=\max\{L+M^\ast, M + \lfloor \hdim \rfloor +1\}$ and $k_2 := \max\{L,M +\lfloor \hdim\rfloor +1\}$}.

\medskip

We now consider the case $s \geq t$. Using that $f_t \ast g_s(x)= (g^{\vee})_s \ast (f^{\vee})_{t} (x^{-1})$, together with the case $s \leq t$ proven above, we get
\begin{align} \label{eq:ftgsconvolution}
|f_t \ast g_s(x)| \lesssim_{M,L} {\|f^{\vee}\|_{(k_2)}}\|g^{\vee}\|_{(k_1)} \left(\frac{t}{s}\right)^{M} \frac{s^{-\hdim}}{(1 + s^{-1}|x|)^{L}}.
\end{align}
Note that for each $\alpha \in \mathbb{N}_0^d$, we have 
$X^\alpha (f^\vee) =(\widetilde{X}^\alpha f)^\vee$, and 
\[
\widetilde{X}^\alpha =\sum_{\substack{|\beta| \leq |\alpha| \\ [\beta] \geq [\alpha]}} P_{\alpha, \beta} X^\beta
\]
with each $P_{\alpha,\beta}$ being a polynomial of homogeneous degree $[\beta] -[\alpha]$. Using this fact we may deduce that for each $\ell \in \mathbb{N}$ and $f \in \mathcal{S} (\N)$, there holds $\|f^\vee\|_{(\ell)} \lesssim_{\ell} \|f\|_{(\ell + \ell^{'})}$, where 
\[
{\ell' := \max\{[\beta]:  \beta \in \mathbb{N}_0^d \text{ with } |\beta| \leq \lceil \ell \rfloor\}}
\]
with $\lceil \ell \rfloor$ being defined by \eqref{eq:defintion of M}. From this 
and
\eqref{eq:ftgsconvolution} it follows that for $s \geq t$,
\begin{align} \label{eq:ftgsconvolution_2}
|f_t \ast g_s(x)| \lesssim_{M,L} {\|f\|_{(k_2 + k_2')}\|g\|_{(k_1 + k_1')} } \left(\frac{t}{s}\right)^{M} \frac{s^{-\hdim}}{(1 + s^{-1}|x|)^{L}}.
\end{align}

Finally, putting \eqref{eq:ftgsconvo} in the case $t \geq s$ and \eqref{eq:ftgsconvolution_2} in the case $s \geq t$ together, we see that the desired estimate \eqref{eq:almost orthogonality estimate} holds for $k:= \max\{k_1 + k_1', k_2 + k_2'\}$. 
\end{proof}

\subsection{Calder\'on reproducing formulae}
This subsection is devoted to the construction of Calder\'on reproducing formulae, which will play a central role in our theory. 
The precise forms of these formulae are as follows.

\begin{definition} 
Let $\crk \in \SV(\N)$.
\begin{enumerate}
[(i)]\item The function $\crk$ is said to satisfy the \emph{discrete Calder\'on condition} if there exists $\drk\in\SV(\N)$ such that for all $f \in \SV'(N)$,
 \begin{align} \label{eq:discrete_calderon}
  f = \sum_{j \in \mathbb{Z}} f \ast \crk_{2^{-j}} \ast \drk_{2^{-j}} 
 \end{align} 
with convergence in $\SV'(N)$.

\item 
The function $\crk$ is said to satisfy the \emph{continuous Calder\'on condition} if there exists $\psi \in \SV(\N)$ such that for all $f \in \SV'(N)$,
\begin{align} \label{eq:continuous_calderon}
 f = \int_0^{\infty} f \ast \phi_t \ast \psi_t \; \frac{dt}{t} 
\end{align}
in $\SV'(N)$.
\end{enumerate}
 \end{definition}
The formula \eqref{eq:continuous_calderon} has to be interpreted in the sense that
\begin{align} \label{eq:weaksense1}
(f, \varphi) = \int_0^{\infty} (f \ast \phi_t \ast \psi_t, \varphi) \; \frac{dt}{t}
\end{align}
for all $\varphi \in \SV(\N)$.

For the construction of functions satisfying the Calder\'on condition \eqref{eq:discrete_calderon} (or \eqref{eq:continuous_calderon}) on a general homogeneous group, we will make use of functional calculus of some adequate homogeneous convolution operators \cite{glowacki1986stable, dziubanski1992schwartz, dziubanski1989remark}. Following \cite{glowacki1986stable}, we
define the distribution $P \in \mathcal{S}'(N)$ by
\begin{align} \label{eq:def_con_ker_P}
( P, f )  =  \lim_{\varepsilon \to 0} \int_{\N \setminus B_\varepsilon(e_\N)} \big( f( e_\N) - f(x) \big) \; \frac{d\mu_\N(x)}{\rho(x)^{\hdim+1}}, \quad f \in \mathcal{S}(N),
\end{align} 
where $\rho$ denotes a homogeneous quasi-norm that is smooth away from the origin.
Then $P$ is a homogeneous distribution of degree $-(Q+1)$, and hence its associated convolution operator 
\begin{align} \label{eq:def_P}
\P f = f \ast P, \quad f \in \mathcal{S}(N),
\end{align}
is homogeneous of degree $1$, which is  positive and essentially self-adjoint; see, e.g., \cite{dziubanski1989remark, dziubanski1992schwartz}. A direct calculation shows that $\P$ admits the integral representation for $f \in \SC(\N)$ 
\begin{align*}
    \P f(x) &=  \big(P, f(x \; \cdot^{-1}) \big)  
=\lim_{\varepsilon \to 0} \int_{\rho(z) \geq \varepsilon} \big( f(x) - f(xz^{-1}) \big) \; \frac{d\mu_\N(z)}{\rho(z)^{\hdim+1}} \\
&= \lim_{\varepsilon \to 0} \int_{\N \setminus B_\varepsilon(x)} \frac{f(x) - f(y)}{\rho( y^{-1}x )^{\hdim + 1}} \, d\mu_\N(y), \numberthis \label{eq:int_rep_P}
\end{align*}
for $x \in N$.
\begin{remark} ~
    \begin{enumerate}[(1)]
\item Although any choice of quasi-norm smooth away from the origin in \eqref{eq:def_con_ker_P} yields a  left-invariant, positive, essentially self-adjoint  operator $\P$ of homogeneity $1$, the resulting operator $\P$ can be quite different. For example, this has been shown in the special case of stratified groups, in particular the Heisenberg group, in a series of papers \cite{FeFr15, RoTh16, BrOlOr96, BrFoMo13}.

\item A closely related notion to the operators $\P$ defined above are the so-called \emph{fractional $p$-Laplacians} on general homogeneous groups $\N$, defined for $f \in \SC(N)$ by
\begin{align} \label{eq:int_rep_frac_p_L}
    (-\Delta_p)^sf(x) := 2 \lim_{\varepsilon \to 0} \int_{\N \setminus B_\varepsilon(x)} \frac{ |f(x) - f(y)|^{p-2} (f(x) - f(y))}{| y^{-1}x |^{\hdim + sp}} \, d\mu_\N(y), \quad x \in \N, 
\end{align}
 where $| \cdot |$ is an arbitrary homogeneous quasi-norm on $\N$, $s \in (0, 1)$ and $p  \in (1, \infty)$; see~\cite{RuSu19}. For $s = \frac{1}{2}$, $p = 2$ these generally non-linear, non-local operators $(-\Delta_2)^{1/2}$ are immediately seen to coincide with the operator $\P$ if the quasi-norms employed in \eqref{eq:int_rep_P} and \eqref{eq:int_rep_frac_p_L} are chosen to be the same. 
    \end{enumerate}
\end{remark}

Using the convolution operator $\P$ defined in Equation \eqref{eq:def_P}, we next construct functions satisfying the Calder\'on conditions.
The proof of the following proposition is based on a combination of ideas that can be found in, e.g., \cite{glowacki2013lp, fuehr2012homogeneous, geller2006continuous}.
 
\begin{proposition} \label{prop:construction_crk}
The following assertions hold:
\begin{enumerate}
    [\rm (i)]\item There exists $\crk \in \SV(N)$ satisfying the discrete Calder\'{o}n condition \eqref{eq:discrete_calderon} with $\crk = \drk$.
    \item There exist $\crk \in \SV(N)$ satisfying the continuous Calder\'{o}n condition \eqref{eq:continuous_calderon} with $\crk = \drk$.
\end{enumerate}
\end{proposition}
\begin{proof}
 First, let $m \in C_c^{\infty}(\mathbb{R}^+)$ 
 be arbitrary and write $\crk$ 
 for the convolution kernel of the operator $m(\P)$ defined by functional calculus, i.e., $m(\P) f = f \ast \crk$ for $f \in \SV(N)$. Then $\crk \in \mathcal{S}(N)$ by \cite[Theorem 4.1]{dziubanski1992schwartz}. In addition, an application of \cite[Lemma 7.1]{glowacki2013lp} yields that $\crk$ has all moments vanishing, so that $\crk \in \SV (N)$. Since the operator $\P$ is homogeneous of degree $1$, one has the identity
 \begin{align} \label{eq:spectral_dilate}
 m(t\P)f = f \ast \phi_t
 \end{align}
 for all $t > 0$. Denote by $E_{\P}$ the spectral measure of $\P$, so that $\P = \int_0^{\infty} \lambda \; dE_{\P} (\lambda)$.
\\~\\
(i) Choose a real-valued function $m \in C_c^{\infty}(\mathbb{R}^+)$ satisfying 
\begin{align*}
 \sum_{j\in\mathbb{Z}  } |m (2^{-j}\lambda)|^2  = 1 \quad \text{for all } \lambda \in \mathbb{R}^+.
\end{align*}
By the first part of the proof,  its convolution kernel $\crk \in \SV (N)$. 
Using the identity \eqref{eq:spectral_dilate}, it follows that for arbitrary $f \in \SV (N)$,
\begin{align*}
    f &= \int_{0}^{\infty}  \; dE_{\P} (\lambda) f
    = \int_{0}^{\infty} \sum_{j \in \mathbb{Z}}
     m (2^{-j}\lambda) m(2^{-j}\lambda) \;
     dE_{\P} (\lambda) f \\
     &= \sum_{j \in \mathbb{Z}} m(2^{-j}\P) m(2^{-j}\P) f 
     = \sum_{j \in \mathbb{Z}} f \ast \crk_{2^{-j}} \ast \crk_{2^{-j}}, \numberthis \label{eq:discrete_calderon_multiplier}
\end{align*}
with convergence of the series in the $L^2$-norm. 

We first show that the series on the right-hand of \eqref{eq:discrete_calderon_multiplier} converges in $\mathcal{S}_0(N)$.  Indeed, since $f \in \SV(\N)$ and $ \ \crk \ast X^\alpha \crk \in \SV(\N)$, it follows from Lemma \ref{aoe} (with $t=1=2^0$ and $s = 2^{-j}$) that, for any $M,L >0$,
\begin{align}
    \left|X^\alpha \big[f \ast \crk_{2^{-j}} \ast \crk_{2^{-j}}(x)\big]\right| 
 &= 2^{j[\alpha]}\left|f \ast \big( \crk \ast X^\alpha \crk \big)_{2^{-j}}(x)\right| \nonumber\\ &\lesssim 2^{j[\alpha]} 2^{-|j|M} \frac{2^{(j \wedge 0)\hdim}}{(1 + 2^{j\wedge 0}|x|)^L} 
 \label{eq:est_f_kmj}
% &= 2^{-|j|(M-[\alpha] -Q -L)} (1+ |x|)^{-L} ,
\end{align}
holds for all $j \in \mathbb{Z}$,
with implicit constant independent of $j$. 
Observe that the right-hand side of the inequality above can be estimated by
\begin{align*}
2^{j[\alpha]} 2^{-|j|M} \frac{2^{(j \wedge 0)\hdim}}{(1 + 2^{j\wedge 0}|x|)^L}  \leq 
\begin{cases}
2^{-|j|(M -[\alpha])}(1 +|x|)^{-L} \ \ &\text{if } j \geq 0, \\
 2^{-|j|(M +[\alpha] + \hdim -L)}(1 + |x|)^{-L} &\text{if } j <0.
\end{cases}
\end{align*}
Hence, for all $j \in \mathbb{Z}$,
\begin{align} \label{eq:X_f_kmj}
\left|X^\alpha \big[f \ast (\crk \ast \crk)_{2^{-j}}(x)\big]\right| \lesssim  2^{-|j|(M - [\alpha] -L)}(1 + |x|)^{-L}.
\end{align}
For any $\ell \in \mathbb{N}$, if we choose $L =\ell$ and  $M > 2\ell$ in \eqref{eq:X_f_kmj}, it follows that
\begin{align*}
\sum_{j \in \mathbb{Z}} \; \sup_{[\alpha] \leq \ell, x \in \N} (1 +|x|)^\ell \left|X^\alpha \big[f \ast (\crk \ast \crk)_{2^{-j}}(x)\big]\right| 
 \lesssim \sum_{j\in \mathbb{Z}} 2^{-|j|(M-2\ell)} <\infty.
\end{align*}
This shows that $ \sum_{j \in \mathbb{Z}} f \ast \crk_{2^{-j}} \ast \crk_{2^{-j}}$ converges in $\mathcal{S}(\N)$
to some element $f' \in \mathcal{S}(\N)$. In fact, being a limit of functions in $\mathcal{S}_0 (N)$, it follows that $f' \in \SV(\N)$.
 Equation \eqref{eq:discrete_calderon_multiplier} next implies that $f = f'$, and thus
\begin{align} \label{eq:calderon_representation_Schwartz}
f = \sum_{j \in \mathbb{Z}} f \ast \crk_{2^{-j}} \ast \crk_{2^{-j}} 
 \end{align} 
with convergence in $\SV(N)$.

To complete the proof of (i), we note that for any $f_1 \in \SV'(\N)$, $f_2 \in \SV(\N)$,
\begin{align*}
    \Bigl (f_1 - \sum_{j = -M}^{M} f_1 \ast \crk_{2^{-j}} \ast \crk_{2^{-j}} , f_2 \Bigr )
    = (f_1 ,f_2) - \Bigl (f_1  , \sum_{j = -M}^{M} f_2 \ast \crk_{2^{-j}} \ast \crk_{2^{-j}} \Bigr ) \to 0
\end{align*}
as $M \to \infty$.
Since $f_2$ is arbitrary, this proves \eqref{eq:discrete_calderon} with convergence in $\SV'(N)$.
\\~\\
(ii)  Let $m \in C_c^{\infty} (\mathbb{R}^+)$ be a real-valued function such that
 \[
\int_{0}^{\infty} |m(t) |^2 \; \frac{dt}{t} = 1.
\]
By a change of variables, this yields that 
\[
\int_{0}^{\infty} |m(t \lambda) |^2 \; \frac{dt}{t} = 1 \quad \text{for all } \lambda \in \mathbb{R}^+.
\]
Hence, for any $f \in \mathcal{S}_0 (N)$, 
\begin{align*}
f &= \int_0^{\infty} \; dE_{\P} (\lambda) f =  \int_0^{\infty} \int_0^{\infty} m(t \lambda)^2 \; \frac{dt}{t} \; dE_{\P} (\lambda) f  \\
&= \int_0^{\infty} m(t \P)^2 f \; \frac{dt}{t} = \int_0^{\infty} f \ast \crk_t \ast \crk_t \; \frac{dt}{t}, \numberthis \label{eq:continuous_calderon_proof}
\end{align*}
where the last integral converges as a Bochner integral in the $L^2(\N)$.

In order to extend \eqref{eq:continuous_calderon_proof} to elements of $\SV'(\N)$, we first show  the convergence of the integral in the Banach spaces $\mathcal{S}^k(\N)$ for arbitrary $k \in \NN$, cf. \eqref{eq:Sk}. Note that the map $t \mapsto f \ast \crk_t \ast \crk_t$ is continuous from $(0, \infty)$ into $\mathcal{S} (N) \hookrightarrow \mathcal{S}^k(\N)$ as the composition of continuous maps; in particular, it is strongly measurable.
Moreover, an application of
Lemma~\ref{aoe} gives, for any $L, M \in \mathbb{N}$,
\begin{align*}
\left|X^\alpha \big[f \ast \crk_t \ast \crk_{t}(x)\big]\right| 
 &= t^{-[\alpha]}\left|f \ast \big( \crk \ast X^\alpha \crk \big)_{t}(x)\right| \nonumber\\ &\lesssim t^{-[\alpha]}(t\wedge t^{-1})^{M}\frac{(t^{-1}\wedge1)^\hdim}{(1 + (t^{-1} \wedge 1)|x|)^L}  \nonumber\\
 &\lesssim (t \wedge t^{-1})^{(M -[\alpha] -L)}
(1 + |x|)^{-L}. \numberthis \label{eq:con_est_f_kmt}
\end{align*}
In combination, choosing $L = k$ and  $M > 2 k$, it follows that
\begin{align}
\begin{split} \label{eq:Bochner_integrability}
 \int_0^{\infty} \| f \ast \crk_t \ast \crk_t \|_{(k)} \; \frac{dt}{t} &\lesssim \int_0^\infty \sup_{[\alpha] \leq k, x \in \N}(1+|x|)^k \left|X^\alpha \big[f \ast \crk_t \ast \crk_{t}(x)\big]\right|\frac{dt}{t} \\ &\lesssim \int_0^\infty (t \wedge t^{-1})^{M -2k }\frac{dt}{t} <\infty.
\end{split}
\end{align}
This shows that the Bochner integral $\int_0^{\infty} f \ast \crk_t \ast \crk_t \; \frac{dt}{t}$ converges in $\mathcal{S}^k(\N)$
to some element $f' \in \mathcal{S}^k(\N)$. By Equation \eqref{eq:continuous_calderon_proof}, it follows that $f = f'$, and thus
\begin{align} \label{eq:cont_rep_formula_SV}
 f = \int_0^{\infty} f \ast \phi_t \ast \phi_t \; \frac{dt}{t} 
\end{align}
in $\mathcal{S}^k(\N)$.

Let now $f_1 \in \SV'(\N)$ and $f_2 \in \SV(\N)$. Then an application of the Hahn-Banach theorem yields a continuous extension $f_1'$ of $f_1$ to $\mathcal{S}(N)$. By \cite[Equation (3.45)]{FR}, there exists some $k_{f_1} \in \NN$ such that the restriction $f_1'|_{\SV}$ is continuous with respect to $\| \cdot \|_{(k_{f_1})}$. 
Hence, another application of the Hahn-Banach theorem yields that $f_1$ can be extended to a continuous linear functional $f_1''$ on $\mathcal{S}^k(\N)$. Therefore,
\begin{align*} 
     (f_1, f_2)  &= \big ( f_1'' ,  \int_0^{\infty} f_2 \ast \crk_t \ast \crk_t \; \frac{dt}{t} \big ) 
    = \int_0^{\infty} (f''_1 , f_2 \ast \crk_t \ast \crk_t) \; \frac{dt}{t} = \int_0^{\infty} (f_1 , f_2 \ast \crk_t \ast \crk_t) \; \frac{dt}{t} \\
    &= \int_0^{\infty} (f_1 \ast \crk_t \ast \crk_t, f_2) \; \frac{dt}{t}, \numberthis \label{eq:aux_step_Bochner_integral}
\end{align*}
where interchanging the order of functional evaluation and integration is justified by, e.g, \cite[Section 1.2]{hytonen2016analysis}. Since $f_1 \in \SV'(\N)$ and $f_2 \in \SV(\N)$ are arbitrary, this is precisely the integral representation that gives \eqref{eq:continuous_calderon}.
\end{proof}

For the construction of functions satisfying Calder\'{o}n conditions  on graded  groups, one can alternatively invoke spectral multipliers of any positive Rockland operators $\RLO$ (of some homogeneous degree $\hdeg \in \NN$) as was shown in \cite{geller2006continuous, velthoven2022comptes}, where a proof strategy similar to the one of Proposition~\ref{prop:construction_crk} is applied to $\RLO^{\frac{1}{\hdeg}}$ instead of $\P$. This particularly works out for any sub-Laplacian on a stratified group since $-\sL$ is a positive Rockland operator of homogeneous degree $\hdeg = 2$; see \cite{fuehr2012homogeneous, geller2006continuous} for similar constructions in this setting.

\subsection{Sub-mean-value property} 
The following sub-mean-value property for convolution products will be used in several crucial proofs in the sequel. Similar estimates for Euclidean spaces can be found in, e.g., \cite[Chapter V]{stromberg1989weighted} and \cite[Section 2]{Bui1996maximal}.

\begin{lemma} \label{lem:central_estimate}
Let $M, r \in (0,\infty)$ and $f \in \SV'(\N)$. 
\begin{enumerate}
[\rm (i)]\item If $\crk$ satisfies the discrete Calder\'{o}n condition~\eqref{eq:discrete_calderon}, then there exists a constant $C =C (M, r) > 0$ such that 
\begin{align} \label{eq:central_estimate_i}
 |f\ast \crk_{2^{-j}}(y) |
    \leq C \biggl ( \sum_{k \in \mathbb{Z}}
    2^{-|j-k|Mr} \int_{\N} \frac{2^{k\hdim}|f \ast \crk_{2^{-k}}(z)|^r}{(1+2^{k}|z^{-1}y|)^{Mr}} \; d\mu_\N(z) \biggr )^{1/r}
\end{align}
for all $j \in \mathbb{Z}$ and $y \in \N$.
\item If $\crk$ satisfies the continuous Calder\'{o}n condition~\eqref{eq:continuous_calderon}, then there exists a constant $C = C(M,r) > 0$ such that 
\begin{align} \label{eq:central_estimate_ii}
     |f\ast \phi_t(y)|
    \leq C \biggl ( \int_0^\infty
    \Bigl ( \frac{s}{t} \wedge \frac{t}{s} \Bigr )^{Mr} \int_{\N} \frac{s^{-\hdim}|f \ast \crk_s(z)|^r}{(1+s^{-1}|z^{-1}y|)^{Mr}} \; d\mu_\N (z)\frac{ds}{s} \biggr )^{1/r}
\end{align}
for all $t > 0$ and $y \in \N$.
\end{enumerate}
\end{lemma}
\begin{proof}
We only prove assertion (ii);  the proof of (i) being similar. 
Since the right-hand side of \eqref{eq:central_estimate_ii} decreases for increasing $M$, it suffices to show \eqref{eq:central_estimate_ii} for all $M \in \mathbb{N}$. Throughout, we let $f \in \SV'(\N) \setminus \{0\}$ and $M \in \mathbb{N}$ be fixed.

By assumption, there exists $\drk \in \SV(\N)$ such that, for all $f \in \SV'(N)$,
\begin{align}
 f = \int_0^{\infty} f \ast \crk_s \ast \drk_s \; \frac{ds}{s}
\end{align}
with convergence in $\SV'(N)$. 
It follows that, for all $t >0$ and $y \in \N$,
\begin{equation} 
\label{pointwise rep}
\begin{split}
 f\ast\crk_t (y) 
 &=\int_0^{\infty} f \ast \crk_s \ast \drk_s \ast \crk_t(y)\; \frac{ds}{s}\\
 &=\int_0^{\infty} \int_{\N} \big( f \ast \crk_s \big)(z) \big( \drk_s \ast \crk_t \big) (z^{-1}y)\; d\mu_\N(z)  \frac{ds}{s}.
\end{split}
\end{equation}
Since $\crk,\drk \in \SV(N)$, an application of Lemma \ref{aoe} yields that
\begin{align*}
| \drk_s \ast \crk_t (z^{-1}y)| \lesssim_M \Bigl ( \frac{s}{t} \wedge \frac{t}{s} \Bigr )^{2M}\frac{(s^{-1}\wedge t^{-1})^\hdim }{\big(1 + (s^{-1}\wedge t^{-1})|z^{-1}y|\big)^{M}}.
\end{align*}
Inserting this into \eqref{pointwise rep} gives
\begin{align*}
|f\ast\crk_t (y) | &\lesssim_M \int_0^{\infty} \int_{\N} \Bigl ( \frac{s}{t} \wedge \frac{t}{s} \Bigr )^{2M}\frac{(s^{-1}\wedge t^{-1})^\hdim| f \ast \crk_s (z) |}{\big(1 + (s^{-1}\wedge t^{-1})|z^{-1}y|\big)^{M}}\;  d\mu_\N(z) \frac{ds}{s} \\
&\leq \int_0^{\infty} \int_{\N} \Bigl ( \frac{s}{t} \wedge \frac{t}{s} \Bigr )^{2M}\frac{s^{-\hdim} | f \ast \crk_s (z) |}{(1 + s^{-1}|z^{-1}y|)^{M}}    \; d\mu_\N (z)\frac{ds}{s}. \numberthis \label{first pointwise}
\end{align*}
We next split the proof into the cases $r \in (1,\infty)$ and $r \in (0,1]$. 
\\~\\
\textbf{Case 1.} Using the inequality  
\begin{align*}
 \bigl (1 +(s^{-1}\wedge t^{-1})|z^{-1}y| \bigr )^M 
 \geq  \Bigl ( \frac{s}{t} \wedge 1 \Bigr )^{M} (1+s^{-1}|z^{-1}y|)^M
 \geq  \Bigl ( \frac{s}{t} \wedge \frac{t}{s} \Bigr )^{M} (1+s^{-1}|z^{-1}y|)^M, 
\end{align*}
it follows from \eqref{first pointwise} that 
\begin{align} \label{second pointwise}
|f\ast\crk_t (y) | &\lesssim_M \int_0^{\infty} \int_{\N} \Bigl ( \frac{s}{t} \wedge \frac{t}{s} \Bigr )^{M}\frac{s^{-\hdim} | f \ast \crk_s (z) |}{(1 + s^{-1}|z^{-1}y|)^{M}}    \; d\mu_\N (z)\frac{ds}{s}. 
\end{align}
If $r \in (1, \infty)$,  we denote by $r' \in (1,\infty)$ its dual exponent and let $\varepsilon > 0$.
 We write
\begin{align*}
\frac{s^{-\hdim} | f \ast \crk_s (z) |}{(1 + s^{-1}|z^{-1}y|)^{M}} =  \frac{ s^{-\hdim /r} | f \ast \crk_s (z) |}{(1 + s^{-1}|z^{-1}y|)^{M-(Q+\varepsilon)/r'}} \frac{ s^{-\hdim/r'} }{(1 + s^{-1}|z^{-1}y|)^{(Q+\varepsilon)/r'}}, 
\end{align*}
and use H\"{o}lder's inequality for the integral over $N$ on the right-hand side of \eqref{second pointwise} to obtain
\begin{align} \label{eq:power r}
 |f\ast \crk_t(y)| \lesssim_{r, M} \int_0^{\infty}  \Bigl ( \frac{s}{t} \wedge \frac{t}{s} \Bigr )^{M} \left(\int_{\N} \frac{ s^{-\hdim } | f \ast \crk_s (z) |^r}{(1 + s^{-1}|z^{-1}y|)^{Mr-(Q+\varepsilon)r/r'}}\; d\mu_\N (z)\right)^{1/r} \frac{ds}{s}  
\end{align}
where it is used that $(1+|\cdot|)^{-(\hdim + \varepsilon)} \in L^1 (N)$.
Similarly, we write
\begin{align*}
\Bigl ( \frac{s}{t} \wedge \frac{t}{s} \Bigr )^{M}= \Bigl ( \frac{s}{t} \wedge \frac{t}{s} \Bigr )^{M-\varepsilon/r'}  \Bigl ( \frac{s}{t} \wedge \frac{t}{s} \Bigr )^{\varepsilon/r'} ,
\end{align*}
and use H\"{o}lder's inequality for the integral over $(0,\infty)$ on the right-hand side of \eqref{eq:power r} to obtain
\begin{align*}
 |f \ast \crk_t (y)| \lesssim_{r, M}
\left(\int_0^{\infty}  \Bigl ( \frac{s}{t} \wedge \frac{t}{s} \Bigr )^{(M-\varepsilon/r')r}  \int_{\N} \frac{ s^{-\hdim } | f \ast \crk_s (z) |^r}{(1 + s^{-1}|z^{-1}y|)^{[M-(Q+\varepsilon)/r']r}}  \; d\mu_\N (z) \frac{ds}{s}\right)^{1/r},
\end{align*}
where we used that $\int_0^\infty \big( \frac{s}{t} \wedge \frac{t}{s} \big )^{\varepsilon} \; \frac{ds}{s}<\infty$

Lastly, to obtain \eqref{eq:central_estimate_ii}, we choose $M' \in \mathbb{N}$ such that $M'-(Q+\varepsilon)/r' > M$. 
Then replacing $M$ by $M'$ in the previous estimate, we deduce that
\begin{align*}
 |f \ast \crk_t (y)| \lesssim_{r, M}
\left(\int_0^{\infty}  \Bigl ( \frac{s}{t} \wedge \frac{t}{s} \Bigr )^{Mr}  \int_{\N} \frac{ s^{-\hdim } | f \ast \crk_s (z) |^r}{(1 + s^{-1}|z^{-1}y|)^{Mr}}  \; d\mu_\N (z) \frac{ds}{s}\right)^{1/r}
\end{align*}
as the right-hand side decreases for increasing $M$.
\\~\\
\textbf{Case 2.} We next consider the case $r \in (0,1]$. 
Dividing both sides of the inequality \eqref{first pointwise} by $(1+t^{-1}|y^{-1}x|)^M$, and using the inequality
\begin{align*}
\big(1 + (s^{-1} \wedge t^{-1})|z^{-1}y|\big)^{M}\big(1 +t^{-1}|y^{-1}x|\big)^{M} 
\gtrsim \Bigl ( \frac{s}{t} \wedge 1\Bigr )^M(1+s^{-1}|z^{-1}x|)^M,
\end{align*}
we get
\begin{align*}
\frac{|f\ast\crk_t (y) |}{(1 +t^{-1}|y^{-1}x|)^{M}} 
\lesssim \int_0^{\infty} \int_{\N}\Bigl ( \frac{s}{t} \wedge \frac{t}{s} \Bigr )^{M } \frac{s^{-Q}| f \ast \crk_s (z) |}{(1 +s^{-1}|z^{-1}x|)^{M}}  \; d\mu_\N (z)\frac{ds}{s}.
\end{align*}
Hence, for $u > 0$,
\begin{align*}
 \left( \frac{u}{t} \wedge \frac{t}{u}\right)^M \frac{|f\ast\crk_t (y)|}{(1+ t^{-1}|y^{-1}x|)^M} \lesssim \int_0^{\infty}  \int_{\N}\left(\frac{u}{s} \wedge \frac{s}{u} \right)^{M} \frac{s^{-\hdim}| f \ast \crk_s (z) |}{(1 + s^{-1}|z^{-1}x|)^{M}}    \; d\mu_\N (z) \frac{ds}{s},
\end{align*}
where we have used the elementary inequality
$
\left( \frac{u}{t} \wedge \frac{t}{u}\right)^M \Bigl ( \frac{s}{t} \wedge \frac{t}{s} \Bigr )^{M} \leq  \left( \frac{u}{s} \wedge \frac{s}{u}\right)^M.
$

We next introduce the auxiliary function
\begin{align*}
\Lambda_{u,M}f(x) = \sup_{t>0} \sup_{y \in \N} \left( \frac{u}{t} \wedge \frac{t}{u}\right)^{M} \frac{|f\ast\crk_t (y)|}{(1+ t^{-1}|y^{-1}x|)^M}, \quad x \in N,
\end{align*}
for fixed $u >0$, $M \in \mathbb{N}$. Observe that
\begin{align} \label{r inequality}
\Lambda_{u,M}f(x) \lesssim_{M} \big(\Lambda_{u,M}f(x) \big)^{1-r}\int_0^{\infty} \int_{\N} \left(\frac{u}{s} \wedge \frac{s}{u} \right)^{Mr} \frac{s^{-\hdim}| f \ast \crk_s (z) |^r}{(1 + s^{-1}|z^{-1}x|)^{Mr}}    \; d\mu_\N (z)\frac{ds}{s}.
\end{align}
In order to derive the desired conclusion from \eqref{r inequality}, 
we first show the following claim.
\\~\\
\textbf{Claim:} There exists $M_f >0$ (depending on $f$) such that if $M \geq M_f$, then
\begin{align} \label{lambda finite}
   0< \Lambda_{u,M}f(x) < \infty, \quad u > 0, \; x \in N.
\end{align}

For proving the claim \eqref{lambda finite}, note that, by \eqref{bounds_of_convolution}, there exist  $k_f \in \mathbb{N}$ such that
\begin{align*} 
|f \ast \crk_t(y)|  &\lesssim_f (1+|y|)^{k_f} \| \crk_t \|_{(k_f)}\\
&\leq  (1+|y|)^{k_f} \sup_{[\alpha] \leq k_f}\sup_{\ w \in \N} (1+|w|)^{k_f} t^{-\hdim-[\alpha]}\left|(X^\alpha \crk )(t^{-1}w)\right|\\
& \lesssim_f (1+|y|)^{k_f} \sup_{[\alpha] \leq k_f}\sup_{w\in \N} (1+|w|)^{k_f} t^{-\hdim-[\alpha]} (1 + t^{-1}|w|)^{-k_f} \|\phi\|_{(k_f)}
\numberthis \label{eq:fphit}
\end{align*}
for $x, y \in \N$.
Since
\begin{align*}
(1+|y|)^{k_f}  \lesssim (t\vee 1)^{k_f} (1 +|x|)^{k_f} (1 +t^{-1}|y^{-1}x|)^{k_f}
\end{align*}
and
\begin{align*}
(1+|w|)^{k_f} (1 +t^{-1}|w|)^{-k_f} \leq (t\vee 1)^{k_f},
\end{align*}
we can estimate \eqref{eq:fphit} further by 
\begin{align*}
|f \ast \crk_t (y)| &\lesssim (t\vee 1)^{2k_f} (1 +|x|)^{k_f} (1 +t^{-1}|y^{-1}x|)^{k_f} \|\phi\|_{(k_f)}\sup_{[\alpha] \leq k_f}t^{-Q - [\alpha]} \\
& \leq (t \wedge t^{-1})^{-\max(2k_f, Q +k_f)} (1 +|x|)^{k_f} (1 +t^{-1}|y^{-1}x|)^{k_f} \|\phi\|_{(k_f)}.
\end{align*}
Hence, if $M > \max(2k_f, Q +k_f)$, then
\begin{align*}
\Lambda_{u,M}f(x)& \leq \sup_{t >0} \Bigl ( \frac{u}{t} \wedge \frac{t}{u} \Bigr )^{M_f} (t \wedge t^{-1})^{- M_f} (1 + |x|)^{k_f} \leq (u \vee u^{-1})^{M_f} (1 +|x|)^{k_f} 
< \infty.
\end{align*}
for all $u > 0$ and $x \in \N$. This proves the claim \eqref{lambda finite}.

Using the claim \eqref{lambda finite}, note that \eqref{r inequality} yields
\begin{align*} 
|f \ast \crk_u(x) |^{r} &\leq \Lambda_{u,M}f(x)^{r} \\
&\lesssim_{f,M}  \int_0^{\infty} \int_{\N}\left(\frac{u}{s} \wedge \frac{s}{u} \right)^{Mr}  \frac{s^{-\hdim}| f \ast \crk_s (z) |^r}{(1 + s^{-1}|z^{-1}x|)^{Mr}} \; d\mu_\N (z) \frac{ds}{s} \numberthis \label{r inequality 2}
\end{align*}
provided $M\geq M_f$ and $u >0$.
Since the right-hand side of \eqref{r inequality 2} decreases as $M$ increases, this proves the desired estimate \eqref{eq:central_estimate_ii} for all $M \in \mathbb{N}$ with an implicit depending on $f, r$ and $M$.

Lastly, we use \eqref{r inequality 2} to show that the estimate \eqref{eq:central_estimate_ii} also holds for a constant independent of $f$. For this, we may assume that the integral on the right-hand side of \eqref{eq:central_estimate_ii}  is finite.
By \eqref{r inequality 2}, we have
\begin{align*} 
|f \ast \phi_v (y) |^{r} \lesssim_{f, M} \int_0^{\infty}\int_{\N}  \left(\frac{v}{s} \wedge \frac{s}{v} \right)^{2Mr} \frac{s^{-\hdim}| f \ast \crk_s (z) |^r}{(1 + s^{-1}|z^{-1}y|)^{2Mr}}    \; d\mu_\N (z) \frac{ds}{s},
\end{align*}
and hence
\begin{align*} 
\left( \frac{u}{v}\wedge \frac{v}{u}\right)^{Mr}\frac{|f \ast \crk_v (y) |^{r} }{(1+v^{-1}|y^{-1}x|)^{Mr}}
 \lesssim_{f, M} \int_0^{\infty} \int_{\N} \left( \frac{u}{s}\wedge \frac{s}{u}\right)^{Mr}\frac{s^{-Q}| f \ast \crk_s (z) |^r}{(1+s^{-1}|z^{-1}x|)^{Mr}}    \; d\mu_\N (z)\frac{ds}{s},
\end{align*}
where we used the inequality
\begin{align*}
 \frac{1}{(1 + s^{-1}|z^{-1}y|)^{2Mr}(1+v^{-1}|y^{-1}x|)^{Mr}} \leq  \left(\frac{v}{s} \wedge \frac{s}{v} \right)^{-Mr} \frac{1}{(1 + s^{-1}|z^{-1}y|)^{Mr}(1+s^{-1}|y^{-1}x|)^{Mr}} .
\end{align*}
Taking the supremum over $v >0$ and $y \in \N$ gives
\begin{align*} 
\Lambda_{u,M}f(x)^r\lesssim_{f, M} \int_0^{\infty} \int_{\N} \left( \frac{u}{s}\wedge \frac{s}{u}\right)^{Mr}\frac{s^{-\hdim}| f \ast \crk_s (z) |^r}{(1+s^{-1}|z^{-1}x|)^{Mr}}   \; d\mu_\N (z) \frac{ds}{s}.
\end{align*}
Since the right-hand side of the estimate above is assumed to be finite, it follows also that $\Lambda_{u,M}f(x)^r < \infty$.
Using this in \eqref{r inequality}, we conclude that 
\begin{align*} 
|f\ast \crk_u (x)| &\leq \Lambda_{u,M}f(x)^{r} \\
&\lesssim_{M,r} \int_0^{\infty} \left(\frac{u}{s} \wedge \frac{s}{u} \right)^{Mr} \int_{\N} \frac{s^{-\hdim}| f \ast \crk_s (z) |^r}{(1 + s^{-1}|z^{-1}x|)^{Mr}}    \; d\mu_\N (z)\frac{ds}{s},
\end{align*}
with implicit constant independent of $f$. This shows the desired estimate \eqref{eq:central_estimate_ii} in the case $r \in (0,1]$ and completes the proof.
\end{proof}

\subsection{Maximal functions} \label{sec:maximal}
For $r \in (0, \infty)$, the associated Hardy-Littlewood maximal function $\mathcal{M}_r f$ of a function $f \in L^1_{\loc} (N)$ is defined by
\begin{align}
\mathcal{M}_r (f)(x) = \sup_{t > 0} \bigg ( \dashint_{B_t(x)} |f(y)|^r \, d\mu_\N (y) \bigg )^{1/r}, \quad x \in N. \label{HLmax}
\end{align}
For $r = 1$, we also simply write $\mathcal{M} = \mathcal{M}_1$. The simple identity
\[
[\mathcal{M}_r (f)]^r = \mathcal{M}(|f|^r)
\]
will often be used to employ maximal inequalities. In particular, it implies that
$
\| \mathcal{M}_r f \|_{L^p} \lesssim \| f \|_{L^p} 
$
for all $r < p \leq \infty$. We will also use the following \emph{majorant property}.

\begin{lemma} \label{lem:majorant}
Let $f \in L^1_{\loc} (\N)$. If $a > \hdim$, then there exists $C>0$ such that
\begin{align}
\int_\N \frac{t^{-\hdim}|f(y)|}{(1+t^{-1}|y^{-1}x|)^a}d\mu_\N(y)
\leq C \mathcal{M}(f)(x) 
\end{align}
for all $x \in \N$ and $t>0$.
\end{lemma}
\begin{proof}
For fixed $x \in \N$ and $t>0$, we write 
\begin{align*}
\int_\N \frac{t^{-\hdim}|f(y)|}{(1+t^{-1}|y^{-1}x|)^a}d\mu_\N(y)& = \int_{B_t(x)}\frac{t^{-\hdim}|f(y)|}{(1+t^{-1}|y^{-1}x|)^a} \; d\mu_\N(y) \nonumber\\
&\quad \quad + \sum_{j =1}^\infty\int_{B_{2^{j}t}(x)\backslash B_{2^{j-1} t}(x)} \frac{t^{-\hdim}|f(y)|}{(1+t^{-1}|y^{-1}x|)^a} \; d\mu_\N(y).
\end{align*}
The first integral can be estimated as 
\begin{align*}
\int_{B_t(x)}\frac{t^{-\hdim}|f(y)|}{(1+t^{-1}|y^{-1}x|)^a} \; d\mu_\N(y) \lesssim \dashint_{B_t(x)} |f(y)| \; d\mu_\N(y) \leq \mathcal{M}(f)(x),
\end{align*}
while, for each $j \in \mathbb{N}$,
\begin{align*} 
\int_{B_{2^{j}t}(x)\backslash B_{2^{j-1} t}(x)} \frac{t^{-\hdim}|f(y)|}{(1+t^{-1}|y^{-1}x|)^a} \; d\mu_\N(y)
&\lesssim 2^{-j(a-\hdim)} \dashint_{B_{2^{j}t}(x)}|f(y)| \;d\mu_\N (y) \\
&\leq 2^{-j
(a-\hdim)}\mathcal{M}(f)(x),
\end{align*}
where it used that 
$(1+t^{-1}|y^{-1}x|)^a \asymp 2^{ja}$ for $y \in B_{2^{j}t}(x)\backslash B_{2^{j-1} t}(x)$.
Since $a -\hdim >0$, a combination of both estimates easily yields the result.
\end{proof}

In addition to the Hardy-Littlewood maximal function, the Peetre-type maximal function of a distribution will play an essential role in the theory developed in this paper. Given $f \in \mathcal{S}_0'(N)$ and $a,t > 0$, the \emph{Peetre-type maximal function} of $f$ associated to $\phi \in \mathcal{S}_0(\N)$ is defined as
\begin{align*}
	(\crk_t^* f)_\PTpar(x) =
	\sup_{y \in \N} \frac{|f \ast \crk_t(y)| }{(1+t^{-1}|y^{-1}x|)^\PTpar},\quad x \in \N.
\end{align*}
Note that $f \ast \phi_t (x) \leq (\crk_t^* f)_\PTpar(x)$ for all $x \in N$ and $a,t>0$.
We record the following simple consequence of Lemma \ref{lem:central_estimate}.

\begin{corollary}
\label{coro:central_estimate}
Let $M, a, r \in (0,\infty)$ with $M >a$ and $f \in \SV'(\N)$. 
\begin{enumerate}
[\rm (i)]\item If $\crk$ satisfies the discrete Calder\'{o}n condition ~\eqref{eq:discrete_calderon}, then there exists a constant $C =C (M, r) > 0$ such that 
\begin{align} \label{eq:central_r_peetre_i}
  (\crk_{2^{-j}}^* f)_\PTpar(x) 
    \leq C \biggl ( \sum_{k \in \mathbb{Z}}
    2^{-|j-k|(M-a)r} \int_{\N} \frac{2^{k\hdim}|f \ast \crk_{2^{-k}}(z)|^r}{(1+2^{k}|z^{-1}x|)^{ar}} \; d\mu_\N(z) \biggr )^{1/r}
\end{align}
for all $j \in \mathbb{Z}$ and $x \in \N$.
\item If $\crk$ satisfies the continuous Calder\'{o}n condition~\eqref{eq:continuous_calderon}, then there exists a constant $C = C(M,r) > 0$ such that 
\begin{align} \label{eq:central_r_peetre_ii}
      (\crk_t^* f)_\PTpar(x) 
    \leq C \biggl ( \int_0^\infty
    \Bigl ( \frac{s}{t} \wedge \frac{t}{s} \Bigr )^{(M-a)r} \int_{\N} \frac{s^{-\hdim}|f \ast \crk_s(z)|^r}{(1+s^{-1}|z^{-1}x|)^{ar}} \; d\mu_\N (z)\frac{ds}{s} \biggr )^{1/r}
\end{align}
for all $t > 0$ and $x \in \N$.
\end{enumerate}
\end{corollary}
\begin{proof}
Dividing both sides of \eqref{eq:central_estimate_i} by $(1 +2^{j}|y^{-1}x|)^{ar}$, taking the supremum over $y \in \N$, and using the fundamental inequality
\begin{align*}
\frac{1}{(1 +2^k|z^{-1}y|)^{Mr}(1+2^j|y^{-1}x|)^{ar}}
\lesssim \frac{2^{|j -k|ar}}{(1 +2^k |z^{-1}x|)^{ar}} ,
\end{align*}
directly gives \eqref{eq:central_r_peetre_i}. The proof of \eqref{eq:central_r_peetre_ii} is similar.
\end{proof}

\subsection{Vector-valued inequalities}
For $p,q\in (0,\infty]$, the spaces $L^p(\ell^q)$ and $\ell^q(L^p)$ are defined to consist of all sequences $(f_j)_{j=-\infty}^\infty$ of measurable functions  $f_j : \N \to \mathbb{C}$ with finite
quasi-norms
\begin{align*}
\big\|(f_j)_{j=-\infty}^\infty\big\|_{L^p(\ell^q)} :=  \Biggl \| \Biggl( \, \sum_{j=-\infty}^\infty |f_j (\, \cdot \,)|^q \Biggr)^{\frac{1}{q}} \Biggr \|_{L^p}
\end{align*}
and
\begin{align*}
\big\|(f_j)_{j=-\infty}^\infty\big\|_{\ell^q(L^p)} :=\Biggl ( \, \sum_{j=-\infty}^\infty
\|f_j\|^q_{L^p} \Biggr)^{\frac{1}{q}},
\end{align*}
respectively (with the usual modifications for $q = \infty$).

We start by recalling the following boundedness property of the Hardy-Littlewood maximal operator.

\begin{lemma} \label{vector-valued HL}
Let $r >0$ and $p, q \in (0, \infty]$. If $r < p <\infty$ and $r < q \leq \infty$ (resp. $r < p \leq \infty$ and $0<q\leq \infty$), then the Hardy-Littlewood maximal operator $\mathcal{M}_r$ is bounded
on $L^p(\ell^q)$ (resp. $\ell^q(L^p)$).
\end{lemma}

The assertions for $L^p(\ell^q)$ are a special case of the \emph{Fefferman-Stein maximal inequalities} for general spaces of homogeneous type (see, e.g., \cite{grafakos2009vector, S} for proofs), while the assertions for $\ell^q(L^p)$ easily follow from the scalar-valued case.

Given a sequence $(g_j)_{j = - \infty}^{\infty}$ of measurable functions $g_j : N \to [0, \infty)$ and $\delta > 0$, we define the functions
\begin{align} \label{eq:hdl}
h^{(\delta)}_\ell   := \sum_{j =-\infty}^\infty 2^{-|j-\ell|\delta} g_j, \quad  \ell \in \mathbb{Z}.
\end{align}
Some basic properties of these functions are collected in the following lemma.
Its proof is an adaption of \cite[Lemma 2]{Rychkov} and included here for the sake of being self-contained. 

\begin{lemma}[\cite{Rychkov}] \label{Ul} 
Let $p,q \in (0, \infty]$ and $\delta >0$. For a sequence $(g_j)_{j=-\infty}^\infty$ of measurable functions $g_j : \N \to [0,\infty)$, define $(h^{(\delta)}_{\ell})_{\ell = - \infty}^{\infty}$ as in Equation \eqref{eq:hdl}. 

Then there is a constant $C=C(p,q,\delta) > 0$ such that
\begin{align}
\big\|(h^{(\delta)}_{\ell})_{\ell = - \infty}^{\infty} \big\|_{L^p(\ell^q)} &\leq C \big\|(g_j)_{j=-\infty}^\infty\big\|_{L^p(\ell^q)}, \label{eq:Ry_Lp_lq} \\
\big\|(h^{(\delta)}_{\ell})_{\ell = - \infty}^{\infty} \big\|_{\ell^q(L^p)}& \leq C \big\|(g_j)_{j=-\infty}^\infty\big\|_{\ell^q(L^p)}. \label{eq:Ry_lq_Lp}
\end{align}
In addition, there exists $C' = C'(q, \delta) > 0$ such that the pointwise estimate 
\begin{align}\label{Rylq}
\big\| (h^{(\delta)}_{\ell} (x))_{\ell = - \infty}^{\infty} \big\|_{\ell^q}& \leq C' \big\| ( g_j(x) )_{j = -\infty}^\infty \big\|_{\ell^q}
\end{align}
holds for all $x \in \N$.
\end{lemma}

\begin{proof}
For simplicity, we write $h_{\ell} := h^{(\delta)}_{\ell}$ throughout the proof. 

We start by showing \eqref{Rylq}. If $q \geq  1$, then a change of variable and Minkowski's inequality yields
\begin{align*}
\big\| ( h_\ell(x) )_{\ell = -\infty}^\infty \big\|_{\ell^q} 
&=\biggl (\sum_{\ell =-\infty}^\infty \biggl |\sum_{j =-\infty}^\infty 2^{-|j|\delta} g_{j+\ell}(x) \biggr |^q \biggr )^{1/q} \leq
\sum_{j=-\infty}^\infty \biggl ( \sum_{\ell =-\infty}^\infty \Bigl |2^{-|j|\delta} g_{j+\ell}(x)   \Bigr |^q \biggr )^{1/q} \\
&\leq \sum_{j=-\infty}^\infty 2^{-|j|\delta} \biggl (  \sum_{\ell =-\infty}^\infty \bigl |  g_{j+\ell}(x)   \bigr |^q \biggr )^{1/q} 
= \frac{2^\delta +1}{2^\delta -1} \big\| ( g_j(x) )_{j = -\infty}^\infty \big\|_{\ell^q}.
\end{align*}
If $0< q <1$, then 
\begin{align*}
|h_\ell (x)|^q \leq \sum_{j =-\infty}^\infty 2^{-|j-\ell|\delta q} |g_j(x)|^q.
\end{align*}
Therefore, summing over $\ell$ and changing the order of the summations gives
\begin{align*}
 \big\| ( h_\ell(x) )_{\ell = -\infty}^\infty \big\|_{\ell^q}^q   \leq \sum_{j =-\infty}^{\infty}  |g_j(x)|^q \sum_{\ell =-\infty}^{\infty} 2^{-|j-\ell|\delta q}   = \frac{2^{\delta q} +1}{2^{\delta q} -1}  \big\| ( g_j(x) )_{j = -\infty}^\infty \big\|_{\ell^q}^q, 
\end{align*}
which shows \eqref{Rylq}.

The inequality \eqref{eq:Ry_Lp_lq} follows immediately from \eqref{Rylq}.

It remains to prove \eqref{eq:Ry_lq_Lp}. If $p \geq 1$, then Minkowski's inequality yields
\begin{align*}
\|h_\ell\|_{L^p}  \leq \sum_{j =-\infty}^\infty 2^{-|j-\ell|\delta} \|g_j \|_{L^p},
\end{align*}
so that an argument similar to proving \eqref{Rylq} yields \eqref{eq:Ry_lq_Lp}.
If $0< p <1$, then 
\begin{align*}
\|h_\ell\|_{L^p}^p \leq \sum_{j =-\infty}^\infty 2^{-|j-\ell|\delta p}\|g_j\|_{L^p}^p.    
\end{align*}
Again, arguing similarly as in the proof of \eqref{Rylq} yields
\begin{align*}
\big\|\big(\|h_\ell\|_{L^p}^p\big)_{\ell =-\infty}^\infty \big\|_{\ell^{q/p}} \leq C(p,q,\delta) \big\|\big(\|g_j\|_{L^p}^p\big)_{j =-\infty}^\infty \big\|_{\ell^{q/p}} ,
\end{align*}
which  gives \eqref{eq:Ry_lq_Lp}.
\end{proof}

In addition to the spaces $L^p(\ell^q)$ and $\ell^q(L^p)$ defined above, we will also use certain localized versions of these spaces, where averages are taken over small scales. For defining these spaces, we will use a family of ``dyadic balls'' with properties as in the following lemma. For this, we fix, once and for all, a maximal family  $\big\{ B_{(2 \cqn)^{-1}}(x_{l}) \big\}_{l \in \mathbb{N}}$ of disjoint balls,
and set 
\begin{equation} \label{Bkl}
B^k_l := B_{2^{k}} \big(\delta_{2^{k}}(x_{l}) \big).
\end{equation}
for $k \in \mathbb{Z}$.
The following lemma corresponds to \cite[Lemma 7.14]{FS}.

\begin{lemma}[\cite{FS}] \label{dyba} 
There exist integers $m, m', m''$ such that, for each $k\in \mathbb{Z}$,
\begin{enumerate}
[\rm (i)] \item  $\bigcup\limits_{l \in \mathbb{N}}B_l^k =\N$; 
\item
$\sum\limits_{l \in \mathbb{N}}\chi_{B^k_l}(x) \leq m$ for all $x \in \N$;
\item if $l \in \mathbb{N}$ and $k' \geq k$, there are at most $m'$ values of $l'$ for which $B^{k}_l \cap B^{k'}_{l'} \neq \emptyset$;
\item if $l \in \mathbb{N}$ and $k' < k$, there are at most  $m''2^{(k-k')\hdim}$ values of 
$l'$ for which 
$B^{k}_l \cap B^{k'}_{l'} \neq \emptyset$.
\end{enumerate}
\end{lemma}

\begin{remark}
Strictly speaking, the dyadic balls as defined in \cite{FS} are of the form $B^k_l= B_{(2\gamma)^k}(\delta_{(2\gamma)^k}(x_l))$ rather than \eqref{Bkl}, as used in this paper. However, the dyadic balls defined by \eqref{Bkl} still satisfy all the properties of \cite[Lemma 7.14]{FS} stated as Lemma \ref{dyba} above. For example, given any $x \in \N$, the maximality of the balls $\big\{ B_{(2 \cqn)^{-1}}(x_{l}) \big\}_{l \in \mathbb{N}}$ implies that there exists $l \in \mathbb{N}$ such that the distance from $\delta_{2^{-k}}(x)$ to $B_{(2\gamma)^{-1}}(x_l)$
is less than $(2\gamma)^{-1}$; that is, there exists $y \in B_{(2\gamma)^{-1}}(x_l)$ such that $|y^{-1}\delta_{2^{-k}}(x)| < (2\gamma)^{-1}$. Therefore,
\begin{align*}
\big| \big(\delta_{2^k}(x_l)\big)^{-1} x\big|&=2^k \big| x_l^{-1} \delta_{2^{-k}}(x)\big| \leq 2^k\gamma\left(| (x_l^{-1} y|+ \big|y^{-1}\delta_{2^{-k}}(x)\big|\right) \\
&<2^k \gamma\big( (2\gamma)^{-1} + (2\gamma)^{-1}\big) =2^k,
\end{align*}
which yields that $x \in B^k_l$.  The properties (ii)--(iv) in Lemma \ref{dyba} can also be verified  by a similar argument as in \cite{FS}.
\end{remark}

We will denote the collection of all dyadic balls introduced above  by  
    \[ \mathcal{B} = \big\{ B_l^k: k \in \mathbb{Z}, l \in \mathbb{N} \big\}, \] 
    and set, for each $k \in \mathbb{Z}$, 
\begin{equation*}
\mathcal{B}_k = \{B^{k}_l: l \in \mathbb{N}\}.
\end{equation*}
Note that
$
\mathcal{B}=\cup_{k \in\mathbb{Z}} \; \mathcal{B}_k.$
For $B \in \mathcal{B}$, we denote by $x_B$ and $r_B$ the center and radius of $B$, respectively, and we use $k_B$ to denote the (unique) integer $k$ for which $B \in 
\mathcal{B}_k$.

Following \cite{Rychkov}, for any given sequence $(g_{j})_{j=-\infty}^{\infty}$ of measurable functions $g_j : N \to \mathbb{C}$, we set
\begin{equation*}
\big\|(g_{j})_{j=-\infty}^{\infty}\big\|_{\mathcal{C}_{q}}: = \sup_{k \in \mathbb{Z}, l \in \mathbb{N}} \left( \dashint_{B^k_l}
\sum_{j \geq -k } \big|g_j (x) \big|^q \; d\mu_\N(x)\right)^{1/q}
\end{equation*}
if $q < \infty$, and 
\begin{equation*}
\big\|(g_{j})_{j=-\infty}^{\infty}\big\|_{\mathcal{C}_{\infty}}: =\sup_{k \in \mathbb{Z}, l \in \mathbb{N}} \ \sup_{j \geq -k }  \; \dashint_{{B}^k_l}
 \big|g_j (x)\big| \; d\mu_\N(x),
\end{equation*}
if $q = \infty$. 
The following lemma provides an analogue of Lemma \ref{Ul}  for the spaces $\mathcal{C}_q$.

\begin{lemma} \label{infi}
Let $q \in (0, \infty]$ and $\delta >0$. For a sequence $(g_{j})_{j \in \mathbb{Z}}$ of
measurable functions $g_j : \N \to [0,\infty)$, define $(h^{(\delta)}_{\ell})_{\ell = - \infty}^{\infty}$ as in Equation \eqref{eq:hdl}.

If there exists some constant $C_1 > 0$ such that, for all $j \in \mathbb{Z}$ and all $B \in \mathcal{B}_{-j}$,
\begin{equation} \label{supinf}
\esssup_{x \in B } |g_{j} (x)| \leq C_1 \, \mathop{\mathrm{ess\,inf}}_{x \in B} |g_{j}(x)|,
\end{equation}
then there is a constant $C_2 = C_2 (\delta, q,C_1) > 0$ such that
\begin{align} \label{eq:bounded on Cq}
\big\| (h^{(\delta)}_{\ell})_{\ell = - \infty}^{\infty} \big\|_{\mathcal{C}_{q}} &\leq C_2
\big\|(g_{j})_{j=-\infty}^{\infty}\big\|_{\mathcal{C}_{q}}.
\end{align}
\end{lemma}

\begin{proof}
Throughout, we simply write $h_{\ell} = h_{\ell}^{(\delta)}$.
We split the proof into the cases $q = \infty$ and $q < \infty$.
\\~\\
\textbf{Step 1}. Suppose first that $q =\infty$. Fixing $k \in \mathbb{Z}$ and an arbitrary  dyadic ball $B \in \mathcal{B}_k$, it suffices to show that, for every $\ell \geq -k$, 
\begin{align} \label{desir}
 \dashint_{{B}}|h_{\ell }(x)| \; d\mu_\N(x) \leq_{\delta, C_1}
\big\|(g_{j})_{j=-\infty}^{\infty}\big\|_{\mathcal{C}_{\infty}}.
\end{align}
To this end, letting $\ell \geq -k$, we write
\begin{align*}
\dashint_{{B}}|h_{\ell }(x)| \;d\mu_\N(x) & \leq  \sum_{j < -k }2^{-|j-\ell|\delta} \dashint_{{B}}|g_{j}(x)| \; d\mu_\N(x)
+ \sum_{j \geq -k} 2^{-|j-\ell| \delta} \dashint_{{B}}|g_{j}(x)| \;d\mu_\N(x)\\
&=:I_1 +I_2.
\end{align*}
Then, by the definition of $\mathcal{C}_\infty$, it follows that
\begin{align*}
I_2 \leq \sum_{j \geq -k} 2^{-|j-\ell|\delta} \big\| (g_{m})_{m=-\infty}^{\infty}\big\|_{\mathcal{C}_{\infty}}
\lesssim_\delta
 \big\|(g_{m})_{m=-\infty}^{\infty}\big\|_{\mathcal{C}_{\infty}}.
\end{align*}

Next, for estimating $I_1$, we use Lemma \ref{dyba} (iii) for $j < -k$ (i.e., $-j >k$), to obtain at most $m'$ many balls  $B_1^{(j)},\cdots, B_{\theta_j}^{(j)} \in \mathcal{B}_{-j}$ ($\theta_j \leq m'$)  such that 
$B_i^{(j)}   \cap {B} \neq \emptyset$, $i=1,\cdots, \theta_j$. Since $B\subseteq \bigcup_{i=1}^{\theta_j} B_i^{(j)}$ by Lemma \ref{dyba} (i), it follows by the assumption \eqref{supinf} that
\begin{align*}
\dashint_{{B}}|g_{j}(x)| \; d\mu_\N(x) &\leq  \esssup_{x \in B} |g_j (x)| \leq \max_{1\leq i \leq \theta_j} \esssup_{x \in B_i^{(j)}} |g_j (x)| 
\leq C_1 \max_{1 \leq i \leq \theta_j}  \mathop{\mathrm{ess\,inf}}_{x \in B_i^{(j)}} |g_j (x)|\\
&\leq C_1 \max_{1 \leq i \leq \theta_j} \dashint_{B_i^{(j)}}  |g_j(x)|\; d\mu_\N(x) 
 \leq C_1 \|(g_m)_{m=-\infty}^\infty\|_{\mathcal{C}_\infty}. \numberthis  \label{eq:estimate_sup_integral}
\end{align*}
Hence,
\begin{align*}
I_1 \leq C_1 \big\|(g_m)_{m=-\infty}^\infty\big\|_{\mathcal{C}_\infty}\sum_{j<-k}  2^{-|j-\ell|\delta} 
\lesssim_{\delta} \big\|(g_m)_{m=-\infty}^\infty\big\|_{\mathcal{C}_\infty}.
\end{align*}
Combining the estimates for $I_1$ and $I_2$ proves the result for $q = \infty$.
\\~\\
\textbf{Step 2}.
Suppose that $0< q <\infty$ and fix $k \in \mathbb{Z}$. We start by proving the following claim.
\\~\\
\textbf{Claim:} For any $\varepsilon >0$ such that $\delta - \varepsilon > 0$,  
\begin{align}
\sum_{j \geq -k} |h_j(x)|^q \lesssim_{q, \delta, \varepsilon} \sum_{j < -k}2^{(j+ k)
(\delta -\varepsilon)q} |g_j(x)|^q + \sum_{j \geq -k }|g_j(x)|^q.
\end{align}

To see the claim, we first let $q \geq 1$. In this case, we use the triangle inequality to estimate
\begin{align*}
\bigg(\sum_{j \geq -k} |h_j(x)|^q \bigg)^{1/q} &= \bigg(\sum_{j \geq -k} \bigg|\sum_{\ell =-\infty}^\infty 2^{-|\ell -j |\delta} g_\ell(x) \bigg|^q \bigg)^{1/q} \\
 &\leq  \bigg(\sum_{j \geq -k} \bigg|\sum_{\ell<-k} 2^{-|\ell -j |\delta} g_\ell(x) \bigg|^q \bigg)^{1/q} + \bigg(\sum_{j \geq -k} \bigg|\sum_{\ell \geq -k} 2^{-|\ell -j |\delta} g_\ell(x) \bigg|^q \bigg)^{1/q} \\
 &= : I_1 + I_2.
\end{align*}
By H\"{o}lder's inequality and the fact that $\sum_{\ell < -k}2^{-|\ell -j|\varepsilon q'} 
 <\infty$, we have
\begin{align*}
I_1 &\lesssim_{q, \varepsilon} \bigg(\sum_{j \geq -k} \ \sum_{\ell < -k} 2^{-|\ell -j |(\delta-\varepsilon ) q } |g_\ell(x) |^q \bigg)^{1/q} 
= \bigg( \sum_{\ell < -k} \sum_{j \geq -k} 2^{(\ell -j )(\delta-\varepsilon ) q } |g_\ell(x) |^q \bigg)^{1/q} \\
&= \bigg(\sum_{\ell < - k}  2^{\ell (\delta-\varepsilon ) q } |g_\ell(x) |^q  \sum_{j \geq -k} 2^{ -j (\delta-\varepsilon ) q } \bigg)^{1/q} \lesssim_{q,\delta,\varepsilon} \bigg( \sum_{\ell < -k}  2^{(\ell +k  )(\delta-\varepsilon ) q } |g_\ell(x) |^q \bigg)^{1/q}.
\end{align*}
Similarly, for $I_2$ we  also have
\begin{align*}
I_2 &\lesssim_{q, \varepsilon} \bigg(\sum_{j \geq -k} \ \sum_{\ell \geq -k} 2^{-|\ell -j |(\delta-\varepsilon ) q } |g_\ell(x) |^q \bigg)^{1/q} 
= \bigg( \sum_{\ell \geq -k} \sum_{j \geq -k} 2^{-|\ell -j |(\delta-\varepsilon ) q } |g_\ell(x) |^q \bigg)^{1/q} \\
&\lesssim_{q,\delta,\varepsilon} \bigg( \sum_{\ell \geq -k}   |g_\ell(x) |^q \bigg)^{1/q} .
\end{align*}
It follows that
\begin{align*}
\sum_{j \geq -k} |h_j(x)|^q \leq (I_1 + I_2)^q  
\lesssim_{q,\delta,\varepsilon} \sum_{j < -k}2^{(j+ k)
(\delta -\varepsilon)q} |g_j(x)|^q + \sum_{j \geq -k }|g_j(x)|^q,
\end{align*}
as claimed.

If $0< q <1$, we use the inequality $(\sum_{\ell \in \mathbb{Z}}   |a_\ell|)^q \leq \sum_{\ell \in \mathbb{Z}} |a_\ell|^q$, which directly yields
\begin{align*}
\sum_{j \geq -k} |h_j(x)|^q 
&\leq \sum_{j \geq -k } \ \sum_{\ell =-\infty}^\infty 2^{-|\ell -j |\delta q} |g_\ell(x) |^q\\
&=  \sum_{\ell < -k} \ \sum_{j \geq -k }2^{-|\ell -j |\delta q} |g_\ell(x) |^q + \sum_{\ell \geq -k} \ \sum_{j \geq -k }2^{-|\ell -j |\delta q} |g_\ell(x) |^q\\
&\lesssim_{q,\delta}  \sum_{\ell < -k} 2^{(\ell+k)\delta q} |g_\ell(x) |^q + \sum_{\ell \geq -k} \  |g_\ell(x) |^q .
\end{align*}
Hence, the claim is also true in this case.

We now fix a small $\varepsilon>0$ such that $\delta -\varepsilon >0$. 
From the claim and the obvious fact that $\dashint_{B} \sum_{j \geq -k} |g_j |^q \; d\mu_N \leq \| (g_j)_{j = - \infty}^{\infty} \|_{\mathcal{C}_q}^q$ we see that, in order to prove the assertion of the lemma, it remains to show that\begin{align} \label{eq:remains}
\dashint_B \sum_{j < -k}2^{(j+ k)
\delta'} |g_j(x)|^q \; d\mu_\N (x) \lesssim \big\|(g_{j})_{j=-\infty}^{\infty}\big\|_{\mathcal{C}_{q}}^q 
\end{align}
for all $B \in \mathcal{B}_k$, where $\delta':=(\delta -\varepsilon)q >0$.
As in Step 1, for each integer $j < -k$ (i.e., $-j>k$),  we let $B_1^{(j)},\cdots, B_{\theta_j}^{(j)}$  ($\theta_j \leq m'$) be the balls in  $\mathcal{B}_{-j}$  such that $B_i^{(j)}  \cap {B} \neq \emptyset$, $i=1,\cdots, \theta_j$. Then, by \eqref{supinf}, 
\begin{align*}
\dashint_{{B}}|g_{j}(x)|^q \; d\mu_\N(x) &\leq   \max_{1\leq i \leq \theta_j} \esssup_{x \in B_i^{(j)}} |g_j(x)|^q 
\leq C_1^q \max_{1 \leq i \leq \theta_j}  \mathop{\mathrm{ess\,inf}}_{x \in B_i^{(j)}} |g_j(x)|^q\\
&\lesssim_{q, C_1} \max_{1 \leq i \leq \theta_j} \dashint_{B_i^{(j)}}  |g_j(x)|^q\; d\mu_\N(x) 
 \leq  \big\|(g_m)_{m=-\infty}^\infty\big\|_{\mathcal{C}_q}^q,
\end{align*}
which yields that
\begin{align*}
\dashint_B \sum_{j < -k}2^{(j+ k)
\delta'} |g_j(x)|^q \; d\mu_\N (x) \lesssim \sum_{j < -k}2^{(j+ k)
\delta'} \big\|(g_m)_{m=-\infty}^\infty\big\|_{\mathcal{C}_q}^q \lesssim \big\|(g_m)_{m=-\infty}^\infty\big\|_{\mathcal{C}_q}^q,
\end{align*}
as desired.
\end{proof}

Lastly, we prove the following lemma on the spaces $\mathcal{C}_q$. Parts of the proof are inspired by arguments for Euclidean spaces that can be found in \cite{Bui}.

\begin{lemma} \label{intt}
Let $q \in (0, \infty]$ and $a >0$. Suppose  $\frac{Q}{a} <r < q$ (resp. $\frac{Q}{a} < r <1$) if $q <\infty$ (resp. $q =\infty$).
\begin{enumerate}
    [\rm (i)]\item 
 For a sequence $(g_{j})_{j=-\infty}^{\infty}$
of measurable functions $g_j : \N \to \mathbb{C}$, define
\begin{equation*}
g_{j}^{\ast}(x) = \left( \int_\N \frac{2^{j\hdim}|g_{j}(z)|^r}{(1+2^{j}|z^{-1}x|)^{ar}} \; d\mu_\N(z)\right)^{1/r}, \quad x\in \N.
\end{equation*}
Then there exists $C > 0$ such that
\begin{equation} \label{225}
\big\|(g_{j}^\ast)_{j=-\infty}^{\infty}\big\|_{\mathcal{C}_{q}} \leq C\big\|(g_{j})_{j=-\infty}^{\infty}\big\|_{\mathcal{C}_{q}}.
\end{equation}

\item For a measurable function $H:\N \times (0, \infty) \to \mathbb{C}$, define
\begin{align*}
  h_j(x) =& \bigg( \int_{2^{-j}}^{2^{-j+1}} \hspace{-10pt} |H (x, t)|^r \, \frac{d t}{t} \bigg)^{1/r}, \\
    h'_j(x) =& \bigg ( \int_{2^{-j}}^{2^{-j+1}} \hspace{-7pt} \int_{\N} \frac{\tau^{-\hdim} |H( z,t)|^r}{(1+t^{-1}|z^{-1}x|)^{ar}} d\mu_\N (z) \, \frac{d t}{t} \bigg)^{1/r}
\end{align*}
for $j \in \Z, x\in \N$. Then there exists $C > 0$ such that
\begin{align*}
    \big\|(h'_j)_{j=-\infty}^{\infty}\big\|_{\mathcal{C}_{q}} \leq C \big\|(h_{j})_{j=-\infty}^{\infty}\big\|_{\mathcal{C}_{q}}.
\end{align*}
\end{enumerate}
\end{lemma}
\begin{proof}
(i) 
We first consider the case $0 < q <\infty$.
For this, fix $k \in \mathbb{Z}$ and let $B \in \mathcal{B}_k$ be arbitrary.  
We decompose each function $g_j$ as the sum $g_j= u_j + v_j$ of the functions $u_{j}: =g_{j} \cdot \chi_{B(x_B, 3\gamma r_B)}$ 
and $v_{j}:=g_{j}-u_{j}$, where $x_B$ and $r_B$ denote the center and radius of the ball $B$, respectively.
Since $g_j^* \lesssim u_j^* + v_j^*$ for $j \in \mathbb{Z}$, it suffices to prove the required estimates for $u_j$ and $v_j$.

For obtaining the estimate for $u^*_{j}$, we use the majorant property (Lemma \ref{lem:majorant}) and the Fefferman-Stein inequalities (Lemma \ref{vector-valued HL}) for the Hardy-Littlewood maximal operator $\mathcal{M}_r$ to estimate 
\begin{align*}  
\int_{B} \sum_{j \geq -k} [u_{j}^{\ast}(x)]^{q} \; d\mu_\N(x) &\lesssim \int_{B} \sum_{j \geq -k}
[(\mathcal{M}_{r}u_{j})(x) ]^{q} \;d\mu_\N(x)  \leq \big\|(\mathcal{M}_{r}u_{j})_{j= -k}^\infty \big\|_{L^{q}(\ell^{q})}^{q} \\
& \lesssim \big\|(u_{j})_{j= -k}^\infty \big\|_{L^{q}(\ell^{q})}^{q}
=\int_{B(x_B, 3\gamma r_B)}
\sum_{j \geq -k} |u_{j}(x)|^{q} \; d\mu_\N(x). \numberthis \label{eq:ujstar}
\end{align*}
By Lemma \ref{dyba},  there exists a positive integer $m'''$ such that $B(x_B, 3\gamma r_B)$ is covered by at most $m'''$ many balls  $B_1, \cdots, B_\theta \in \mathcal{B}_k$ $(\theta\leq m''')$.  Since $m'''$ is independent of $k$ and the choice of $B \in \mathcal{B}_k$, we can further estimate \eqref{eq:ujstar} by 
\begin{equation} \label{226}
\dashint_{B} \sum_{j \geq -k} [u_{j}^{\ast}(x)]^{q} \; d\mu_\N(x) 
\lesssim   \sum_{i=1}^\theta \dashint_{B_i}
\sum_{j\geq -k} |u_{j}(x)|^{q} \; d\mu_\N(x) \lesssim  \big\|(g_{j})_{j=-\infty}^{\infty}\big\|_{\mathcal{C}_{q}}^q,
\end{equation}
which is the required estimate for $u^*_j$.

For obtaining the estimate for $v_{j}^\ast$, we first use H\"{o}lder's inequality, together with the fact that $(1 + |\cdot|)^{-ar} \in L^1(\N)$, to obtain that, for $x \in \N \setminus B(x_B, 3\gamma r_B)$,                              
\begin{align*}
[v_{j}^{\ast}(x)]^{q} &\leq \int_{\N \backslash B(x_B, 3\gamma r_B)} \frac{2^{j\hdim}|v_{j}(z)|^{q} }{(1+2^{j}|z^{-1}x|)^{ar}} \; d\mu_\N(z) \left(  \int_{\N } 
\frac{2^{j \hdim}}{ (1+2^{j}|z^{-1}x|)^{ar}}\; d\mu_\N(z)\right)^{\frac{q/r}{(q/r)'}}\\
&\lesssim \sum_{\ell =1}^\infty 
\int_{B(x_B, 3\gamma 2^{\ell} r_B) \backslash B(x_B, 3\gamma 2^{\ell - 1} r_B)} \frac{2^{j\hdim}|v_{j}(z)|^{q} }{(1+2^{j}|z^{-1}x|)^{ar}} \; d\mu_\N(z)\\
&\lesssim 
\sum_{\ell =1}^\infty  2^{-(j+k+\ell)ar}2^{j\hdim}
\int_{B(x_B, 3\gamma 2^{\ell} r_B)} |v_{j}(z)|^{q} \; d\mu_\N(z). \numberthis \label{itgr}
\end{align*}
For each $\ell \in \mathbb{N}$, let $B_1^{(\ell)}, \cdots, B_{\omega_\ell}^{(\ell)}$ be the balls in $\mathcal{B}_{k+\ell}$ such that
\begin{align} \label{eq:ballomega}
B_i^{(\ell)} \cap B(x_B, 3\gamma 2^{\ell} r_B)  \neq \varnothing, \quad  i =1,\cdots, \omega_\ell.
\end{align}
By Lemma \ref{dyba}, there exists an integer 
$\widetilde{m}$ such that $\omega_\ell \leq \widetilde{m}$ for all $\ell \in \mathbb{N}$. Using this fact, we deduce from \eqref{itgr}  that
\begin{align*} 
\dashint_{B} \sum_{j\geq -k} [v_{j}^{\ast}(x)]^{q} \; d\mu_\N(x) &\lesssim   \sum_{\ell =1}^\infty \sum_{i=1}^{\omega_\ell }\sum_{j \geq -k} 2^{-(j+k+\ell)ar}2^{j\hdim} \int_{B^{(\ell)}_i} |v_{j}(z)|^{q} \; d\mu_\N(z)\\
&\asymp \sum_{\ell =1}^\infty \sum_{i=1}^{\omega_\ell} \sum_{j \geq -k} 2^{-(j+k+\ell)(ar-\hdim)} \dashint_{B^{(\ell)}_i} |v_{j}(z)|^{q} \; d\mu_\N(z)\\
&\lesssim \sum_{\ell =1}^\infty \sum_{i=1}^{\omega_\ell} 2^{-\ell (ar -\hdim)}\sum_{j \geq - (k + \ell)} \dashint_{B^{(\ell)}_i} |v_{j}(z)|^{q} \; d\mu_\N(z)\\
& \lesssim \sum_{\ell =1}^\infty \sum_{i=1}^{\omega_\ell} 2^{-\ell (ar -\hdim)} \big\| (v_j)_{j=-\infty}^\infty\big\|_{\mathcal{C}_q}^q \\
& \lesssim  \big\| (g_j)_{j=-\infty}^\infty\big\|_{\mathcal{C}_q}^q, \numberthis \label{228}
\end{align*}
which finishes the proof for the case $q<\infty$.

We next prove the assertion in the case $q =\infty$, which is based on similar ideas as the case just proven. Fix an arbitrary $k \in \mathbb{Z}$ and take an arbitrary $B \in \mathcal{B}_k$. 
Let ${u}_j$ and $v_j$ be as before.  Since $0< r <1$, the operator $\mathcal{M}_r$ is bounded on $L^1(\N)$. Hence, for every $j \geq -k$,
\begin{align} \label{eq:ujstarq=infty}
\dashint_{B} {u}_{j}^{\ast}(x) \; d\mu_\N(x) \lesssim \dashint_{B} 
(\mathcal{M}_{r}{u}_{j})(x)  \;d\mu_\N(x)  \lesssim\frac{1}{\mu_\N (B)} \int_{B(x_B, 3\gamma 2^{\ell} r_B)} |g_j(x)|\; d\mu_\N (x).
\end{align}
Similarly to the argument showing 
\eqref{226}, the ball $B(x_B, 3\gamma  r_B)$ can be covered by finitely many dyadic balls from $\mathcal{B}_k$, which yields that
\[ \dashint_{B} {u}_{j}^{\ast}(x) \; d\mu_\N(x) \lesssim \big\| (g_j)_{j =-\infty}^\infty\big\|_{\mathcal{C}_\infty}. \]

Lastly, for each $v_j$ with $j \geq -k$,  we use the H\"{o}lder's inequality to write 
\begin{align*}
v_j^\ast (x) 
&\leq  \left( \int_\N \frac{2^{j\hdim}|g_{j}(z)|}{(1+2^{j}|z^{-1}x|)^{ar}} \; d\mu_\N(z)\right)
\left( \int_\N \frac{2^{j\hdim}}{(1+2^{j}|z^{-1}x|)^{ar}} \; d\mu_\N(z)\right)^{(1-r)/r} \\
& \lesssim  \int_{\N\backslash B(x_B, 3\gamma  r_B)} \frac{2^{j\hdim}|g_{j}(z)|}{(1+2^{j}|z^{-1}x|)^{ar}} \; d\mu_\N(z).
\end{align*}
This estimate, together with a similar argument as showing \eqref{228}, yields 
\begin{align*}
\dashint_B v_j^\ast (x)d\mu_\N (x)& \lesssim \sum_{\ell =1}^\infty \sum_{i =1}^{\omega_\ell} 2^{-(j+k+\ell)ar}2^{j\hdim} \int_{B^{(\ell)}_i} |v_{j}(z)| \; d\mu_\N(z)\\
& \lesssim  \big\| (g_j)_{j=-\infty}^\infty\big\|_{\mathcal{C}_\infty},
\end{align*}
where $B^{(\ell)}_i, \cdots B^{(\ell)}_{\omega_\ell}$ are the balls from $\mathcal{B}_{k+\ell}$ which satisfy \eqref{eq:ballomega}.
In combination, this proves \eqref{225} for the case $q=\infty.$
\\~\\
(ii) The proof of assertion (ii) follows from the same proof as showing (i), due to an application of Fubini's theorem and
\begin{align*}
    (1+2^{j}|z^{-1}x|)^{ar} \asymp (1+t^{-1}|z^{-1}x|)^{ar}
\end{align*}
for all $j \in \Z$, $t \in [2^{-j}, 2^{-j+1}]$ and $x, z \in \N$, with an implicit constant independent of these variables.
\end{proof}

\section{Besov and Triebel-Lizorkin spaces} \label{sec:BTL}
This section is devoted to the basic theory of Besov and Triebel-Lizorkin spaces on homogeneous groups. The definition of these spaces is as follows.

\begin{definition}
 Let $\phi \in \SV(N)$ satisfy the discrete Calder\'{o}n condition \eqref{eq:discrete_calderon}.

 \begin{enumerate}[(1)]
    \item For $p \in (0,\infty)$, $q \in (0, \infty]$ and $\ord \in \mathbb{R}$, the \emph{Triebel-Lizorkin space} $\TLS(\N)$ associated to $\crk$ is defined as the space of all $f \in \mathcal{S}_0'(N)$ satisfying
 \[
  \| f \|_{\TLS} := \bigg\| \bigg( \sum_{j \in \mathbb{Z}} 2^{j\sigma q} | f \ast \crk_{2^{-j}} |^q \bigg)^{1/q} \bigg\|_{L^p} < \infty
 \]
and equipped with the quasi-norm $\| \cdot \|_{\TLS}$, with the usual modification when $q =\infty$.
    \item For $q \in (0, \infty]$ and $\ord \in \mathbb{R}$, the
\emph{Triebel-Lizorkin space} $\F_{\infty,q}^{\ord}(\N)$ associated to $\crk$ is defined as the space of all $f \in \mathcal{S}_0'(N)$ satifying
\begin{equation*}
\|f\|_{\F^{\ord}_{\infty,q}} : =\left\|\big( 2^{j\ord}f \ast \crk_{2^{- j}}\big)_{j=-\infty}^\infty\right\|_{\mathcal{C}_q} < \infty
\end{equation*}
and equipped with the quasi-norm $\| \cdot \|_{\F^{\ord}_{\infty,q}}$.
\item For $p, q \in (0,\infty]$ and $\ord \in \mathbb{R}$, the \emph{Besov space} $\BS(\N)$ associated to $\crk$ is defined as the space of all $f \in \mathcal{S}_0'(N)$ satisfying
 \[
  \| f \|_{\BS} :=  \bigg( \sum_{j \in \mathbb{Z}} 2^{j\ord q} \| f \ast \crk_{2^{-j}}\|_{L^p}^q \bigg)^{1/q} < \infty
 \]
and equipped with the quasi-norm $\| \cdot \|_{\BS}$, with the usual modification when $q =\infty$.
\end{enumerate}
\end{definition} 

The spaces and quasi-norms just defined are clearly well-posed. We prove various basic properties and characterizations in the next sections.

\subsection{Independence results} \label{subsection:independence}
We start by showing that the Besov and Triebel-Lizorkin spaces are
independent of the choice of the function satisfying the Calder\'on condition, with equivalent (quasi-)norms for different choices. 

\begin{theorem} \label{thm:indep_crk}
Suppose that $\phi, \eta \in \mathcal{S}_0(\N)$ satisfy the discrete Calder\'{o}n condition \eqref{eq:discrete_calderon}.
Then the following statements are true.

\begin{enumerate}
    [\rm (i)]\item  If $p \in (0,\infty)$, $q \in (0, \infty]$ and $\ord \in \mathbb{R}$, then 
\begin{align*}
    \bigg\| \bigg( \sum_{j \in \mathbb{Z}} 2^{j\sigma q} | f \ast\crk_{2^{-j}} |^q \bigg)^{1/q} \bigg\|_{L^p}
    \asymp
    \bigg\| \bigg( \sum_{j \in \mathbb{Z}} 2^{j\sigma q} | f \ast\eta_{2^{-j}} |^q \bigg)^{1/q} \bigg\|_{L^p}
\end{align*}
for all $f \in \SV'(N)$, with the usual modification when $q =\infty$.
    \item  If  $q \in (0, \infty]$ and $\ord \in \mathbb{R}$, then 
\begin{align*}
    \left\|\big( 2^{j\ord}f \ast \crk_{2^{- j}}\big)_{j=-\infty}^\infty\right\|_{\mathcal{C}_q} 
    \asymp
    \left\|\big( 2^{j\ord}f \ast \eta_{2^{- j}}\big)_{j=-\infty}^\infty\right\|_{\mathcal{C}_q}
\end{align*}
for all $f \in \SV'(N)$.
    \item  If $p, q \in (0,\infty]$ and $\ord \in \mathbb{R}$, then 
\begin{align*}
    \bigg( \sum_{j \in \mathbb{Z}} 2^{j\ord q} \| f \ast \crk_{2^{-j}}\|_{L^p}^q \bigg)^{1/q} 
    \asymp
    \bigg( \sum_{j \in \mathbb{Z}} 2^{j\ord q} \| f \ast \eta_{2^{-j}}\|_{L^p}^q \bigg)^{1/q}
\end{align*}
for all $f \in \SV'(N)$, with the usual modification when $q =\infty$.
\end{enumerate}

\end{theorem}
\begin{proof}
We prove the cases (i), (ii) and (iii) in parallel.  By symmetry, it suffices to show only one of the claimed inequalities. Our argument is divided into two steps. Throughout the proof, we fix $p, q \in (0, \infty]$, $\ord \in \mathbb{R}$ and $f \in \mathcal{S}'_0 (N)$. Recall that the Peetre-type maximal function $(\crk_t^* f)_\PTpar$ of $f$ is defined by
\[
(\crk_t^* f)_\PTpar(x) =
	\sup_{y \in \N} \frac{|f \ast \crk_t(y)| }{(1+t^{-1}|y^{-1}x|)^\PTpar},\quad x \in \N,
\]
for $a, t>0$.
\\~\\
\textbf{Step~1.} Throughout this step, we fix $a > 0$. 
Since $\crkk$ satisfies the discrete Calder\'{o}n condition, there exists some $\drkk \in \SV(\N)$ such that  
\begin{align*}
    f = \sum_{j \in \mathbb{Z}} f \ast \crkk_{2^{-j}} \ast \drkk_{2^{-j}},
\end{align*}
with convergence in $\SV'(\N)$. Hence, for $k \in \Z$,
\begin{align*}
    f \ast \crk_{2^{-k}} = \sum_{j \in \mathbb{Z}} f \ast \crkk_{2^{-j}} \ast \drkk_{2^{-j}} \ast \crk_{2^{-k}}
\end{align*}
with convergence in $\SV'(\N)$. Each of the above convolution products defines a smooth function on $\N$, so that the identity
\begin{align*}
    f \ast \crk_{2^{-k}}(x) &= \sum_{j \in \mathbb{Z}} \int_{\N} 
   \big( f \ast \crkk_{2^{-j}} \big) (y) \big( \drkk_{2^{-j}} \ast \crk_{2^{-k}} \big) (y^{-1}x) \; d\mu_\N (y)
\end{align*}
holds pointwise for all $x \in \N$. For arbitrary $M, L >0$, an application of Lemma~\ref{aoe} allows us to estimate
\begin{align*}
    2^{k \ord} |f \ast \crk_{2^{-k}}(x)| &\lesssim 2^{k \ord} \sum_{j \in \mathbb{Z}} \int_{\N} 
    |f \ast \crkk_{2^{-j}}(y)| |\drkk_{2^{-j}} \ast \crk_{2^{-k}}(y^{-1}x)| \; d\mu_\N (y) \\
    & \lesssim \sum_{j \in \mathbb{Z}} 2^{-|j -k |M} 2^{k \ord} \int_{\N} |f \ast \crkk_{2^{-j}}(y)| \frac{2^{(j \wedge k)\hdim}}{(1 + 2^{j\wedge k}|y^{-1}x|)^{L}} \; d\mu_\N (y)
\end{align*} 
for arbitrary $k \in \Z$ and $x \in \N$. For estimating the integral in the right-hand side further, we fix $\varepsilon > 0$ and $M > a + |\ord|$ and choose $L = a + \hdim + \varepsilon$. Then, using the simple $(1+2^j|y^{-1}x|)^a \leq 2^{|j-k|a} (1 +2^{j\wedge k}|y^{-1}x|)^a$ and the definition of the Peetre-type maximal function, it follows that
\begin{align*}
    |f \ast \crkk_{2^{-j}}(y)| \frac{2^{(j \wedge k)\hdim}}{(1 + 2^{j\wedge k}|y^{-1}x|)^{L}} 
    & \leq  2^{|j - k|a} \frac{|f \ast \crkk_{2^{-j}}(y)|}{(1 + 2^j|y^{-1}x|)^a} \frac{2^{(j \wedge k)\hdim}}{(1 + 2^{j\wedge k}|y^{-1}x|)^{\hdim + \varepsilon}} \\
    &\leq 2^{|j - k|a} (\crkk^*_{2^{-j}} f)_{a}(x) \frac{2^{(j \wedge k)\hdim}}{(1 + 2^{j\wedge k}|y^{-1}x|)^{\hdim + \varepsilon}},
\end{align*}
so that a change of variable gives 
\begin{align}
\begin{split} \label{eq:ptw_est_Step_1_indep_crk}
    2^{k \ord} |f \ast \crk_{2^{-k}}(x)| 
    &\lesssim \sum_{j \in \mathbb{Z}} 2^{-|j -k |(M-a-|\ord|)}2^{j\ord} (\crkk^*_{2^{-j}} f)_{a}(x) \int_{\N} \frac{2^{(j \wedge k)\hdim}}{(1 + 2^{j\wedge k}|y^{-1}x|)^{Q + \varepsilon}} \; d\mu_\N (y)\\ 
    &\lesssim \sum_{j \in \mathbb{Z}} 2^{-|j -k |(M-a-|\ord|)} 2^{j\ord}(\crkk^*_{2^{-j}} f)_{a}(x).
\end{split}
\end{align}
We next consider the cases (i), (ii) and (iii) separately.
\\~\\
\textit{Case (i).} If $p<\infty$, then taking the $L^p(\ell^q)$-norm in \eqref{eq:ptw_est_Step_1_indep_crk} and an application of \eqref{eq:Ry_lq_Lp} yields
\[
\Bigl \| \bigl ( 2^{j \ord} f \ast \crk_{2^{-j}} \big)_{j = -\infty}^\infty \Bigr \|_{L^p(\ell^q)} \lesssim \Bigl \| ( 2^{j \ord} (\crkk^*_{2^{-j}} f)_{a} )_{j = -\infty}^\infty \Bigr \|_{L^p(\ell^q)}, \]
as desired.
\\~\\
\textit{Case (ii).} For $q \leq \infty$, we aim to apply Lemma~\ref{infi} to the functions $(2^{j \ord} (\crkk^*_{2^{-j}} f)_{a})_{j = - \infty}^{\infty}$. For this, first note that the hypothesis \eqref{supinf} from Lemma~\ref{infi} is satisfied since
for all $j \in \mathbb{Z}$ and all $B \in \mathcal{B}_{-j}$, whenever $x, x' \in B$, we have
\begin{align}
\begin{split} \label{eq:justification_supinf}
(\crkk^*_{2^{-j}} f)_{a}(x) & = \sup_{y \in \N} \frac{|f \ast \crkk_{2^{-j}}(y)| }{(1+2^j|y^{-1}x'|)^a}\frac{(1+2^j|y^{-1}x'|)^a}{(1+2^j|y^{-1}x|)^a} \\
    &\leq (2\gamma)^a \sup_{y \in \N} \frac{|f \ast \crkk_{2^{-j}}(y)| }{(1+2^j|y^{-1}x'|)^a} =  (2\gamma)^a  (\crkk^*_{2^{-j}} f)_{a}(x'),
\end{split}
\end{align}
 where $\gamma > 0$ is the constant in the quasi-norm \eqref{quas}.
This implies that 
\eqref{supinf} holds.
Therefore, taking the $\mathcal{C}_q$-norm in \eqref{eq:ptw_est_Step_1_indep_crk} and an application of Lemma~\ref{infi} yields
\[
\Bigl \| \bigl ( 2^{j \ord} f \ast \crk_{2^{-j}} \big)_{j = -\infty}^\infty \Bigr \|_{\mathcal{C}_{q}} \lesssim \Bigl \| ( 2^{j \ord} (\crkk^*_{2^{-j}} f)_{a} )_{j = -\infty}^\infty \Bigr \|_{\mathcal{C}_{q}},
\]
which proves the desired estimate.
\\~\\
\textit{Case (iii).} The proof is the same as in case (i). Taking the $\ell^q(L^p)$-norm of \eqref{eq:ptw_est_Step_1_indep_crk}, together with the estimate \eqref{eq:Ry_Lp_lq}, gives
\[
\Bigl \| \bigl ( 2^{j \ord} f \ast \crk_{2^{-j}} \bigr )_{j = -\infty}^\infty \Bigr \|_{\ell^q(L^p)} \lesssim \Bigl \| ( 2^{j \ord} (\crkk^*_{2^{-j}} f)_{a} )_{j = -\infty}^\infty \Bigr \|_{\ell^q(L^p)}.
\]
This completes the proof of Step 1.
\\~\\
\textbf{Step~2.} 
Using the sub-mean-value property (Lemma \ref{lem:central_estimate}), it follows that, for arbitrary $a, r > 0$, there exists $C = C_{r, a} >0$ such that 
\begin{align*}
|f \ast \eta_{2^{-j}}(y)|^r 
\leq C \sum_{k \in \mathbb{Z}}
2^{-|j-k|ar} \int_{\N} \frac{2^{k\hdim}|f \ast \eta_{2^{-k}}(z)|^r}{(1+2^{k}|z^{-1}y|)^{ar}} \; d\mu_\N(z)
\end{align*}
for all $y \in \N$.
Multiplying both sides by $2^{j\ord r} \, (1+2^j|y^{-1}x|)^{-ar}$ for $x \in \N$ and using the basic inequality 
\begin{align} \label{eq:fund_inequality_weights}
 \frac{1}{(1+2^k|z^{-1}y|)^{ar}(1 +2^j|y^{-1}x|)^{ar}} \lesssim \frac{2^{|j-k|ar}}{(1+2^k|z^{-1}x|)^{ar}}   
\end{align}
gives
\begin{align*}
   2^{j\ord r} \frac{|f \ast \eta_{2^{-j}}(y)|^r}{(1+2^j|y^{-1}x|)^{ar}} 
    &\lesssim \sum_{k \in \mathbb{Z}}
2^{-|j-k|ar} 2^{j\ord r} \int_{\N} \frac{2^{k\hdim}|f \ast \eta_{2^{-k}}(z)|^r}{(1+2^{k}|z^{-1}z|)^{ar}} \; d\mu_\N(z) \\
   &\lesssim \sum_{k \in \mathbb{Z}}
2^{-|j-k|(ar-|\ord|r)} 2^{k\ord r} \int_{\N} \frac{2^{k\hdim}|f \ast \eta_{2^{-k}}(z)|^r}{(1+2^{k}|z^{-1}x|)^{ar}} \; d\mu_\N(z),
\end{align*}
which easily implies that
\begin{align}
    2^{j\ord} (\crkk^*_{2^{-j}} f)_{a}(x)
    &\lesssim \biggl ( \sum_{k \in \mathbb{Z}}
    2^{-|j-k|(ar-|\ord|r)} 2^{k\ord r} \int_{\N} \frac{2^{k\hdim}|f \ast \eta_{2^{-k}}(z)|^r}{(1+2^{k}|z^{-1}x|)^{ar}} \; d\mu_\N(z) \biggr )^{1/r} \label{eq:Mf_p=infty}
\end{align}
for any $x \in \N$.
\\~\\
\textit{Case (i) and (iii)}. Assume that $r > \hdim/a$ and $a > |\ord|$. Then an application of Lemma \ref{lem:majorant} yields
\begin{align*}
    2^{j\ord} (\crkk^*_{2^{-j}} f)_{a}(x) &\lesssim \biggl ( \sum_{k \in \mathbb{Z}} 2^{-|j-k|(ar-|\ord|r)} 2^{k \sigma r} \mathcal{M}(| f \ast \eta_{2^{-k}}|^r )(x) \biggr )^{1/r} \\
    &= \biggl ( \sum_{k \in \mathbb{Z}} 2^{-|j-k|(ar-|\ord|r)} \Bigl [ \mathcal{M}_r \bigl ( 2^{k\ord} |f \ast \eta_{2^{-k}}| \bigr )(x) \Bigr ]^r \biggr )^{1/r} \numberthis  \label{eq:Mf_p<infty}
\end{align*}
for all $x \in \N$ and $j \in \Z$. If $q < \infty$, we take the $L^p(\ell^q)$-quasi-norm of both sides of \eqref{eq:Mf_p<infty} to obtain
\begin{align*}
\bigg\| \bigg( \sum_{j \in \mathbb{Z}} 2^{j\sigma q} \bigl [& (\crkk^*_{2^{-j}} f)_{a} \bigr ]^q \bigg)^{1/q} \bigg\|_{L^p} \\
    &\lesssim
\bigg\| \bigg( \sum_{j \in \mathbb{Z}} \biggl ( \sum_{k \in \mathbb{Z}} 2^{-|j-k|(ar-|\ord|r)} \Bigl [ \mathcal{M}_r \bigl ( 2^{k\ord} |f \ast \eta_{2^{-k}}| \bigr ) \Bigr ]^r \biggr )^{q/r} \bigg)^{1/q} \bigg\|_{L^p} \\
& \lesssim
\bigg\| \bigg( \sum_{k \in \mathbb{Z}} \mathcal{M}_r \Bigl ( 2^{k\ord} |f \ast \eta_{2^{-k}}| \Bigr )^q \bigg)^{1/q} \bigg\|_{L^p} \\
& \lesssim
\bigg\| \bigg( \sum_{k \in \mathbb{Z}}   2^{k\ord q} |f \ast \eta_{2^{-k}}|^q \bigg)^{1/q} \bigg\|_{L^p},
\end{align*}
where the third and fourth inequality used the inequality \eqref{eq:Ry_Lp_lq} and Lemma~\ref{vector-valued HL}, respectively. Note that Lemma~\ref{vector-valued HL} is applicable provided that $r < p \wedge q$, which is satisfied precisely for $a > \frac{ar}{q \wedge q} > \frac{\hdim}{p \wedge q}$. The case $q = \infty$ follows similarly by choosing $r = 1 \wedge p$. This shows (i).

The case (iii) follows by similar arguments, taking the $\ell^q(L^p)$-norm in Equation \eqref{eq:Mf_p<infty} and applying Lemma~\ref{vector-valued HL}, which in this case only requires $r < p$.
\\~\\
\textit{Case (ii).}
Assume that $a>|\ord|$. Taking the $\mathcal{C}_q$-quasi-norms on both sides of \eqref{eq:Mf_p=infty} yields
\begin{align*}
    \Bigl \| \Bigl ( 2^{j\ord} (\crkk^*_{2^{-j}} f)_{a} \Bigr )_{j \in \Z} \Bigr \|_{\mathcal{C}_q}
    &\lesssim \biggl \| \biggl ( \biggl ( \sum_{k \in \mathbb{Z}}
    2^{-|j-k|(ar-|\ord|r)} 2^{k\ord r} \int_{\N} \frac{2^{k\hdim}|f \ast \eta_{2^{-k}}(z)|^r}{(1+2^{k}|z^{-1}x|)^{ar}} d\mu_\N(z) \biggr )^{1/r} \biggr )_{j \in \Z} \biggr \|_{\mathcal{C}_q} \\ 
    &\lesssim \biggl \| \biggl ( 2^{k\ord} \biggl ( \int_{\N} \frac{2^{k\hdim}|f \ast \eta_{2^{-k}}(z)|^r}{(1+2^{k}|z^{-1}x|)^{ar}} \; d\mu_\N(z) \biggr )^{1/r} \biggr )_{k \in \Z} \biggr \|_{\mathcal{C}_q},
\end{align*}
where the last step used \eqref{eq:bounded on Cq} of Lemma~\ref{infi}, whose hypothesis is satisfied 
due to an estimate that is almost identical to
\eqref{eq:justification_supinf} from Step~1. Therefore, the desired estimate \[
\Bigl \| ( 2^{j \ord} (\crkk^*_{2^{-j}} f)_{a} )_{j = -\infty}^\infty \Bigr \|_{\mathcal{C}_{q}} \lesssim \Bigl \| \bigl ( 2^{j \ord} f \ast \crkk_{2^{-j}} \big)_{j = -\infty}^\infty \Bigr \|_{\mathcal{C}_{q}}, \label{eq:estimate_PM_2_q}
\] follows from an application of Lemma~\ref{intt} if we additionally assume that $a > \frac{ar}{q} > \frac{\hdim}{q}$ when $q < \infty$ and that $a > \hdim$ when $q = \infty$.
\end{proof}

\medskip

We next state two consequences of (the proof of) Theorem \ref{thm:indep_crk} that will be used in the remainder. The first consequence is the following characterization of Besov and Triebel-Lizorkin spaces in terms of Peetre-type maximal functions.

\begin{corollary} \label{cor:PM_qn_equiv}
Suppose that $\crk \in \mathcal{S}_0(\N)$ satisfies the discrete Calder\'{o}n condition~\eqref{eq:discrete_calderon}.
Then the following statements hold. 

\begin{enumerate}
[\rm (i)]\item  If $p \in (0,\infty)$, $q \in (0, \infty]$, $\ord \in \mathbb{R}$ and  $a > \max \{\frac{\hdim}{p \land q}, |\ord| \}$, then 
\begin{align*}
\bigg\| \bigg( \sum_{j \in \mathbb{Z}} 2^{j\sigma q} | f \ast\crk_{2^{-j}} |^q \bigg)^{1/q} \bigg\|_{L^p}
\asymp
\bigg\| \bigg( \sum_{j \in \mathbb{Z}} 2^{j\sigma q} | (\crk^*_{2^{-j}} f)_{a} |^q \bigg)^{1/q} \bigg\|_{L^p}
    \end{align*}
for all $f \in \mathcal{S}_0'(\N)$, with the usual modification when $q =\infty$.
\item  If  $q \in (0, \infty]$, $\ord \in \mathbb{R}$ and $a > \max \{\frac{\hdim}{q}, |\ord| \}$, then 
\begin{align*}
\biggl \| \Bigl ( 2^{j\ord}f \ast \crk_{2^{- j}} \Bigr)_{j=-\infty}^\infty \biggr\|_{\mathcal{C}_q} 
\asymp
\biggl \| \Bigl ( 2^{j\ord} (\crk^*_{2^{-j}} f)_{a} \Bigr)_{j=-\infty}^\infty \biggr \|_{\mathcal{C}_q}
\end{align*}
for all $f \in \mathcal{S}_0'(\N)$.
\item  If $p, q \in (0,\infty]$, $\ord \in \mathbb{R}$ and $a > \max \{\frac{\hdim}{p}, |\ord| \}$, then
\begin{align*}
\bigg( \sum_{j \in \mathbb{Z}} 2^{j\ord q} \| f * \crk_{2^{-j}}\|_{L^p}^q \bigg)^{1/q} 
\asymp
\bigg( \sum_{j \in \mathbb{Z}} 2^{j\ord q} \| (\crk^*_{2^{-j}} f)_{a} \|_{L^p}^q \bigg)^{1/q}
\end{align*}
for all $f \in \mathcal{S}_0'(\N)$, with the usual modification when $q =\infty$.
\end{enumerate}
\end{corollary}

The second consequence of Theorem~\ref{thm:indep_crk} is a characterization using spectral multipliers.

\begin{corollary} \label{cor:spectral_multiplier}
Let $\N$ be a homogeneous Lie group and let
$\P$ be the positive, self-adjoint convolution operator of homogeneous degree $1$ defined by \eqref{eq:def_P}.
Let $m \in \mathcal{S}(\mathbb{R}^+)$ be such that
\begin{align*} %\label{}
\supp m &\subseteq [ (1/2), 2],\\
|m(\lambda)| \geq c >0 \ \ &\mbox{for} \ \ (3/5) \leq \lambda \leq (5/3). 
\end{align*}
for a fixed constant $c > 0$.
\begin{enumerate}
[\rm (i)]\item If $p \in (0,\infty)$, $q \in (0,\infty]$ and $\ord \in \R$, then for all $f \in \mathcal{S}_0'(N)$,
\begin{align*}
    \| f \|_{\TLS} \asymp \biggl \|\Bigl ( \sum_{j=-\infty}^\infty  \big|2^{j \ord}m(2^{-j}\P)f\big|^{q} \Bigr )^{1/q} \biggr \|_{L^p},
\end{align*}
with the usual modification when $q =\infty$.

\item If $p = \infty$, $q \in (0,\infty]$ and $\ord \in \R$, then for all $f \in \mathcal{S}_0'(N)$,
\begin{equation*}
    \|f\|_{\F^{\ord}_{\infty,q}} \asymp \biggl \|\Bigl ( 2^{j\ord} m(2^{-j}\P)f \Bigr )_{j=-\infty}^\infty \biggr \|_{\mathcal{C}_q}
\end{equation*}

\item If $p, q \in (0,\infty]$ and $\ord \in \mathbb{R}$, then for all $f \in \mathcal{S}_0'(N)$,
\begin{align*}
    \| f \|_{\BS} \asymp \biggl (\sum_{j=-\infty}^\infty  \Bigl \| 2^{j \ord} m(2^{-j}\P)f \Bigr \|_{L^p}^q \biggr )^{1/q},
\end{align*}
with the usual modification when $q =\infty$.
\end{enumerate}

If $\N$ is graded, then instead of $\P$ one can equivalently employ $\RLO^{\frac{1}{\nu}}$ in $\rm (i)$ - $\rm (iii)$ for any positive Rockland operator of homogeneous degree $\hdeg \in \NN$. 

If $\N$ is stratified and equipped with the canonical homogeneous dilations, in particular if $\N = \H$, then, depending on the chosen quasi-norm in \eqref{eq:def_con_ker_P}, $\P$ may be chosen to coincide with $(\sL)^{\frac{1}{2}} = \RLO^{\frac{1}{\nu}}$ for a given (negative) sub-Laplacian $\RLO := -\sL$ or other fractional sub-Laplacians.

On the trivially stratified group $\N = \R^n$, equipped with the usual isotropic dilations, this holds true for the choice $\RLO := -\sL := -\Delta := - \sum_{j = 1}^n \partial_j^2$.
\end{corollary}

\subsection{Basic properties}
In this subsection, we show various additional basic properties of the Besov and Triebel-Lizorkin spaces, such as their completeness and elementary inclusions.

\begin{proposition} \label{prop:basic_emb}
 Let 
$0< p \leq \infty$, $0< q \leq \infty$ and 
$\ord \in \mathbb{R}$. Then
\begin{align} \label{eq:basic_embedding_BTL}
\B_{p,\min(p,q)}^\ord (\N) \subseteq \TLS (\N) \subseteq \B_{p, \max(p,q)}^\ord (\N)
\end{align}
with the inclusion maps being continuous; in particular, 
\begin{align*}
\B_{p,p}^\ord (\N) =\F_{p,p}^\ord (\N)
\end{align*}
with equivalent norms.
%\end{enumerate}
\end{proposition}
\begin{proof}
We prove the embeddings \eqref{eq:basic_embedding_BTL} only for $p =\infty$ since for $0< p <\infty$ they follow from elementary inequalities concerning $L^p(\ell^q)$ and $\ell^q(L^p)$, see, e.g.,~\cite[Proposition 2.3.2]{triebel1983theory}.

From the definitions of the $\B_{\infty,q}^\ord$(\N) and $\F_{\infty,q}^\ord(\N)$-quasi-norms, it is easy to see that 
for any $\ord \in \mathbb{R}$ and $0< q \leq \infty$,
\begin{align*}
\|f\|_{\F_{\infty, q}^\ord} \leq \|f\|_{\B_{\infty, q}^\ord}.
\end{align*}
Hence, in order to prove \eqref{eq:basic_embedding_BTL} for $p =\infty$, it remains to show that for all $\ord \in \mathbb{R}$ and $0< q \leq \infty$,
\begin{align} \label{eq:infty_embedding}
  \|f\|_{\B_{\infty, \infty}^\ord} \lesssim \|f\|_{\F_{\infty,q}^\ord}.  
\end{align}
 
To this end, let $x_0 \in \N$ be arbitrary. Then, for any $j \in \mathbb{Z}$, there exists a ball $B \in \mathcal{B}_{-j}$ such that $x_0 \in B$.  Note that for every $x \in B$ we have 
\[
(\crk_{2^{-j}}^\ast f)_a(x) \geq \sup_{y \in B} \frac{\big|f\ast \crk_{2^{-j}}(y) \big|}{(1 + 2^j|y^{-1}x|)^a}\geq \frac{\big|f\ast \crk_{2^{-j}}(x_0) \big|}{(1 + 2^j|x_0^{-1}x|)^a} \geq 2^{-a}\big|f\ast \crk_{2^{-j}}(x_0) \big|.
\]
Hence, 
\begin{align*}
\big|f\ast \crk_{2^{-j}}(x_0) \big| \leq 2^a \inf_{x \in B} \big|(\crk_{2^{-j}}^\ast f)_a(x)  \big|.
\end{align*}
It follows that 
\begin{align*}
\big|2^{j\ord}f\ast\crk_{2^{-j}}(x_0)\big| &\lesssim \left(\inf_{x \in B} \big|2^{j\ord}(\crk_{2^{-j}}^\ast f)_a(x)  \big|^q\right)^{1/q} \leq \left(\dashint_B \big|2^{j\ord}(\crk_{2^{-j}}^\ast f)_a(x)  \big|^q \; d\mu_\N (x)\right)^{1/q} \\&\leq  \left(\dashint_{B} \sum_{\ell \geq j}\big|2^{\ell \ord}(\crk_{2^{-\ell}}^\ast f)_a(x)  \big|^q \; d\mu_\N (x)\right)^{1/q} \\
&\leq \sup_{B' \in \mathcal{B}_{-j}}\left(\dashint_{B'} \sum_{\ell \geq j}\big|2^{\ell\ord}(\crk_{2^{-\ell}}^\ast f)_a(x)  \big|^q \; d\mu_\N (x)\right)^{1/q} \\
& \leq \big\|\big(2^{\ell \ord} (\crk_{2^{-\ell}}^\ast f)_a\big) \big\|_{\mathcal{C}^q}  \asymp
\|f\|_{\F_{\infty,q}^\ord},
\end{align*}
where in the last step we have used Peetre-type maximal function characterization of $\F_{\infty,q}^\ord (\N)$ from Corollary \ref{cor:PM_qn_equiv}. Taking the supremum over $x_0 \in \N$ and $j \in \mathbb{Z}$ yields the desired estimate \eqref{eq:infty_embedding}. This completes the proof.
\end{proof}

For showing the completeness of Besov and Triebel-Lizorkin spaces, we will use a general result concerning the completeness of quasi-normed spaces embedded into tempered distributions. Given a quasi-normed space $E$ which is continuously embedded into $\SV'(\N)$, we say the $E$ has the \textit{Fatou property} if for all sequences $(f_n)_{n=1}^\infty$ in $E$ satisfying
\[
f_n \rightarrow f \ \text{ in } \ \SV'(\N) \quad  \text{ and } \quad \liminf_{n \rightarrow \infty}\|f_n\|_E <\infty
\]
it follows that $f \in E$ and $\|f\|_E \leq \liminf_{n\rightarrow \infty}\|f_n\|_E$.

The following lemma is a quasi-normed version of  \cite[Lemma 14.4.7]{HNVW}. We include its short proof for completeness.

\begin{lemma} \label{lem:completeness_of_quasinormed_spaces}
Let $E$ be a quasi-normed space which is continuously embedded into $\SV'(\N)$. If $E$ has the Fatou property, then $E$ is complete.  
\end{lemma}
\begin{proof}
Let $(f_n)_{n=1}^\infty$ be a Cauchy sequence in $E$; in particular, $(f_n)_{n =1}^{\infty}$ is bounded, and thus $\liminf_{n \rightarrow \infty}\|f_n\|_E <\infty$.  Moreover, since $\SV'(\N)$ is complete and $E$ is continuously embedded into $\SV'(\N)$, there exists $f \in \SV'(\N)$ such that $f_n \rightarrow f$ in $\SV'(\N)$. Hence, by the Fatou property of $E$, it follows that $f \in E$. For showing that $f_n \rightarrow f$ in $E$, let $\varepsilon >0$ be arbitrary and choose $k \in \mathbb{N}$ such that $\|f_n -f_{m}\| < \varepsilon$ for all $n, m >k$. 
Then, by again the Fatou property, 
\[
\|f_n -f\|_E \leq \liminf_{m \rightarrow \infty} \|f_n -f_{m}\|_E \leq \varepsilon,
\]
which shows the completeness of $E$.
\end{proof}

\begin{proposition}
 Let $p,q\in (0,\infty]$ and $\ord \in \mathbb{R}$.  
 \begin{enumerate}
     [\rm (i)]\item The inclusions
     \begin{align} \label{eq:Schwartz_embedding}
     \SV(\N) \subseteq \BS(\N) \subseteq \SV'(\N) \quad \text{and} \quad \SV(\N) \subseteq \TLS(\N) \subseteq \SV'(\N)
     \end{align}
     hold, with the inclusions being continuous.
         \item  
      $\BS(\N)$ and $\TLS(\N)$ are quasi-Banach spaces (Banach spaces if $p,q \in [1,\infty]$).
 \end{enumerate}
\end{proposition}
\begin{proof}
(i) To prove \eqref{eq:Schwartz_embedding}, it suffices to show 
\begin{align} \label{eq:Schwartz_embedding_besov}
    \SV(\N) \subseteq \BS(\N) \subseteq \SV'(\N)
     \end{align}
due to the basic embeddings \eqref{eq:basic_embedding_BTL}.

We first prove the continuous embedding $ \SV(\N) \subseteq \BS(\N)$. Let $\varphi \in \SV(\N)$ and let 
$\crk \in \SV(\N)$ satisfy the discrete Calder\'{o}n condition \eqref{eq:discrete_calderon}.
By  Lemma \ref{aoe} we have, for any $M,L>0$,  
\begin{align} \label{eq:convoluition_omega_km}
\big|\varphi \ast \crk_{2^{-j}} (x)\big|
\lesssim \|\varphi\|_{(k)} 2^{-|j|M} \frac{(2^j \wedge 1)^\hdim}{ \big(1 + (2^{j} \wedge 1)^\hdim|x| \big)^L},
\end{align}
where $k$ is a positive integer  depending only on $M$ and $L$. 

If $0< p <\infty$, we fix $M,L$ such that $M -|\sigma|-\hdim|p-1|>0$ and $L >\hdim/p$. Then from \eqref{eq:convoluition_omega_km}
it follows that 
\begin{align*}
 \|2^{j\ord} \varphi \ast \crk_{2^{-j}}\|_{L^p}  \lesssim \|\varphi\|_{(k)}2^{-|j|(M -|\sigma|-\hdim|p-1|)},
\end{align*}
and hence
\begin{align*}
\|\varphi\|_{\BS} \lesssim \big\| \big(2^{-|j|(M -|\sigma|-\hdim|p-1|}\big)_{j =-\infty}^\infty \big\|_{\ell^q} \|\varphi\|_{(k)} \lesssim \|\varphi\|_{(k)}
\end{align*}
for all $0<q \leq \infty$.
If $p =\infty$, then using \eqref{eq:convoluition_omega_km} it is also easily follows that $\|\varphi\|_{\B_{\infty,q}^\ord} \lesssim \|\varphi\|_{(k)}$. In combination, this shows $\SV (\N) \hookrightarrow \BS(\N)$.

We next show the continuous embedding
$\BS(\N) \subseteq \SV'(\N)$.
Since $\crk$ satisfies the discrete Calder\'{o}n condition \eqref{eq:discrete_calderon}, there exists $\drk \in \SV(\N)$ such that 
\begin{align*}
f= \sum_{j =-\infty}^\infty f \ast \crk_{2^{-j}} \ast \drk_{2^{-j}}
\end{align*}
with convergence in $\SV'(\N)$.   It follows that, for every $\varphi \in \SV(\N)$,
\begin{align}  \label{eq:f_omega_0}
( f,\varphi ) = \sum_{j=-\infty}^\infty ( f\ast \crk_{2^{-j}}, \;  \varphi \ast \drk^{\vee}_{2^{-j}} ).
\end{align}
Consequently,
\begin{align} \label{eq:f_omega}
   |( f,\varphi)| & \lesssim   \sum_{j \in \mathbb{Z}} 
\|f\ast \crk_{2^{-j}}\|_{L^\infty} \|\varphi \ast \drk^{\vee}_{2^{-j}}\|_{L^1}
\end{align}
provided the infinite series converges.
To show this, note that by Lemma \ref{aoe}, 
for any $M,L>0$, there exists $k =k(M,L) \in \mathbb{N}$ such that  
\begin{align*}
 \big|\varphi \ast \drk^{\vee}_{2^{-j}}(x)\big|
\lesssim \|\varphi\|_{(k)} 2^{-|j|M} \frac{(2^j \wedge 1)^\hdim}{ \big(1 + (2^{j} \wedge 1)^\hdim|x| \big)^L}, 
\end{align*}
so that
\begin{align} \label{eq:omega_k_theta}
\| \varphi \ast \drk^{\vee}_{2^{-j}} \|_{L^1} \lesssim \|\varphi\|_{(k)} 2^{-|j|M} .
\end{align}
As such, it remains to prove adequate estimates for $\| f \ast \phi_{2^{-j}} \|_{L^{\infty}}$. We consider the cases $p < \infty$ and $p = \infty$ separately.
\\~\\
\textbf{Case 1.} In this case, we fix $0 < p < \infty$ and choose $0 < r < p$.
To estimate $\|f\ast \crk_{2^{-j}}\|_{L^\infty}$, we use Lemma \ref{lem:central_estimate} to obtain
\begin{align*}
 |f\ast \crk_{2^{-j}}(x) |^r
    \lesssim \sum_{k \in \mathbb{Z}}
    2^{-|j-k|Mr} \int_{\N} \frac{2^{k\hdim}|f \ast \crk_{2^{-k}}(z)|^r}{(1+2^{k}|z^{-1}x|)^{\hdim + \varepsilon_1}} \; d\mu_\N(z), \quad x \in \N
\end{align*}
provided that $Mr >Q +\varepsilon_1$ for some $\varepsilon_1 > 0$.
Fix $x \in \N$. For any $y \in B_{2^{-j}}(x)$, we have $(1 +2^{j}|y^{-1}x|)^{\hdim + \varepsilon_1} \asymp 1$, and hence
\begin{align*} 
 |f\ast \crk_{2^{-j}}(x) |^r
&\lesssim \sum_{k \in \mathbb{Z}}
2^{-|j-k|Mr} \int_{\N} \frac{2^{k\hdim}|f \ast \crk_{2^{-k}}(z)|^r}{(1+2^{k}|z^{-1}x_0|)^{\hdim +\varepsilon_1}(1+2^{j}|y^{-1}x|)^{\hdim +\varepsilon_1}} \; d\mu_\N(z)\\
&\lesssim \sum_{k \in \mathbb{Z}}
2^{-|j-k|(Mr-\hdim -\varepsilon_1)} \int_{\N} \frac{2^{k\hdim}|f \ast \crk_{2^{-k}}(z)|^r}{(1+2^{k}|z^{-1}y|)^{\hdim +\varepsilon_1}} \; d\mu_\N(z), \numberthis \label{eq:fkmj}
\end{align*}
where we used the elementary inequality 
\begin{align*}(1+2^{k}|z^{-1}y|)^{\hdim +\varepsilon_1} \lesssim 2^{|j-k|(\hdim +\varepsilon_1)}(1+2^{k}|z^{-1}x|)^{\hdim +\varepsilon_1}(1+2^{j}|y^{-1}x|)^{\hdim +\varepsilon_1} .
\end{align*}
Using Lemma~\ref{lem:majorant}, we can further estimate the right-hand side to obtain
\begin{align} \label{eq:x0estimate}
|f\ast \crk_{2^{-j}}(x) |^r
\lesssim \sum_{k \in \mathbb{Z}}
2^{-|j-k|(Mr-\hdim -\varepsilon_1)} \mathcal{M}\big(|f \ast \crk_{2^{-k}}|^r\big)(y)
\end{align}
whenever $y \in B_{2^{-j}}(x)$, with an implicit constant independent of $y$.
Raising the estimate \eqref{eq:x0estimate} to the power $p/r$ gives
\begin{align*}
|f \ast \crk_{2^{-j}}(x) |^p &\lesssim \dashint_{B_{2^{-j}}(x)} \left(\sum_{k \in \mathbb{Z}}
2^{-|j-k|(Mr-\hdim -\varepsilon_1)} \mathcal{M}\big(|f \ast \crk_{2^{-k}}|^r\big)(y) \right)^{p/r} d\mu_\N (y) \\
&\lesssim 2^{j\hdim} \left\| \sum_{k \in \mathbb{Z}}
2^{-|j-k|(Mr-\hdim -\varepsilon_1)} \mathcal{M}\big(|f \ast \crk_{2^{-k}}|^r\big)\right\|_{L^{p/r}}^{p/r} \\
&   \leq 2^{j\hdim} \left(\sum_{k \in \mathbb{Z}}
2^{-|j-k|(Mr-\hdim -\varepsilon_1)} \left\|\mathcal{M}\big(|f \ast \crk_{2^{-k}}|^r\big)\right\|_{L^{p/r}}\right)^{p/r} \\
& \lesssim 2^{j\hdim} \left(\sum_{k \in \mathbb{Z}}
2^{-|j-k|(Mr-\hdim -\varepsilon_1)} \left\| f \ast \crk_{2^{-k}}\right\|_{L^{p}}^r\right)^{p/r}, \numberthis \label{eq:pointwise_f_phi_j_x0}
\end{align*}
where we used Minkowski's inequality and  the $L^{p/r}$ boundedness of $\mathcal{M}$, which hold as $p/r > 1$.

If $r <1$, then using H\"{o}lder's inequality to the summation on the right-hand side of \eqref{eq:pointwise_f_phi_j_x0}, we get
\begin{align}
|f\ast \crk_{2^{-j}}(x) |^p & \lesssim 2^{j\hdim} \left[\left(\sum_{k \in \mathbb{Z}}
2^{-|j-k|(Mr-\hdim -\varepsilon_1-\varepsilon_2)/r} \left\| f \ast \crk_{2^{-k}}\right\|_{L^{p}}\right)^{r} \left(\sum_{k \in \mathbb{Z}}
2^{-|j-k|\varepsilon_2/r}\right)^r\right]^{p/r} \nonumber\\
& \lesssim 2^{j\hdim} \left(\sum_{k \in \mathbb{Z}}
2^{-|j-k|(Mr-\hdim -\varepsilon_1-\varepsilon_2)/r} \left\| f \ast \crk_{2^{-k}}\right\|_{L^{p}}\right)^{p}, \label{eq:pointwise_f_phi_j_x0_2}
\end{align}
where $\varepsilon_2 > 0$ is an arbitrarily small positive number. If, instead, $r >1$, then by using the elementary inequality $(\sum_k|a_k|)^{1/r} \leq \sum_k |a_k|^{1/r}$ to the right-hand side of \eqref{eq:pointwise_f_phi_j_x0}, we see that \eqref{eq:pointwise_f_phi_j_x0_2} still holds. 
In combination, this implies
\begin{align} \label{eq:fkmt_Linfty}
  \|f\ast \crk_{2^{-j}}\|_{L^\infty} \lesssim  2^{j\hdim/p}  \sum_{k \in \mathbb{Z}}
2^{-|j-k|(Mr-\hdim -\varepsilon_1-\varepsilon_2)/r} \left\| f \ast \crk_{2^{-k}}\right\|_{L^{p}}. 
\end{align}

Using the estimates \eqref{eq:omega_k_theta} and \eqref{eq:fkmt_Linfty}, it follows from \eqref{eq:f_omega} that
\begin{align} \label{eq:sumjsumk}
|( f,\varphi)| & \lesssim   \sum_{j \in \mathbb{Z}} 
\|f\ast \crk_{2^{-j}}\|_{L^\infty} \|\varphi \ast \drk_{2^{-j}}\|_{L^1} \nonumber\\
& \lesssim \|\varphi\|_{(k)}\sum_{j \in \mathbb{Z}} 2^{-|j|M} 2^{j\hdim/p} 2^{-j\ord}\sum_{k \in \mathbb{Z}}
2^{-|j-k|(Mr-\hdim -\varepsilon_1-\varepsilon_2)/r} 2^{(j-k)\ord}\big\| 2^{k\ord}f \ast \crk_{2^{-k}}\big\|_{L^{p}} \nonumber\\
& \leq \|\varphi\|_{(k)}\sum_{j \in \mathbb{Z}} 2^{-|j|\delta_1}\sum_{k \in \mathbb{Z}} 
2^{-|j-k|\delta_2} \big\|2^{k\ord} f \ast \crk_{2^{-k}}\big\|_{L^{p}} ,
\end{align}
where $\delta_1 := M - \hdim/p -|\ord|$ and $\delta_2:=
M-\hdim /r-\varepsilon_1 /r-\varepsilon_2/r -|\ord|$. 

We now fix the small positive numbers $\varepsilon_1,\varepsilon_2$, and choose a sufficiently large $M$ such that $\delta_1,\delta_2 >0$. Then if $q >1$, we use H\"{o}lder's inequality to get
\begin{align}\label{eq:qgeq1}
\begin{split}
\sum_{j \in \mathbb{Z}} 2^{-|j|\delta_1}&\sum_{k \in \mathbb{Z}}
2^{-|j-k|\delta_2} \big\|2^{k\ord} f \ast \crk_{2^{-k}}\big\|_{L^{p}} \\
&\lesssim  \sum_{j \in \mathbb{Z}} 2^{-|j|\delta_1}\left(\sum_{k \in \mathbb{Z}} 
2^{-|j-k|\delta_2 q'} \right)^{1/q'}
\left(\sum_{k \in \mathbb{Z}} 
 \big\|2^{k\ord} f \ast \crk_{2^{-k}}\big\|_{L^{p}}^q \right)^{1/q} \\
 &\lesssim \|f\|_{\BS}.
\end{split}
\end{align}
If, instead, $0< q \leq 1$, using the elementary inequality $\sum_k |a_k| \leq (\sum_k |a_k|^q)^{1/q}$, we have
\begin{align} \label{eq:qleq1}
\begin{split}
\sum_{j \in \mathbb{Z}} 2^{-|j|\delta_1}&\sum_{k \in \mathbb{Z}}
2^{-|j-k|\delta_2} \big\|2^{k\ord} f \ast \crk_{2^{-k}}\big\|_{L^{p}} \\
&\lesssim  \sum_{j \in \mathbb{Z}} 2^{-|j|\delta_1}
\left(\sum_{k \in \mathbb{Z}} 2^{-|j-k|\delta_2 q}
 \big\|2^{k\ord} f \ast \crk_{2^{-k}}\big\|_{L^{p}}^q \right)^{1/q} \\
 &\lesssim \|f\|_{\BS}.
\end{split}
\end{align}
In combination, this shows
\begin{align*}
|( f,\varphi )| \lesssim \|\varphi\|_{(k)}\|f\|_{\BS}
\end{align*}
for all $q \in (0, \infty]$, 
which implies that $\BS(\N)$ is continuously embedded into $\SV'(\N)$ whenever $p < \infty$.
\\~\\
\textbf{Case 2.} In this case, we treat $p = \infty$. 
It follows from \eqref{eq:f_omega_0} that
\begin{align*}  
   |( f,\varphi)| & = \Biggl | \sum_{j \in \mathbb{Z}}
 ( 2^{j\ord} f\ast\crk_{2^{-j}},\ 2^{-j\ord} \varphi \ast \drk^{\vee}_{2^{-j}}
 ) \Biggr | \\
 &  \lesssim  \left(\sup_{\ell \in \mathbb{Z}}\|2^{\ell \ord}f\ast \crk_{2^{-\ell}}\|_{L^\infty}    \right)\sum_{j \in \mathbb{Z}}
\|2^{-j\ord}\varphi \ast \drk^{\vee}_{2^{-j}}\|_{L^1}. 
\end{align*}
By the estimate \eqref{eq:omega_k_theta}, we have
\begin{align}
  \sum_{j \in \mathbb{Z}}
\|2^{-j\ord}\varphi \ast \drk^{\vee}_{2^{-j}}\|_{L^1} \lesssim \|\varphi\|_{(k)} \sum_{j \in \mathbb{Z}} 2^{-|j|(M -|\sigma|)} \lesssim \|\varphi\|_{(k)},
\end{align}
while
\begin{align} \label{eq:p_infinity_Besov}
\sup_{\ell \in \mathbb{Z}}\|2^{\ell \ord}f\ast \crk_{2^{-\ell}}\|_{L^\infty} \leq \|f\|_{\B_{\infty,q}^\ord}.
\end{align}
Therefore, 
\begin{align*}
    |( f,\varphi)| \lesssim  \|f\|_{\B_{\infty,q}^\ord}\|\varphi\|_{(k)},  
\end{align*}
which shows that $\B_{\infty,q}^\ord(\N)$ is continuously embedded into $\SV'(\N)$.
\\~\\ 
(ii) It is easy to verify that $\BS(\N)$ and $\TLS(\N)$ are quasi-normed spaces (normed spaces if $p,q \geq 1$).  To prove the completeness of $\BS(\N)$ and $\TLS(\N)$, we use Lemma \ref{lem:completeness_of_quasinormed_spaces}, and thus it suffices to verify that all these spaces have the Fatou property.

Let us verify that all Besov spaces $\BS(\N)$ have the Fatou property. 
Choose a function $\crk \in \SV(\N)$ satisfying the discrete Calder\'{o}n condition \eqref{eq:discrete_calderon} and let $(f_n)_{n=1}^{\infty}$
be a sequence in $\BS(\N)$ such that 
$\liminf_{n\rightarrow \infty}\|f_n\|_{\BS} <\infty$ and $f_n \rightarrow f$ in $\SV'(\N)$. Then, for each $j \in \mathbb{Z}$ and $x \in \N$, 
\begin{align} \label{eq:pointwise_fn_phi}
\lim_{n \rightarrow \infty}f_n \ast \crk_{2^{-j}} (x)  = f \ast \crk_{2^{-j}} (x),
\end{align}
cf. \ref{sec:convolution}. Hence, if $p, q <\infty$, it follows from Fatou's lemma that
\begin{align*}
\left\|f\right\|_{\BS}&  = \left(\sum_{j \in \mathbb{Z}} 2^{j\ord q}\left\|f \ast \crk_{2^{-j}} \right\|_{L^p}^q \right)^{1/q} 
= \left(\sum_{j \in \mathbb{Z}} 2^{j\ord q}\left\|\lim_{n\rightarrow \infty}\big|f_n \ast \crk_{2^{-j}} \big|\right\|_{L^p}^q \right)^{1/q} \\
& \leq\left(\sum_{j \in \mathbb{Z}} 2^{j\ord q}\liminf_{n\rightarrow \infty} \left\|f_n \ast \crk_{2^{-j}} \right\|_{L^p}^q \right)^{1/q}  \leq \liminf_{n\rightarrow \infty} \|f_n\|_{\BS},
\end{align*}
and similarly if $q = \infty$. If $p =\infty$, then using that
\begin{align*}
\left\|\lim_{n\rightarrow 
\infty}\big|f_n \ast \crk_{2^{-j}} \big|\right\|_{L^\infty} \leq  \sup_{k \geq 1} \inf_{n\geq k}\esssup_{x \in \N}\big|f_{n}(x)\big|
=\liminf_{n\rightarrow \infty}\|f_n \ast \crk_{2^{-j}}\|_{L^\infty},
\end{align*}
we also easily deduce that
$
\left\|f\right\|_{\B^\ord_{\infty,q}}\leq \liminf_{n\rightarrow \infty} \|f_n\|_{\B^\ord_{\infty,q}}$ for $q \in (0, \infty]$.

Similarly one can verify that all Triebel-Lizorkin space $\TLS(\N)$ have the Fatou property.
\end{proof}

\section{Continuous maximal characterizations} \label{sec:continuous}
In this section, we prove various characterizations of the Besov and Triebel-Lizorkin spaces in terms of continuous (quasi-)norms defined by functions satisfying the continuous Calder\'on condition \eqref{eq:continuous_calderon}. These characterizations will play a key role in obtaining wavelet and frame characterizations in the following sections.

In addition to the Peetre-type maximal function of a distribution $f \in \SV'(N)$, defined by
\[
(\crk_t^{*} f)_a(x) :=\sup_{y \in \N} \frac{|f \ast \crk_s(y)| }{(1+s^{-1}|y^{-1}x|)^\PTpar}, \quad x \in N,
\]
for some $\crk \in \SV(N)$ and $a, t> 0$ (cf. Section \ref{sec:maximal}), 
we also define
\begin{align}
    (\crk_t^{**} f)_a(x) := \sup_{t/2 \leq s \leq 2t} (\crk_s^* f)_\PTpar(x) =  \sup_{\substack{y \in \N \\ t/2 \leq s \leq 2t}} \frac{|f \ast \crk_s(y)| }{(1+s^{-1}|y^{-1}x|)^\PTpar}, \quad x \in N.
\end{align}
Observe that $
    (f * \crk_t)(x) \leq (\crk_t^{*} f)_a(x) \leq (\crk_t^{**} f)_a(x)
$ for arbitrary $x \in N$ and $a,t>0$.

The following theorem is the central results of this paper.

\begin{theorem} \label{thm:cont_char}
Let $\crk \in \mathcal{S}_0(\N)$ satisfy the continuous Calder\'{o}n
condition~\eqref{eq:continuous_calderon}.
\begin{enumerate}[\rm (i)]
    \item Let $p \in (0,\infty)$, $q\in (0,\infty]$ and $\ord \in \mathbb{R}$. If $a > \max \{\frac{\hdim}{p \land q}, |\ord| \}$, then the norm equivalences
\begin{align}
\begin{split} \label{eq:cont_char_1}
    \|f\|_{\F^{\ord}_{p,q}}
    &\asymp
    \biggl\| \biggl( \int_0^\infty t^{-\ord q} |f * \crk_t|^{q} \; \frac{dt}{t}\biggr)^{1/q} \biggr \|_{L^p} \\
    &\asymp
    \biggl\| \biggl( \int_0^\infty t^{-\ord q} \bigl[(\crk_t^* f)_a \bigr]^q \; \frac{dt}{t}\biggr)^{1/q}\biggr\|_{L^p} \\
    &\asymp
    \biggl\| \biggl( \int_0^\infty t^{-\ord q} \bigl[(\crk_t^{**} f)_a \bigr]^q \; \frac{dt}{t}\biggr)^{1/q}\biggr\|_{L^p}
\end{split}
\end{align}
hold for all $f \in \mathcal{S}_0'(\N)$, with the usual modification when $q =\infty$.
    \item  Let $q \in (0, \infty]$ and $\ord \in \mathbb{R}$. If $a > \max \{ \frac{\hdim}{q}, |\ord| \}$ for $q < \infty$, then the norm equivalences
\begin{align}
\begin{split} \label{eq:cont_char_2_q}
    \|f\|_{\F^{\ord}_{\infty,q}}
    &\asymp \sup_{x \in \N, t > 0} \biggl( \dashint_{B_t(x)} \int_0^t
    \tau^{-\ord q} \bigl |f * \crk_{\tau} (y) \bigr|^q \; \frac{d\tau}{\tau}  d\mu_\N(y) \biggr)^{1/q} \\
    &\asymp \sup_{x \in \N, t > 0} \biggl( \dashint_{B_t(x)} \int_0^t
    \tau^{-\ord q} \bigl [ (\crk_\tau^* f)_a(y) \bigr ]^q \; \frac{d\tau}{\tau}  d\mu_\N(y) \biggr)^{1/q} \\
    &\asymp \sup_{x \in \N, t > 0} \biggl( \dashint_{B_t(x)} \int_0^t
    \tau^{-\ord q} \bigl [ (\crk_\tau^{**} f)_a(y) \bigr ]^q \; \frac{d\tau}{\tau}  d\mu_\N(y) \biggr)^{1/q}
\end{split}
\end{align}
hold for all $f \in \mathcal{S}_0'(\N)$, while if $a > \max \{ \hdim, |\ord| \}$ for $q = \infty$, then the norm equivalences
\begin{align}
\begin{split} \label{eq:cont_char_2_infty}
    \|f\|_{\F^{\ord}_{\infty,\infty}}
    &\asymp \sup_{x \in \N, t > 0} \; \sup_{\tau \in (0, t]} \tau^{-\ord} \dashint_{B_t(x)}
    |f * \crk_{\tau} (y) \bigr| \; d\mu_\N(y) \\
    &\asymp \sup_{x \in \N, t > 0} \; \sup_{\tau \in (0, t]} \tau^{-\ord} \dashint_{B_t(x)}
    (\crk_\tau^* f)_a(y) \; d\mu_\N(y) \\
    &\asymp \sup_{x \in \N, t > 0} \; \sup_{\tau \in (0, t]} \tau^{-\ord} \dashint_{B_t(x)}
    (\crk_\tau^{**} f)_a(y) \; d\mu_\N(y)
\end{split}
\end{align}
hold for all $f \in \mathcal{S}_0'(\N)$.
\item Let $p, q \in (0,\infty]$ and $\ord \in \mathbb{R}$. If $a > \max \{ \frac{\hdim}{p}, |\ord| \}$, then the norm equivalences
\begin{align}
\begin{split} \label{eq:cont_char_3}
    \|f\|_{\B^{\ord}_{p,q}}
    &\asymp
    \biggl( \int_0^\infty t^{-\ord q} \| f * \crk_t \|_{L^p}^{q} \; \frac{dt}{t}\biggr)^{1/q} \\
    &\asymp
    \biggl( \int_0^\infty t^{-\ord q} \| (\crk_t^* f)_a \|_{L^p}^{q} \; \frac{dt}{t}\biggr)^{1/q} \\
    &\asymp
    \biggl( \int_0^\infty t^{-\ord q} \| (\crk_t^{**} f)_a \|_{L^p}^{q} \; \frac{dt}{t}\biggr)^{1/q}.
\end{split}
\end{align}
hold for all $f \in \mathcal{S}_0'(\N)$ when $q < \infty$, with the usual modification when $q =\infty$.
\end{enumerate}
\end{theorem}

\begin{proof}
For proving the four estimates in (i) - (iii), we split the proof into four steps, each of them covering the cases (i) - (iii) in parallel.
Throughout the proof, we fix $f \in \SV'(N)$ and let $\eta \in \SV(N)$ be a function satisfying the discrete Calder\' on condition \eqref{eq:discrete_calderon}. We will compute the discrete (quasi-)norms of $\TLS(\N)$ and $\BS(\N)$ always with respect to this function $\eta$.
\\~\\
\textbf{Step~1.} 
In this step, we will estimate the discrete (quasi-)norms of $f$ (defined with respect to $\eta$) by the continuous (quasi-)norms involving the Peetre-type maximal function $(\crk_t^* f)_a$ of $f$ for fixed $a>0$. For this, let $\drk \in \SV(\N)$ be such 
\begin{align*}
    f \ast \crkk_{2^{-j}} = \int_0^{\infty} f \ast \crk_t \ast \drk_t \ast \crkk_{2^{-j}} \, \frac{dt}{t}, \quad j \in \mathbb{Z}.
 \end{align*}
 Since each of the above convolution products defines a smooth function on $\N$, we get 
\begin{align}
	|f \ast \crkk_{2^{-j}}(x)| \leq \int_0^{\infty} \hspace{-6pt} \int_{\N} |f \ast \crk_t(y)| |\drk_t \ast \crkk_{2^{-j}}(y^{-1}x)| \, d\mu_\N (y)  \, \frac{dt}{t} \label{eq:cont_char_convolution}
\end{align}
for all $x \in \N$. We estimate the right-hand side of \eqref{eq:cont_char_convolution} further by first estimating the factors in the integrand individually. By Lemma~\ref{lem:central_estimate}, given  $b, r > 0$, it follows that
\begin{align*}
    |f \ast \crk_t(y)| &\lesssim_{b,r}
    \biggl ( \int_0^\infty
    \Bigl (\frac{\tau}{t} \wedge \frac{t}{\tau} \Bigr )^{br} \tau^{-\hdim} \int_{\N} \frac{|f \ast \crk_\tau(z)|^r}{(1+\tau^{-1}|z^{-1}y|)^{br}} \; d\mu_\N (z) \frac{d\tau}{\tau} \biggr )^{1/r} \\
    &\lesssim_{a,r} \bigg( \int_0^\infty \Bigl (\frac{\tau}{t} \wedge \frac{t}{\tau} \Bigr )^{br} [(\crk_\tau^* f)_a(x)]^r \int_{\N} \tau^{-\hdim} \frac{(1+\tau^{-1}|z^{-1}x|)^{ar}}{(1+\tau^{-1}|z^{-1}y|)^{br}} \; d\mu_\N (z) \frac{d\tau}{\tau} \biggr )^{1/r} \numberthis \label{eq:cont_char_aux_est_2_Step1}
\end{align*}
for all $x, y \in \N$ and $t>0$, where the second inequality used that
\begin{align*}
	|f \ast \crk_\tau(z)|^r \lesssim_{a,r} [(\crk_\tau^* f)_a(x)]^r (1+\tau^{-1}|z^{-1}x|)^{ar}, \quad x, z \in \N.
\end{align*}
 For estimating \eqref{eq:cont_char_aux_est_2_Step1} further, we choose $b = \hdim/r + \varepsilon/r + a$ for some $\varepsilon > 0$, so that 
 $br = \hdim + \varepsilon + ar$, and estimate the inner integral in \eqref{eq:cont_char_aux_est_2_Step1} by
\begin{align*}
	\int_{\N}  \tau^{-\hdim} \frac{(1+\tau^{-1}|z^{-1}x|)^{ar}}{(1+\tau^{-1}|z^{-1}y|)^{br}} \; d\mu_\N (z)  
 &\lesssim (1+\tau^{-1}|y^{-1}x|)^{ar} \int_{\N} \hspace{-3pt} \frac{\tau^{-\hdim} }{(1+\tau^{-1}|z^{-1}y|)^{\hdim + \varepsilon}} \; d\mu_\N (z) \\
 &\lesssim (1+\tau^{-1}|y^{-1}x|)^{ar}
\end{align*}
for $x, y \in \N$, $\tau \in (0, \infty)$. In combination, this gives 
\begin{align*}
    |f \ast \crk_t(y)| 
    &\lesssim \bigg( \int_0^\infty \Bigl (\frac{\tau}{t} \wedge \frac{t}{\tau} \Bigr )^{br} [(\crk_\tau^* f)_a(x)]^r (1+\tau^{-1}|y^{-1}x|)^{ar} \; \frac{d\tau}{\tau} \biggr )^{1/r} \numberthis \label{eq:cont_char_aux_est_2_Step1_2}
\end{align*}
for any $x, y \in N$ and $t>0$.

For estimating the other factor in \eqref{eq:cont_char_convolution}, we simply use Lemma~\ref{aoe} to obtain
\begin{align} \label{eq:estimate_second_factor}
    |\drk_t \ast \crkk_{2^{-j}}(y^{-1}x)| \lesssim 2^{-|j-k |M} \frac{2^{(j \wedge k)\hdim}}{(1+ 2^{j\wedge k}|y^{-1}x|)^{L}}, \quad x, y \in \N,
\end{align}
for any $t \in [2^{-k}, 2^{-k + 1}]$, $k \in \mathbb{Z}$, and where $M, L \in \mathbb{N}$ can be taken arbitrarily large. 

By combining \eqref{eq:cont_char_aux_est_2_Step1_2} and \eqref{eq:estimate_second_factor}, the integrand of \eqref{eq:cont_char_convolution} can thus be estimated as
\begin{align*}
	&|f \ast \crk_t(y)| |\drk_t \ast \crkk_{2^{-j}}(y^{-1}x)| \\
 &\quad \quad \leq 2^{-|j-k |M}  \bigg( \int_0^\infty  \frac{2^{(j \wedge k)\hdim r} (1+\tau^{-1}|y^{-1}x|)^{ar}}{(1+ 2^{j\wedge k}|y^{-1}x|)^{Lr}}    \Bigl (\frac{\tau}{t} \wedge \frac{t}{\tau} \Bigr )^{br} [(\crk_\tau^* f)_a(x)]^r  \; \frac{d\tau}{\tau} \biggr )^{1/r} \\
  &\quad \quad \leq 2^{-|j-k |M} \bigg( \sum_{\ell \in \mathbb{Z}} \int_{2^{-\ell}}^{2^{- \ell + 1}}  2^{-|\ell - k|br} \frac{2^{(j \wedge k)\hdim r} (1+\tau^{-1}|y^{-1}x|)^{ar}}{(1+ 2^{j\wedge k}|y^{-1}x|)^{Lr}}     [(\crk_\tau^* f)_a(x)]^r  \; \frac{d\tau}{\tau} \biggr )^{1/r} \numberthis \label{eq:estimate_integrand}
\end{align*}
for any $t \in [2^{-k}, 2^{-k + 1}]$, $k \in \mathbb{Z}$, where the last step used that 
\begin{align*}
	\Bigl (\frac{\tau}{t} \wedge \frac{t}{\tau} \Bigr )^{br} \lesssim 2^{-|\ell - k|br}.
\end{align*}
We next estimate the $y$-depend expression in \eqref{eq:estimate_integrand}. For this, note that 
since $\tau \in [2^{-\ell},2^{-\ell+1}]$, $\ell \in \Z$, and
$
	|\ell - k \wedge j| \leq |\ell - k| + |j -k|, 
$
we have
\begin{align*}
(1+\tau^{-1}|y^{-1}x|)^{ar} &\asymp (1+2^\ell |y^{-1}x|)^{ar} \leq 2^{|\ell -j \wedge k|ar} (1+2^{j\wedge k}|y^{-1}x|)^{ar} \\
&\leq 2^{|\ell -k|ar}2^{|j-k|ar} (1+2^{j\wedge k}|y^{-1}x|)^{ar},
\end{align*}
and hence
\begin{align*}
 \frac{2^{(j \wedge k)\hdim r} (1+\tau^{-1}|y^{-1}x|)^{ar}}{(1+ 2^{j\wedge k}|y^{-1}x|)^{L r}} \lesssim  2^{|\ell - k|ar} 2^{|j - k|ar} \frac{2^{(j \wedge k)\hdim r}}{(1+ 2^{j\wedge k}|y^{-1}x|)^{(L - a)r}}
\end{align*}
for all $x, y \in \N$. Using this estimate in \eqref{eq:estimate_integrand}, it follows that 
\begin{align*}
	&|f \ast \crk_t(y)| |\drk_t \ast \crkk_{2^{-j}}(y^{-1}x)| \\
  &\quad \quad \lesssim 2^{-|j-k |(M-a)} \frac{2^{(j \wedge k)\hdim}}{(1+ 2^{j\wedge k}|y^{-1}x|)^{(L-a)}}  \bigg( \sum_{\ell \in \mathbb{Z}} \int_{2^{-\ell}}^{2^{- \ell + 1}}  2^{-|\ell - k|(b-a)r}     [(\crk_\tau^* f)_a(x)]^r  \; \frac{d\tau}{\tau} \biggr )^{1/r} \numberthis \label{eq:estimate_integrand2}
\end{align*}
Choosing $L>a+\hdim$ yields
\[
\int_N \frac{2^{(j \wedge k)\hdim }}{(1+ 2^{j\wedge k}|y^{-1}x|)^{(L - a)}} \; d\mu_N (y) \lesssim 1,
\]
so that integrating the estimate \eqref{eq:estimate_integrand2} over $N$ gives that for all $t \in [2^{-k},2^{-k+1}]$,
\begin{align*}
\int_N |f \ast \crk_t(y)|& |\drk_t \ast \crkk_{2^{-j}}(y^{-1}x)| \; d\mu_N (y)\\ &\lesssim 2^{-|j-k |(M-a)} \bigg( \sum_{\ell \in \mathbb{Z}} 2^{-|\ell - k|(b-a)r}  \int_{2^{-\ell}}^{2^{- \ell + 1}}    [(\crk_\tau^* f)_a(x)]^r  \; \frac{d\tau}{\tau} \biggr )^{1/r}\\
&\lesssim 2^{-|j-k |(M-a)} \bigg( \sum_{\ell \in \mathbb{Z}} 2^{-|\ell - k|(b-a)r} 2^{- \ell \sigma r} \int_{2^{-\ell}}^{2^{- \ell + 1}} \tau^{-\ord r}    [(\crk_\tau^* f)_a(x)]^r  \; \frac{d\tau}{\tau} \biggr )^{1/r},
\end{align*}
where the last inequality used $\tau^{-1}\asymp 2^{\ell}$. 
We then   
integrate both sides on $[2^{-k},2^{-k+1}]$ with respect to the measure $dt/t$, and multiply both sides by $2^{j\ord}$, to obtain \begin{align*}
2^{j\ord}\int_{2^{-k}}^{2^{-k+1}}&\int_N |f \ast \crk_t(y)||\drk_t \ast \crkk_{2^{-j}}(y^{-1}x)| \; d\mu_N (y) \; \frac{dt}{t}\\ &\lesssim 2^{-|j-k |(M-a)} \bigg( \sum_{\ell \in \mathbb{Z}} 2^{-|\ell - k|(b-a)r}  2^{(j-\ell) \ord r}\int_{2^{-\ell}}^{2^{- \ell + 1}}   \tau^{-\ord r} [(\crk_\tau^* f)_a(x)]^r  \; \frac{d\tau}{\tau} \biggr )^{1/r}\\
& \leq 2^{-|j-k |(M-a-|\sigma|)} \bigg( \sum_{\ell \in \mathbb{Z}} 2^{-|\ell - k|(b-a-|\sigma|)r}  \int_{2^{-\ell}}^{2^{- \ell + 1}}   \tau^{-\ord r} [(\crk_\tau^* f)_a(x)]^r  \; \frac{d\tau}{\tau} \biggr )^{1/r},
\end{align*}
where the last inequality used $2^{(j -\ell)\ord r}  \leq 2^{|j-k||\ord| r}2^{|k-\ell| |\ord| r}$. 
Finally, summing both sides of the last displayed inequality over $k \in \mathbb{Z}$,  and recalling the estimate \eqref{eq:cont_char_convolution}, it follows that
\begin{align*}
&2^{j\ord} |f\ast \crkk_{2^{-j}} (x)| \\
&\quad \lesssim \sum_{k \in \Z} 2^{-|j-k|(M-a-|\ord|)}  \biggl ( \sum_{\ell \in \Z} 2^{-|\ell-k|(b-a - |\ord|)r}  \int_{2^{-\ell}}^{2^{-\ell+1}} \hspace{-6pt} \tau^{-\ord r} \bigl [ (\crk_\tau^* f)_a(x) \bigr ]^r \; \frac{d\tau}{\tau} \biggr )^{1/r} \numberthis \label{eq:pointwise_f_eta}
\end{align*}
for all $x \in \N$ and $j \in \Z$.

We next treat the cases (i)-(iii) separately.
\\~\\
\textit{Case (i).}
If $q \in (0, \infty)$, we choose $r = q$ in \eqref{eq:pointwise_f_eta} and choose $M > a + |\ord|$ and $br = \hdim + \delta + ar + |\ord|r$, for any $\delta > 0$.
Our earlier choice of $b$ is not violated by this if we pick $\varepsilon = \delta + |\ord|r > 0$, which we may without loss of generality. We now apply the inequality \eqref{eq:Ry_Lp_lq} twice, once for the summation in $k$ and once for the summation in $\ell$, thus obtaining
\begin{align*}
    \|f\|_{\F^{\ord}_{p,q}}
    &\asymp \big\| \big(2^{j\ord}   f \ast \crkk_{2^{-j}}\big)_{j=-\infty}^\infty\big\|_{L^p(\ell^q)} \\
    & \lesssim \left\| \left(\sum_{k \in \Z} 2^{-|j-k|(M-a-|\ord|)}  \biggl ( \sum_{\ell \in \Z} 2^{-|\ell-k|(b -a - |\ord|) q} \int_{2^{-\ell}}^{2^{-\ell+1}} \hspace{-6pt} \tau^{-\ord q} \bigl [ (\crk_\tau^* f)_a(x) \bigr ]^q \; \frac{d\tau}{\tau} \biggr )^{1/q}   \right)_{j =-\infty}^\infty  \right\|_{L^p(\ell^q)} \\
    & \lesssim 
    \left\| \left(   \biggl (\sum_{\ell \in \Z} 2^{-|\ell-k|(b -a - |\ord|) q} \int_{2^{-\ell}}^{2^{-\ell+1}} \hspace{-6pt} \tau^{-\ord q} \bigl [ (\crk_\tau^* f)_a(x) \bigr ]^q \; \frac{d\tau}{\tau} \biggr )^{1/q}   \right)_{k =-\infty}^\infty  \right\|_{L^p(\ell^q)} 
    \\
    &= \left\| \left(    \sum_{\ell \in \Z} 2^{-|\ell-k|(b -a - |\ord|) q} \int_{2^{-\ell}}^{2^{-\ell+1}} \hspace{-6pt} \tau^{-\ord q} \bigl [ (\crk_\tau^* f)_a(x) \bigr ]^q \; \frac{d\tau}{\tau}    \right)_{k =-\infty}^\infty  \right\|_{L^{p/q}(\ell^1)}^{1/q} 
    \\
    & \lesssim 
    \left\| \left(    \int_{2^{-\ell}}^{2^{-\ell+1}} \hspace{-6pt} \tau^{-\ord q} \bigl [ (\crk_\tau^* f)_a(x) \bigr ]^q\; \frac{d\tau}{\tau}    \right)_{\ell =-\infty}^\infty  \right\|_{L^{p/q}(\ell^1)}^{1/q} 
    \\
    &= \biggl \| \biggl( \int_0^\infty \tau^{-\ord q} \bigl[(\crk_\tau^* f)_a \bigr]^q \; \frac{d\tau}{\tau} \biggr )^{1/q} \biggr \|_{L^p},
\end{align*}
as claimed. 

If $q = \infty$, we choose $r = 1$ in \eqref{eq:pointwise_f_eta} and repeat all estimates up to the penultimate one to conclude that
\begin{align*}
     \|f\|_{\F^\ord_{p, \infty}} \lesssim
     \left\| \left(    \int_{2^{-\ell}}^{2^{-\ell+1}} \hspace{-6pt} \tau^{-\ord q} \bigl [ (\crk_\tau^* f)_a(x) \bigr ]^q\; \frac{d\tau}{\tau}    \right)_{\ell =-\infty}^\infty  \right\|_{L^p(\ell^\infty)}.
\end{align*}
In combination with the observation that
\begin{align} \label{eq:aux_est_L1_Linfty}
    \int_{2^{-\ell}}^{2^{-\ell+1}} |g(\tau)| \; \frac{d\tau}{\tau} \leq \ln(2) \Bigl \| \tau^{-\ord} |g(\tau)| \Bigr \|_{L^\infty([2^{-\ell}, 2^{-\ell+1}], \frac{d\tau}{\tau})}
\end{align}
for all $g \in L^1_{loc}([2^{-\ell}, 2^{-\ell+1}], \frac{d\tau}{\tau})$, this yields the desired estimate
\begin{align*}
    \|f\|_{\F^\ord_{p, \infty}} \lesssim \Bigl \| \tau^{-\ord} \|  (\crk_\tau^* f)_a \|_{L^\infty((0, \infty), \frac{d\tau}{\tau})} \Bigr \|_{L^{p}}.
\end{align*}
\\~\\
\textit{Case (iii).}
Clearly, we may assume that $\int_0^{\infty} \tau^{- \sigma q} \| (\phi_{\tau}^* f)_a \|_{L^p}^q \; \frac{d\tau}{\tau} < \infty$, so that $\tau \mapsto \tau^{- \sigma q}(\phi_{\tau}^* f)_a$ is Bochner integrable.
If $q \in (0, \infty)$, we choose $r < p \wedge q$ in \eqref{eq:pointwise_f_eta}, and apply \eqref{eq:Ry_lq_Lp} to obtain
\begin{align*}
\|f\|_{\BS} & \asymp \left\|\left(2^{j\sigma} f \ast \crkk_{2^{-j}}\right)_{j=-\infty}^\infty \right\|_{\ell^q (L^p)}\\
& \lesssim \left\|\left(\biggl ( \sum_{\ell \in \mathbb{Z}} 2^{-|\ell-k|(b-a - |\sigma|)r} \int_{2^{-\ell}}^{2^{-\ell+1}} \hspace{-6pt} \tau^{-\sigma r} \bigl [ (\crk_\tau^* f)_a \bigr ]^r \; \frac{d\tau}{\tau} \biggr )^{1/r} \right)_{k=-\infty}^\infty \right\|_{\ell^q (L^p)}\\
&= \left( \sum_{k\in \mathbb{Z}} \left\| \biggl ( \sum_{\ell \in \mathbb{Z}} 2^{-|\ell-k|(b-a - |\sigma|)r} \int_{2^{-\ell}}^{2^{-\ell+1}} \hspace{-6pt} \tau^{-\sigma r} \bigl [ (\crk_\tau^* f)_a \bigr ]^r \; \frac{d\tau}{\tau} \biggr )^{1/r}\right\|_{L^p}^q\right)^{1/q} \\
& = \left( \sum_{k\in \mathbb{Z}} \left\|  \sum_{\ell \in \mathbb{Z}} 2^{-|\ell-k|(b-a - |\sigma|)r} \int_{2^{-\ell}}^{2^{-\ell+1}} \hspace{-6pt} \tau^{-\sigma r} \bigl [ (\crk_\tau^* f)_a\bigr ]^r \; \frac{d\tau}{\tau} \right\|_{L^{p/r}}^{q/r}\right)^{1/q} .
\end{align*}
Since $p/r >1$, it follows by Minkowski's inequality that
\begin{align}
\|f\|_{\BS} & \lesssim \left( \sum_{k\in \mathbb{Z}} \left(  \sum_{\ell \in \mathbb{Z}} 2^{-|\ell-k|(b-a - |\sigma|)r} \left\|\int_{2^{-\ell}}^{2^{-\ell+1}} \hspace{-6pt} \tau^{-\sigma r} \bigl [ (\phi_\tau^* f)_a\bigr ]^r \; \frac{d\tau}{\tau} \right\|_{L^{p/r}}\right)^{q/r}\right)^{1/q} \nonumber \\
& \leq \left( \sum_{k\in \mathbb{Z}} \left(  \sum_{\ell \in \mathbb{Z}} 2^{-|\ell-k|(b-a - |\sigma|)r}\int_{2^{-\ell}}^{2^{-\ell+1}} \hspace{-6pt} \tau^{-\sigma r} \left\| \bigl [ (\phi_\tau^* f)_a\bigr ]^r \right\|_{L^{p/r}} \; \frac{d\tau}{\tau} \right)^{q/r}\right)^{1/q}, \label{eq:Minkowski}
\end{align}
where the last inequality follows from a standard fact of Bochner integration, see, e.g., \cite[Proposition 1.2.2]{hytonen2016analysis}. 

For estimating the integral of the right-hand side of \eqref{eq:Minkowski}, we use H\"{o}lder's inequality to obtain
\begin{align*}
&\int_{2^{-\ell}}^{2^{-\ell+1}} \hspace{-6pt} \tau^{-\sigma r} \left\| \bigl [ (\phi_\tau^* f)_a\bigr ]^r \right\|_{L^{p/r}} \; \frac{d\tau}{\tau}  \\
&\quad\quad  \leq \left(\int_{2^{-\ell}}^{2^{-\ell+1}} \hspace{-6pt} \Big(\tau^{-\sigma r} \left\| \bigl [ (\phi_\tau^* f)_a \bigr ]^r \right\|_{L^{p/r}}\Big)^{q/r} \; \frac{d\tau}{\tau}\right)^{r/q} \left(\int_{2^{-\ell}}^{2^{-\ell+1}}  \hspace{-6pt} 1 \; \frac{d\tau}{\tau} \right)^{1-(r/q)} \\
& \quad\quad \lesssim \left(\int_{2^{-\ell}}^{2^{-\ell+1}} \hspace{-6pt} \tau^{-\sigma q} \left\| \bigl [ (\phi_\tau^* f)_a \bigr ]^r \right\|_{L^{p/r}}^{q/r}\frac{d\tau}{\tau}\right)^{r/q} \\
& \quad\quad =  \left( \int_{2^{-\ell}}^{2^{-\ell+1}} \hspace{-6pt} \tau^{-\sigma q} \left\|  (\phi_\tau^* f)_a \right\|_{L^{p}}^{q}\frac{d\tau}{\tau}\right)^{r/q} .
\end{align*}
Inserting this back into \eqref{eq:Minkowski}, and using H\"{o}lder's inequality for the summation in $\ell$, we get
\begin{equation} \label{eq:Holder_for_sum}
\begin{split}
\|f\|_{\BS} & \lesssim \left( \sum_{k\in \mathbb{Z}} \left(  \sum_{\ell \in \mathbb{Z}} 2^{-|\ell-k|(b-a - |\sigma|)r} \left(\int_{2^{-\ell}}^{2^{-\ell+1}} \hspace{-6pt} \tau^{-\sigma q} \left\|  (\crk_\tau^* f)_a \right\|_{L^{p}}^{q}\frac{d\tau}{\tau}\right)^{r/q} \right)^{q/r}\right)^{1/q}\\
&\lesssim  \left( \sum_{k\in \mathbb{Z}}   \sum_{\ell \in \mathbb{Z}} 2^{-|\ell-k|[(b-a - |\sigma|)r-\varepsilon]q/r}  \int_{2^{-\ell}}^{2^{-\ell+1}} \hspace{-6pt} \tau^{-\sigma q} \left\|  (\crk_\tau^* f)_a \right\|_{L^{p}}^{q}\frac{d\tau}{\tau} \right)^{1/q} \\
& = \left\|\left(   \sum_{\ell \in \mathbb{Z}} 2^{-|\ell-k|[(b-a - |\sigma|)r-\varepsilon]q/r}  \int_{2^{-\ell}}^{2^{-\ell+1}} \hspace{-6pt} \tau^{-\sigma q} \left\|  (\crk_\tau^* f)_a \right\|_{L^{p}}^{q}\frac{d\tau}{\tau}\right)_{k=-\infty}^\infty  \right\|_{\ell^1}^{1/q},
\end{split}
\end{equation}
where $\varepsilon$ is an arbitrary positive number.
By \eqref{Rylq}, we can estimate \eqref{eq:Holder_for_sum} further as
\begin{align*}
\|f\|_{\BS} & \lesssim \left\|\left(    \int_{2^{-\ell}}^{2^{-\ell+1}} \hspace{-6pt} \tau^{-\sigma q} \left\|  (\crk_\tau^* f)_a \right\|_{L^{p}}^{q}\frac{d\tau}{\tau}\right)_{\ell=-\infty}^\infty  \right\|_{\ell^1}^{1/q} \\
&= \left( \sum_{\ell \in \mathbb{Z}}\int_{2^{-\ell}}^{2^{-\ell+1}} \hspace{-6pt} \tau^{-\sigma q} \left\|  (\crk_\tau^* f)_a \right\|_{L^{p}}^{q}\frac{d\tau}{\tau} \right)^{1/q}\\
 & = \left(\int_0^\infty \hspace{-6pt} \tau^{-\sigma q} \left\|  (\crk_\tau^* f)_a \right\|_{L^{p}}^{q}\frac{d\tau}{\tau} \right)^{1/q},
\end{align*}
which proves the claim.

If $q = \infty$, we choose $r < (p \land 1)$ in \eqref{eq:pointwise_f_eta}. We repeat the proof up to \eqref{eq:Holder_for_sum} and apply \eqref{Rylq} to obtain
\begin{align*}
\|f\|_{\B^\ord_{p, \infty}} \lesssim \left\|\left( \int_{2^{-\ell}}^{2^{-\ell+1}} \hspace{-6pt} \tau^{-\sigma} \left\|  (\crk_\tau^* f)_a \right\|_{L^{p}}\frac{d\tau}{\tau}\right)_{\ell=-\infty}^\infty \right\|_{\ell^\infty}.
\end{align*}
As in (i), we finally use \eqref{eq:aux_est_L1_Linfty} to derive the desired estimate
\begin{align*}
    \|f\|_{\B^\ord_{p, \infty}} & \lesssim \Bigl \| \tau^{-\ord} \|  (\crk_\tau^* f)_a \|_{L^{p}} \Bigr \|_{L^\infty((0, \infty), \frac{d\tau}{\tau})}.
\end{align*}
\\~\\
\textit{Case (ii).}
This case requires a slightly different strategy. If $q \in (0, \infty]$, we apply the $\mathcal{C}_q$-quasi-norms to both sides of \eqref{eq:pointwise_f_eta}, we simplify the estimate by a two-fold application of \eqref{eq:bounded on Cq} in $k$ and $j$, assuming $M, b > a + | \ord |$,
\begin{align*}
   \|f\|_{\F^{\ord}_{\infty,q}} &\asymp \Bigl \| \Bigl ( 2^{j\ord} |f\ast \crkk_{2^{-j}}| \Bigr )_{j \in \Z} \Bigr \|_{\mathcal{C}_q} \\
    &\lesssim \Biggl \| \Biggl ( \sum_{k \in \Z} 2^{-|j-k|(M-a-|\ord|)} \biggl ( \sum_{\ell \in \Z} 2^{-|\ell-k|(b-a - |\ord|)r} \int_{2^{-\ell}}^{2^{-\ell+1}} \hspace{-12pt} \tau^{-\ord r} \bigl [ (\crk_\tau^* f)_a \bigr ]^r \; \frac{d\tau}{\tau} \biggr )^{1/r} \Biggr )_{j \in \Z} \Biggr \|_{\mathcal{C}_q} \\
    &\lesssim \Biggl \| \Biggl ( \int_{2^{-\ell}}^{2^{-\ell+1}} \hspace{-6pt} \tau^{-\ord r} \bigl [ (\crk_\tau^* f)_a \bigr ]^r \; \frac{d\tau}{\tau} \biggr )^{1/r} \Biggr )_{\ell \in \Z} \Biggr \|_{\mathcal{C}_q}.
\end{align*}
The sufficient condition \eqref{supinf} from Lemma~\ref{infi}, which allows us to employ \eqref{eq:bounded on Cq}, is satisfied due to an estimate that is identical to
\eqref{eq:justification_supinf} from Step~1 of the proof of Theorem~\ref{thm:indep_crk}.

If $q \in (0, \infty)$, we choose $r = q$ and write out the quasi-norm as
\begin{align*} 
& \biggl \| \biggl ( \int_{2^{-\ell}}^{2^{-\ell+1}} \hspace{-6pt} \tau^{-\ord q} \bigl [ (\crk_\tau^* f)_a \bigr ]^q \; \frac{d\tau}{\tau} \biggr )^{1/q}_{\ell \in \Z} \biggr \|_{\mathcal{C}_q} \\
& \quad \quad \quad \quad = \sup_{\ell_0 \in \Z, m \in \NN} \biggl( \dashint_{B^{\ell_0}_m} \sum_{\ell = -\ell_0}^{\infty}
\int_{2^{-\ell}}^{2^{-\ell+1}} \hspace{-6pt} \tau^{-\ord q} \bigl [ (\crk_\tau^* f)_a(y) \bigr ]^q \; \frac{d\tau}{\tau}
d\mu_\N(y) \biggr)^{1/q} \\
& \quad \quad \quad \quad = \sup_{\ell_0 \in \Z, m \in \NN} \biggl( \dashint_{B^{\ell_0}_m}
\int_0^{2^{\ell_0+1}} 
\hspace{-6pt} \tau^{-\ord q} \bigl [ (\crk_\tau^* f)_a(y) \bigr ]^q \; \frac{d\tau}{\tau}
      d\mu_\N(y)\biggr)^{1/q}.
\end{align*}
Note that for each fixed $\ell_0 \in \mathbb{Z}$, if we set $t = 2^{\ell_0 +1}$, then 
\[
 \dashint_{B^{\ell_0}_m}\int_0^{2^{\ell_0+1}} 
\hspace{-6pt} \tau^{-\ord q} \bigl [ (\crk_\tau^* f)_a(y) \bigr ]^q \; \frac{d\tau}{\tau} \lesssim 
\dashint_{B_t(x_m)} \int_0^t\tau^{-\ord q} \bigl [ (\crk_\tau^* f)_a(y) \bigr ]^q \; \frac{d\tau}{\tau}.
\]
Hence it follows that
\begin{align*} 
\|f\|_{\F^{\ord}_{\infty,q}}  &\lesssim \biggl \| \biggl ( \int_{2^{-\ell}}^{2^{-\ell+1}} \hspace{-6pt} \tau^{-\ord q} \bigl [ (\crk_\tau^* f)_a \bigr ]^q \; \frac{d\tau}{\tau} \biggr )^{1/q}_{\ell \in \Z} \biggr \|_{\mathcal{C}_q} \\
&\lesssim \sup_{x \in \N, t > 0} \biggl( \dashint_{B_t(x)} \int_0^t
\tau^{-\ord q} \bigl [ (\crk_\tau^* f)_a(y) \bigr ]^q \; \frac{d\tau}{\tau}  d\mu_\N(y)\biggr)^{1/q},
\end{align*}
as desired.

If, on the other hand, $q = \infty$, we set $r = 1$ and write out the quasi-norm as
\begin{align*}  
\biggl \| \biggl ( \int_{2^{-\ell}}^{2^{-\ell+1}} \hspace{-6pt} \tau^{-\ord} (\crk_\tau^* f)_a \frac{d\tau}{\tau} \biggr )_{\ell \in \Z} \biggr \|_{\mathcal{C}_\infty}
&= \sup_{\ell_0 \in \Z, m \in \NN} \; \sup_{\ell \geq -\ell_0 } \; \dashint_{{B}^{\ell_0}_m}
\int_{2^{-\ell}}^{2^{-\ell+1}} \hspace{-6pt} \tau^{-\ord q} (\crk_\tau^* f)_a(y) \,
\frac{d\tau}{\tau}  d\mu_\N(y)  \\
& = \sup_{\ell_0 \in \Z, m \in \NN} \; \sup_{\ell \geq -\ell_0 } \; \int_{2^{-\ell}}^{2^{-\ell+1}} \hspace{-6pt} \tau^{-\ord q} \dashint_{{B}^{\ell_0}_m}
(\crk_\tau^* f)_a(y)  d\mu_\N(y) \,
\frac{d\tau}{\tau},   
\end{align*}
where we have used Fubini's theorem for the second identity. 
Note that for each fixed $\ell_0 \in \mathbb{Z}$, if we set $t = 2^{\ell_0+1}$, then for all $\ell \geq -\ell_0$,
\begin{align*}
\int_{2^{-\ell}}^{2^{-\ell+1}} \hspace{-6pt} \tau^{-\ord q} \dashint_{{B}^{\ell_0}_m}
(\crk_\tau^* f)_a(y)  d\mu_\N(y)&\leq \ln(2) \sup_{\tau \in [2^{-\ell}, 2^{-\ell - 1}]} \tau^{-\ord} \dashint_{{B}^{\ell_0}_m}
(\crk_\tau^* f)_a(y) \, d\mu_\N(y)\\
& \leq \ln(2) \sup_{x\in \N}\sup_{\tau \in (0,t]} \tau^{-\ord} \dashint_{B_t(x)}
(\crk_\tau^* f)_a(y) \, d\mu_\N(y).
\end{align*}
Hence, if we apply the $\mathcal{C}_q$-quasi-norms to both sides of \eqref{eq:pointwise_f_eta} and simplify the estimate by using \eqref{eq:bounded on Cq} as we did for $q < \infty$, then
\begin{align*}
	\|f\|_{\F^{\ord}_{\infty,\infty}}
    \lesssim \biggl \| \biggl ( \int_{2^{-\ell}}^{2^{-\ell+1}} \hspace{-6pt} \tau^{-\ord} (\crk_\tau^* f)_a \frac{d\tau}{\tau} \biggr )_{\ell \in \Z} \biggr \|_{\mathcal{C}_\infty}
    \lesssim \sup_{x \in \N, t > 0} \; \sup_{\tau \in (0, t]} \tau^{-\ord} \dashint_{B_t(x)}
    (\crk_\tau^* f)_a(y) \, d\mu_\N(y),
\end{align*}
which completes Step 1 of the proof.
\\~\\
\textbf{Step 2.} In this step, we simply observe that the quasi-norms in the cases (i) - (iii) involving the Peetre-type maximal function $(\crk_{\tau}^* f)_a$ can be trivially estimated by those involving $(\crk_{\tau}^{**} f)_a$ as $(\crk_{\tau}^* f)_a (x) \leq (\crk_{\tau}^{**} f)_a(x) $ for all $x \in N$ and $a, t > 0$. 
\\~\\
\textbf{Step 3.}
In this step, we estimate the quasi-norms involving the  maximal function 
$(\crk^{**}_t f)_{a}$ by the continuous quasi-norms defined by the continuous Calder\'on condition \eqref{eq:continuous_calderon}. 

We write $t = 2^{-j}u$ for some (unique) $j \in \Z$ and $u \in [1, 2)$ to get the immediate pointwise estimate
\begin{align*}
    (\crk^{**}_{t} f)_{a} (x) &= \sup_{\substack{y \in \N \\ 2^{-j-1}u \leq v \leq 2^{-j+1}u}} \frac{|f * \crk_v(y)|}{(1 + v^{-1} |y^{-1}x|)^a} = \sup_{\substack{y \in \N \\ \frac{u}{2} \leq s \leq 2u}} \frac{|f * \crk_{2^{-j}s}(y)|}{(1 + 2^j s^{-1} |y^{-1}x|)^a} \\
    &\leq 4^a \sup_{\substack{y \in \N \\ \frac{1}{2} \leq s \leq 4}} \frac{|f * \crk_{2^{-j}s}(y)|}{(1 + 2^j s^{-1} |y^{-1}x|)^a} = 4^a \sup_{\frac{1}{2} \leq s \leq 4} (\crk^{*}_{2^{-j}s} f)_{a} (x)
\end{align*}
for any $x \in \N$. 
With $a > 0$ fixed and $t$ given as above, we pick $M := 2a$ and let $r>0$ be arbitrary, 
so that an application 
of Corollary \ref{coro:central_estimate} gives
\begin{align*}
    |(\crk^{*}_{2^{-j}s} f)_{a} (x)| 
    &\leq_{M,r} \biggl ( \int_0^\infty
    \Bigl (\frac{2^{-j}s}{\tau} \wedge \frac{\tau}{2^{-j}s} \Bigr )^{ar} \int_{\N} \frac{\tau^{-\hdim} |f \ast \crk_{\tau}(z)|^r}{(1+\tau^{-1}|z^{-1}x|)^{ar}} \, d\mu_\N (z) \frac{d\tau}{\tau} \biggr )^{1/r} \\
    &\lesssim \biggl ( \sum_{\ell \in \mathbb{Z}} 2^{-|\ell - j|ar} \int_{2^{-\ell}}^{2^{-\ell + 1}}
      \int_{\N} \frac{\tau^{-\hdim} |f \ast \crk_{\tau}(z)|^r}{(1+\tau^{-1}|z^{-1}x|)^{ar}} \, d\mu_\N (z) \frac{d\tau}{\tau} \biggr )^{1/r}
\end{align*} for all $s \in [1/2, 4]$ and $x \in \N$, where the last inequality used that
\begin{align*}
    \Bigl (\frac{2^{-j}s}{\tau} \wedge \frac{\tau}{2^{-j}s} \Bigr )^{ar} \lesssim 2^{-|\ell - j|ar}, \quad s \in [1/2, 4].
\end{align*}
Thus,
\begin{align}
\begin{split} \label{eq:cont_char_crucial_est_Step3}
    |(\crk^{**}_t f)_{a} (x)| &\lesssim \biggl ( \sum_{\ell \in \Z} 2^{-|\ell - j|ar} \int_{2^{-\ell}}^{2^{-\ell+1}} \int_{\N} \frac{\tau^{-\hdim} |f \ast \crk_{\tau}(z)|^r}{(1+\tau^{-1}|z^{-1}x|)^{ar}} \, d\mu_\N (z) \, \frac{d\tau}{\tau} \biggr )^{1/r}
\end{split}
\end{align}
for all $x \in \N$ and $t \in [2^{-j},2^{-j+1})$ for $j \in \Z$.  From \eqref{eq:cont_char_crucial_est_Step3} we can now derive the desired estimates for the cases (i) - (iii).
\\~\\
\textit{Case (i).} If $q \in (0, \infty)$, then choose $r \in (0, p \land q)$. Since we assume that $a > \frac{\hdim}{p \land q}$, we can even pick an $r$ that satisfies $a > \frac{\hdim}{r} > \frac{\hdim}{p \land q}$. Now, Fubini's theorem and the majorant property of the Hardy-Littlewood maximal function (Lemma \ref{lem:majorant}) allows us to estimate 
\begin{align*}
    &\int_{2^{-\ell}}^{2^{-\ell+1}} \int_{\N} \frac{\tau^{-\hdim} |f \ast \crk_{\tau}(z)|^r}{(1+\tau^{-1}|z^{-1}x|)^{ar}} \, d\mu_\N (z) \, \frac{d\tau}{\tau}\\
    &\quad\quad =\int_{\N} \int_{2^{-\ell}}^{2^{-\ell+1}}  \frac{\tau^{-\hdim} |f \ast \crk_{\tau}(z)|^r}{(1+\tau^{-1}|z^{-1}x|)^{ar}} \, \frac{d\tau}{\tau}\, d\mu_\N (z)  \\
    &\quad\quad \asymp \int_{\N} \frac{2^{\ell \hdim}}{(1+2^\ell |z^{-1}x|)^{ar}} \Bigl (\int_{2^{-\ell}}^{2^{-\ell+1}}  |f \ast \crk_{\tau}(z)|^r\, \frac{d\tau}{\tau} \Bigr ) \, d\mu_\N (z)  \\
    &\quad\quad \lesssim \mathcal{M} \Bigl ( \int_{2^{-\ell}}^{2^{-\ell+1}}  |f \ast \crk_{\tau}|^r \, \frac{d\tau}{\tau} \Bigr )(x).
\end{align*}
Inserting this into \eqref{eq:cont_char_crucial_est_Step3}, we obtain the pointwise estimate
\begin{align*}
    |(\crk^{**}_t f)_{a} (x)| &\lesssim \biggl ( \sum_{\ell \in \Z} 2^{-|\ell - j|ar} 
    \mathcal{M} \Bigl ( \int_{2^{-\ell}}^{2^{-\ell+1}}  |f \ast \crk_{\tau}|^r \, \frac{d\tau}{\tau} \Bigr )(x) \biggr )^{1/r}
\end{align*}
 for all $x \in \N$. This easily yields
\begin{align*}
    \biggl\| \biggl( \int_0^\infty & t^{-\ord q} \bigl[(\crk_t^{**} f)_a \bigr]^q \; \frac{dt}{t}\biggr)^{1/q}\biggr\|_{L^p} \\
    &\lesssim \Biggl \| \Biggl ( \sum_{j \in \Z} \int_{2^{-j}}^{2^{-j +1}} t^{-\ord q} \biggl ( \sum_{\ell \in \Z} 2^{-|\ell - j|ar} \mathcal{M} \Bigl ( \int_{2^{-\ell}}^{2^{-\ell+1}}  |f \ast \crk_{\tau}|^r \, \frac{d\tau}{\tau} \Bigr ) \biggr )^{q/r} \, \frac{dt}{t} \Biggr )^{1/q} \Biggr \|_{L^p} \\
    &\lesssim \Biggl \| \Biggl ( \sum_{j \in \Z} \int_{2^{-j}}^{2^{-j +1}} \biggl ( \sum_{\ell \in \Z} 2^{-|\ell - j|(a - |\ord|)q} \mathcal{M} \Bigl (2^{\ell \ord r} \int_{2^{-\ell}}^{2^{-\ell+1}}  |f \ast \crk_{\tau}|^r \, \frac{d\tau}{\tau} \Bigr )\biggr )^{q/r} \, \frac{dt}{t} \Biggr )^{1/q} \Biggr \|_{L^p}\\
   & \lesssim 
    \Biggl \| \Biggl ( \sum_{j \in \Z}  \biggl ( \sum_{\ell \in \Z} 2^{-|\ell - j|(a - |\ord|)q} \mathcal{M} \Bigl (2^{\ell \ord r} \int_{2^{-\ell}}^{2^{-\ell+1}}  |f \ast \crk_{\tau}|^r \, \frac{d\tau}{\tau} \Bigr )\biggr )^{q/r}  \Biggr )^{1/q} \Biggr \|_{L^p}\\
    &=\Biggl \|\Biggl ( \sum_{\ell \in \Z} 2^{-|\ell - j|(a - |\ord|)q} \mathcal{M} \Bigl (2^{\ell \ord r} \int_{2^{-\ell}}^{2^{-\ell+1}}  |f \ast \crk_{\tau}|^r \, \frac{d\tau}{\tau} \Bigr ) \Biggr )_{j=-\infty}^\infty \Biggr \|_{L^{p/r}(\ell^{q/r})}^{1/r} \\
    &\lesssim \Biggl \|\Biggl ( \mathcal{M} \Bigl (2^{\ell \ord r} \int_{2^{-\ell}}^{2^{-\ell+1}}  |f \ast \crk_{\tau}|^r \, \frac{d\tau}{\tau} \Bigr ) \Biggr )_{\ell=-\infty}^\infty \Biggr \|_{L^{p/r}(\ell^{q/r})}^{1/r},
\end{align*}
where the last inequality used the estimate \eqref{eq:Ry_Lp_lq}, which is applicable whenever $a > |\ord|$, as we assumed in the statement.
Because of $0 < r < p \land q$, we may apply the Fefferman-Stein inequalities (Lemma \ref{vector-valued HL}) to conclude
\begin{align*}
    \biggl\| \biggl( \int_0^\infty t^{-\ord q} \bigl[(\crk_t^{**} f)_a \bigr]^q \;\frac{dt}{t}\biggr)^{1/q}\biggr\|_{L^p} 
    &\lesssim \Biggl \|\Biggl (  2^{\ell \ord r} \int_{2^{-\ell}}^{2^{-\ell+1}}  |f \ast \crk_{\tau}|^r \, \frac{d\tau}{\tau}  \Biggr )_{\ell=-\infty}^\infty \Biggr \|_{L^{p/r}(\ell^{q/r})}^{1/r}\\
     &=\Biggl \| \Biggl ( \sum_{\ell \in \Z}  2^{\ell \ord q} \biggl (\int_{2^{-\ell}}^{2^{-\ell+1}}  |f \ast \crk_{\tau}|^r \, \frac{d\tau}{\tau} \biggr )^{q/r}\Biggr )^{1/q} \Biggr \|_{L^p}\\
     & \lesssim \Biggl \| \Biggl ( \sum_{\ell \in \Z}  2^{\ell \ord q}  \int_{2^{-\ell}}^{2^{-\ell+1}}  |f \ast \crk_{\tau}|^q \, \frac{d\tau}{\tau}  \Biggr )^{1/q} \Biggr \|_{L^p} \\  
    &\asymp \biggl \| \biggl ( \int_{0}^{\infty}  t^{-\ord q} |f \ast \crk_{\tau}|^q \, \frac{d\tau}{\tau} \biggr )^{1/q} \biggr \|_{L^p},
\end{align*}
where we the third step used H\"{o}lder's inequality:
\begin{align*}
\int_{2^{-\ell}}^{2^{-\ell +1}} |f \ast \crk_\tau|^r \, \frac{d\tau}{\tau}& \leq \left(\int_{2^{-\ell}}^{2^{-\ell +1}} |f \ast \crk_\tau|^q \, \frac{d\tau}{\tau}\right)^{r/q} \left(\int_{2^{-\ell}}^{2^{-\ell +1}}  1 \, \frac{d\tau}{\tau} \right)^{(q-r)/q}  \\
&\lesssim \left(\int_{2^{-\ell}}^{2^{-\ell +1}} |f \ast \crk_\tau|^q \, \frac{d\tau}{\tau}\right)^{r/q} .
\end{align*}

For $q = \infty$ we can repeat the proof almost verbatim if we replace the integral in $dt/t$ by the essential supremum and use \eqref{eq:aux_est_L1_Linfty} instead of H\"{o}lder's inequality at the end.
\\~\\ 
\textit{Case (iii).}
The case (iii) follows from an argument almost identical to those proving case (i) for both $q < \infty$ and $q = \infty$ by instead using the quasi-norms
\begin{align}
\begin{split}
\label{eq:aux_quasi_norm_lq_Lp}
    \biggl ( \int_0^\infty t^{-\ord q} \|  \, \cdot \, \|_{L^p}^{q} \frac{dt}{t} \biggr )^{1/q} &= \biggl ( \sum_{j \in \Z} \int_{2^{-j}}^{2^{-j +1}} t^{-\ord q} \|  \, \cdot \, \|_{L^p}^{q} \frac{dt}{t} \biggr )^{1/q}, \\
    \Bigl \| t^{-\ord} \| \, \cdot \, \|_{L^p}  \bigr \|_{L^\infty((0, \infty), \frac{dt}{t})} &= \sup_{j \in \Z} \Bigl \| t^{-\ord} \| \, \cdot \, \|_{L^p} \bigr \|_{L^\infty([2^{-j}, 2^{-j+1}], \frac{dt}{t})},
\end{split}
\end{align}
respectively, and the inequality \eqref{eq:Ry_lq_Lp}. The use of Lemma~\ref{vector-valued HL} in this case only requires $r < p$, which, in view of $ar > \hdim$, is satisfied if and only if $a > \frac{\hdim}{p}$.
\\~\\
\textit{Case (ii).}
If $q \in (0, \infty)$, we first observe that by using the properties (i) and (ii) of
the set of dyadic balls $\mathcal{B} = \{B_m^k: j_0 \in \mathbb{Z}, m \in \mathbb{N}\}$ stated in
Lemma \ref{dyba}, it follows that
\begin{align}
\begin{split} \label{eq:cont_char_quasi_norm_Step3_q}
\sup_{y \in \N, t > 0} \biggl( \dashint_{B_t(y)} \int_0^t
\bigl|& \tau^{-\ord} [\, \cdot \,](x) \bigr|^q \; \frac{d\tau}{\tau}  d\mu_\N(x)\biggr)^{1/q}  \\
&\lesssim \sup_{m \in \NN, j_0 \in \mathbb{Z}} \biggl( \dashint_{B^{j_0}_m} \int_0^{2^{j_0 +1}}
\bigl| \tau^{-\ord} [\, \cdot \,](x) \bigr|^q \; \frac{d\tau}{\tau} d\mu_\N(x)\biggr)^{1/q}  \\
&= \sup_{m \in \NN, j_0 \in \mathbb{Z}} \biggl( \dashint_{B^{j_0}_m} \sum_{j=-j_0}^\infty \int_{2^{-j}}^{2^{-j+1}}
\bigl| \tau^{-\ord} [\, \cdot \,](x) \bigr|^q \; \frac{d\tau}{\tau} d\mu_\N(x)\biggr)^{1/q}  \\
&\asymp \sup_{m \in \NN, j_0 \in \Z} \biggl( \dashint_{B^{j_0}_m} \sum_{j = -j_0}^{\infty}
\bigl| 2^{j \ord} [\, \cdot \,](x) \bigr|^q \; d\mu_\N(x)\biggr)^{1/q} .  
\end{split}
\end{align} 
Using this estimate, together with \eqref{eq:cont_char_crucial_est_Step3}, we  find that
\begin{align*}
    \sup_{y \in \N, t > 0} \biggl(& \dashint_{B_t(y)} \int_0^t
    \bigl| \tau^{-\ord} (\crk^{**}_{\tau} f)_{a} (x) \bigr|^q \; \frac{d\tau}{\tau}  d\mu_\N(x)\biggr)^{1/q} \\
    &\lesssim \sup_{j_0 \in \Z, m \in \NN} \Biggl( \dashint_{B^{j_0}_m} \sum_{j = -j_0}^{\infty}
     2^{-j \ord q} \biggl ( \sum_{\ell \in \Z} 2^{-|\ell - j|ar} \\
     &\hspace{55pt} \times \int_{2^{-\ell}}^{2^{-\ell+1}} \int_{\N} \frac{\tau^{-\hdim} |f \ast \crk_{\tau}(z)|^r}{(1+\tau^{-1}|z^{-1}x|)^{ar}} \, d\mu_\N (z) \, \frac{d\tau}{\tau} \biggr )^{q/r} \, d\mu_\N(x)\Biggr)^{1/q} \\
     &\lesssim \sup_{j_0 \in \Z, m \in \NN} \Biggl( \dashint_{B^{j_0}_m} \sum_{j = -j_0}^{\infty}
      \biggl ( \sum_{\ell \in \Z} 2^{-|\ell - j|(a - |\ord|)q} 2^{\ell \ord q} \\
     &\hspace{55pt} \times \int_{2^{-\ell}}^{2^{-\ell+1}} \int_{\N} \frac{\tau^{-\hdim} |f \ast \crk_{\tau}(z)|^r}{(1+\tau^{-1}|z^{-1}x|)^{ar}} \, d\mu_\N (z) \, \frac{d\tau}{\tau} \biggr )^{q/r} \, d\mu_\N(x)\Biggr)^{1/q}.
\end{align*}
Assuming that $a > |\ord| \geq 0$, we now apply \eqref{eq:bounded on Cq} to majorize this expression by the simpler 
\begin{align*}
    \sup_{\substack{\ell_0 \in \Z \\ m \in \NN}} \Biggl( \dashint_{B^{\ell_0}_m} \sum_{\ell = -\ell_0}^{\infty} 2^{\ell \ord q}
    \biggl ( \int_{2^{-\ell}}^{2^{-\ell+1}} \hspace{-7pt} \int_{\N} \frac{\tau^{-\hdim} |f \ast \crk_{\tau}(z)|^r}{(1+\tau^{-1}|z^{-1}x|)^{ar}} \; d\mu_\N (z)
    \frac{d\tau}{\tau} \biggr )^{q/r} \, d\mu_\N(x)\Biggr)^{1/q},
\end{align*}
where the sufficient condition \eqref{supinf} from Lemma~\ref{infi}, which allows us to employ \eqref{eq:bounded on Cq}, is satisfied due to an estimate that is almost identical to
\eqref{eq:justification_supinf}.
The desired estimate follows now from an application of Lemma~\ref{intt}~(ii) if we additionally assume that $a > \frac{ar}{q} > \frac{\hdim}{q}$:
\begin{align*}
     \sup_{\substack{\ell_0 \in \Z \\ m \in \NN}} \Biggl( \dashint_{B^{\ell_0}_m} &\sum_{\ell = -\ell_0}^{\infty}  2^{\ell \ord q}
    \biggl ( \int_{2^{-\ell}}^{2^{-\ell+1}} \hspace{-7pt} \int_{\N} \frac{\tau^{-\hdim} |f \ast \crk_{\tau}(z)|^r}{(1+\tau^{-1}|z^{-1}x|)^{ar}} \; d\mu_\N (z)
    \frac{d\tau}{\tau} \biggr )^{q/r} \, d\mu_\N(x)\Biggr)^{1/q} \\
    &\lesssim \sup_{\substack{\ell_0 \in \Z \\ m \in \NN}} \biggl( \dashint_{B^{\ell_0}_m} \sum_{\ell = -\ell_0}^{\infty} 2^{\ell \ord q}
     \int_{2^{-\ell}}^{2^{-\ell+1}} \hspace{-10pt} |f \ast \crk_{\tau}(z)|^q \, \frac{d\tau}{\tau}  d\mu_\N(x)\biggr)^{1/q} \\
    &\asymp
    \sup_{y \in \N, t > 0} \biggl( \dashint_{B_t(y)} \int_0^t
    \bigl| \tau^{-\ord} f * \crk_{\tau} (x) \bigr|^q \; \frac{d\tau}{\tau}  d\mu_\N(x)\biggr)^{1/q},
\end{align*}
which is the desired claim. For the last asymptotic equality we have used that due to Lemma~\ref{dyba} there exists a $C > 0$ such that any $y \in \N$, $t > 0$ we find a dyadic ball ${B}^{\ell_0}_m$ in our fixed collection \eqref{Bkl} that is close enough to $B_t(y)$ and of comparable radius so that for any $g \in L^1_{loc}(\N)$
\begin{align*}
    C^{-1} \dashint_{B_t(y)} |g(x)| \; d\mu_\N(x) \leq \dashint_{B^{\ell_0}_m} |g(x)| \; d\mu_\N(x) \leq C \dashint_{B_t(y)} |g(x)| \; d\mu_\N(x).
\end{align*}

We only sketch the very similar proof for $q = \infty$. Starting again from \eqref{eq:cont_char_crucial_est_Step3}, we employ the quasi-norm equivalence
\begin{align}
&  \sup_{x \in \N, t > 0} \; \sup_{\tau \in (0, t]} \tau^{-\ord} \dashint_{B_t(x)}
     \int_{2^{-\ell}}^{2^{-\ell+1}} \int_{\N} \frac{\tau^{-\hdim} |f \ast \crk_{\tau}(z)|^r}{(1+\tau^{-1}|z^{-1}y|)^{ar}} \, d\mu_\N (z)  \frac{d\tau}{\tau} d\mu_\N(y) \nonumber \\
    \asymp & \sup_{\ell_0 \in \Z, m \in \NN} \; \sup_{\tau \in (0, 2^{-\ell_0+1}]} \tau^{-\ord} \dashint_{{B}^{\ell_0}_m}
     \int_{2^{-\ell}}^{2^{-\ell+1}} \int_{\N} \frac{\tau^{-\hdim} |f \ast \crk_{\tau}(z)|^r}{(1+\tau^{-1}|z^{-1}y|)^{ar}} \, d\mu_\N (z)  \frac{d\tau}{\tau}  d\mu_\N(y) \nonumber \\
    \asymp & \sup_{\ell_0 \in \Z, m \in \NN} \; \sup_{\ell \geq -\ell_0 } \; 2^{-\ell \ord} \dashint_{{B}^{\ell_0}_m}
      \int_{2^{-\ell}}^{2^{-\ell+1}} \int_{\N} \frac{\tau^{-\hdim} |f \ast \crk_{\tau}(z)|^r}{(1+\tau^{-1}|z^{-1}y|)^{ar}} \, d\mu_\N (z) 
    \frac{d\tau}{\tau}  d\mu_\N(y). \label{eq:cont_char_quasi_norm_Step3_infty}
\end{align}
Note that majorizing the supremum by the integral above is justified by an estimate of the type \eqref{eq:justification_supinf} and resembles the estimate \eqref{eq:estimate_sup_integral} in the proof of Lemma~\ref{infi}. The use of Lemma~\ref{intt}~(ii) in this case requires $a > ar > \hdim$, but again our choice of $r$ does not restrict the assumption $a > \max \{ \hdim, |\ord| \}$. This completes the proof of Step~3.
\\~\\
\textbf{Step 4.}
We finally show that the quasi-norms in (i) - (iii) involving the continuous Calder\'on condition (with respect to $\crk$) can be estimated by the quasi-norms involving discrete Calder\'on condition (with respect to $\crkk$).
As in in Step~1, 
\begin{align*}
        f \ast \crk_t = \sum_{j \in \mathbb{Z}} f \ast \crkk_{2^{-j}} \ast \drkk_{2^{-j}} \ast \crk_t, \quad t > 0,
\end{align*}
with convergence in  $\SV'(\N)$. Employing a dyadic partition for the domain of integration in $t$, for which we index the intervals by $k \in \Z$, we can now proceed precisely as in Step~1 of the proof of Theorem~\ref{thm:indep_crk}. Making use of Lemma~\ref{aoe}, we thus get for any $t \in [2^{-k}, 2^{-k+1}]$, with $k \in \Z$, and any $x \in \N$ that
\begin{align*}
    t^{-\ord} |f \ast \crk_t (x)| 
    &\leq 2^{k \ord} \sum_{j \in \mathbb{Z}} \int_\N |f \ast \crkk_{2^{-j}}(y)| |\drkk_{2^{-j}} \ast \crk_t (y^{-1}x)| \; d\mu_\N (y) \\
    & \lesssim \sum_{j \in \mathbb{Z}} 2^{k \ord} 2^{-|j-k|M}\int_\N |f \ast \crkk_{2^{-j}}(y)|\frac{2^{(j \wedge k)\hdim}}{(1 +2^{j\wedge k}|y^{-1}x|)^{L}}\; d\mu_\N (y) \\
    &\lesssim \sum_{j \in \mathbb{Z}} 2^{-|j -k |(M-a-|\ord|)}2^{j\ord} (\crkk^*_{2^{-j}} f)_{a}(x) \int_{\N} \frac{2^{(j \wedge k)\hdim}}{(1 + 2^{j\wedge k}|y^{-1}x|)^{Q + \varepsilon}} \; d\mu_\N (y)\\ 
    &\lesssim \sum_{j \in \mathbb{Z}} 2^{-|j -k |(M-a-|\ord|)} 2^{j\ord}(\crkk^*_{2^{-j}} f)_{a}(x),
\end{align*}
provided we choose $L = a + \hdim + \varepsilon$ for an arbitrary but fixed $\varepsilon > 0$. Taking the relevant quasi-norms on both sides of the estimate above, together with an application of \eqref{eq:Ry_Lp_lq}, \eqref{eq:bounded on Cq} or \eqref{eq:Ry_lq_Lp} yields the claim (see Corollary~\ref{cor:PM_qn_equiv}).
\end{proof}

\section{Wavelet transform characterizations} \label{sec:wavelet}
This section is devoted to characterizing the Besov and Triebel-Lizorkin spaces by continuous wavelet transforms.

\subsection{Wavelet transforms on solvable extensions}
Throughout this section, let $A = \mathbb{R}^+$. Then $A$ acts on $N$ via the dilations $t \mapsto \delta_t$. The associated semi-direct product group $\NA = \N \rtimes A$ is defined via the product law
\[
 (x,s) (y,t) = (x \, \delta_s (y) , st), \quad  (x,s),  (y,t) \in G,
\]
and inversion $ (x,s)^{-1} = \bigl ( \delta_{s^{-1}}( x^{-1}), s^{-1} \bigr )$. The identity element in $G$ is given by $e_G = (e_N, 1)$ and the left-invariant Haar measure (up to a multiplicative constant) by $d\mu_G (x, s) = d\mu_\N(x) \, \frac{ds}{s^{\hdim +1}}$. The modular function of $G$ is given by $\Delta_G (x,s) = s^{-Q}$ for $(x,s) \in G$. 
The group $G$ is an exponential solvable Lie group, i.e., a connected, simply connected solvable Lie group whose exponential map is a diffeomorphism, cf. \cite[Proposition 5.27]{fuehr2005abstract}. As $G$ is nonunimodular, it necessarily has exponential volume growth.  

The \emph{quasi-regular representation} of $G$ is the map $\pi : G \to \mathcal{U}(L^2 (N))$ defined by
\[
\pi(x,s) f (y) =  s^{-\hdim/2} f(s^{-1} (x^{-1} y)), \quad  (x,s) \in G, \; y \in \N,   
\]
forms a unitary group representation. For fixed $\psi \in L^2 (N)$, the associated \emph{wavelet transform} $V_{\psi} : L^2 (N) \to L^{\infty}(G)$ is defined by
\[
V_{\psi} f (g) = \langle f, \pi(g) \psi \rangle, \quad g \in G.
\]
Throughout, we will mostly consider the wavelet transform associated to a real even function $\psi \in \mathcal{S}_0 (N)$ satisfying 
\begin{align} \label{eq:l2calderon}
f = \int_0^{\infty} f \ast \psi_t \ast \psi_t \; \frac{dt}{t} \quad \quad \text{for all} \quad f \in \mathcal{S}(N). 
\end{align}
with convergence in $L^2 (\N)$.
 The existence of such vectors is guaranteed by (the proof of) Proposition \ref{prop:construction_crk} and \cite[Corollary 7.2]{glowacki2013lp}; see also \cite[Lemma 4.1]{velthoven2022comptes}. We will refer to such a function as being \emph{admissible}.

The significance of an admissible function is that its wavelet transform is an isometry from $L^2 (\N)$ into $L^2 (\NA)$, because
\begin{align}
\begin{split} \label{eq:repro}
\| f \|^2_{L^2(\N)} &= \int_{0}^{\infty} \langle f \ast \drk_t \ast \drk_t, f \rangle \; \frac{dt}{t} = \int_0^{\infty} \int_N | f \ast \psi_t (x) |^2 \; d\mu_N (x) \frac{dt}{t}   \\
&= \int_N \int_0^{\infty} |V_{\psi} f (x,t) |^2 \; d\mu_N (x) \frac{dt}{t^{\hdim+1}} = \| V_{\psi} f \|^2_{L^2 (G)}
\end{split}
\end{align}
for all $f \in \mathcal{S}(N)$, where the penultimate step used that $V_{\drk} f (x,t) = t^{\hdim/2} (f \ast \overline{\drk_t}^{\vee}) (x)$ and that $\crk$ is real and even. Moreover, as in the proof of Proposition \ref{prop:construction_crk}, the reproducing formula \eqref{eq:l2calderon} extends to all of $f \in \SV'(N)$, with convergence in $\SV'(N)$.
Thus, any admissible vector satisfies the continuous Calder\'on condition \eqref{eq:continuous_calderon}. 

We next extend the definition of the wavelet transform to distribution spaces. For this, note that the action of the quasi-regular representation $\pi$ defines a continuous representation on $\mathcal{S}_{0} (N)$, see, e.g., \cite[Lemma 3.1]{christensen2012coorbit}. In particular, 
$\SV(N)$ is $\pi$-invariant. 
If $f \in \SV' (\N)$ and $\psi \in \SV (\N)$, then we also write \[
V_{\psi} f (g) = \langle f, \pi (g) \psi \rangle,\] 
where the pairing $\langle \cdot, \cdot \rangle : \SV' (\N) \times \SV (\N) \to \mathbb{C}$ denotes the anti-linear dual pairing. 

For extending the reproducing formula \eqref{eq:repro} to distributions, we will use the following pointwise estimates of the wavelet transform. The precise statement used here is \cite[Proposition 3.1]{velthoven2022comptes}; see also Lemma \ref{aoe} for closely related estimates.

\begin{lemma} \label{lem:pointwise}
Let $m, m' \in \mathbb{N}$ be arbitrary. Suppose that $f, \psi \in \mathcal{S}_0 (\N)$. Then the following assertions hold:
\begin{enumerate}
[\rm (i)]\item For all $(x,s) \in G$ with $s \leq 1$,
\[
|V_{\psi} f (x,s)| \lesssim s^{Q/2 + m} (1+|x|)^{-m'}.
\]
\item For all $(x,s) \in G$ with $s \geq 1$,
\[
|V_{\psi} f (x,s) |\lesssim s^{-(Q/2 + m)} (1+|x|)^{-m'}.
\]
\end{enumerate}
The implicit constants depend on $G, m, m', f, \psi$.
\end{lemma}

The following lemma provides the desired reproducing formula.

\begin{lemma} \label{lem:repro_sinfty}
Let $\psi \in \SV (\N)$ be admissible. Then 
\begin{align} \label{eq:repro_extended}
 \int_G V_{\psi} f_1 (x,s) \overline{V_{\psi} f_2 (x,s)} \; d\mu_G (x,s) = \langle f_1, f_2 \rangle
\end{align}
for all $f_1 \in \SV' (\N)$ and $f_2 \in \SV (\N)$.
\end{lemma}
\begin{proof}
The proof is similar to that of Proposition \ref{prop:construction_crk}, hence only sketched.

For $k  \in \mathbb{N}$ with $k \geq \hdim + 1$, let $\mathcal{S}^{k} (N)$ be the Banach space as defined in \eqref{eq:Sk}.
Then $\mathcal{S} (N) \hookrightarrow \mathcal{S}^{k}(N)$ and $C^k (N) \hookrightarrow L^2 (N)$.
In particular, this implies that the vector-valued map $g \mapsto \pi(g) \psi$ is continuous from $G$ into $\mathcal{S}^k(\N)$. This, together with Lemma \ref{lem:pointwise}, easily yields that the map 
$g \mapsto V_{\psi} f_2 (g) \pi (g) \psi$ is Bochner integrable. 
Therefore, the integral formula  
\begin{align} \label{eq:repro_proof}
f_2 = \int_G V_{\psi} f_2 (g) \pi(g) \psi \; d\mu_G (g),
\end{align} 
holds as a Bochner integral in $\mathcal{S}^k (N)$. 

For showing \eqref{eq:repro_extended}, recall that $\SV'(N)  \cong \SIP$, and hence an application of the Hahn-Banach theorem provides a continuous extension of $f_1$ to $\mathcal{S}(N)$. Therefore, by \cite[Equation (3.45)]{FR}, there exist $k \geq \hdim + 1$ such that the restriction $f_1|_{\SV}$ is continuous with respect to $\| \cdot \|_{(k)}$. 
Another application of the Hahn-Banach theorem therefore yields that $f_1$ can be extended to a continuous linear functional $f_1'$ on $\mathcal{S}^k(N)$. Thus, evaluating $f_1$ on $f_2$ using \eqref{eq:repro_proof} gives
\[
\langle f_1, f_2 \rangle = \langle f_1', \overline{f_2} \rangle
= \int_G \overline{ V_\psi f_2 (g) } V_\psi f_1 (g) \; d\mu_G (g),
\]
where we interchanged the order of functional evaluation and Bochner integration, see, e.g., \cite[Section 1.2]{hytonen2016analysis}.
\end{proof}

For technical reasons, we will repeatedly use the following extended pairing between elements of $\SV'(\N)$ and $L^2(N)$ in the remainder.

\begin{definition} \label{def:extendedpairing}
Let $\psi \in \SV(\N)$ be admissible. The \emph{extended pairing} between $f_1 \in \SV'(\N)$ and $f_2 \in L^2 (N)$ is defined by
\[
\langle f_1, f_2 \rangle_{\psi} :=  \int_G V_{\psi} f_1 (g) \overline{V_{\psi} f_2 (g)} \; d\mu_G (g)
\]
provided the integral exists.
\end{definition}

Observe that for $f_1 \in \SV'(\N)$ and $f_2 \in \SV(\N)$, the extended pairing $\langle f_1, f_2 \rangle_{\psi}$ coincides with the usual conjugate linear pairing $\langle f_1, f_2 \rangle$ by Lemma \ref{lem:repro_sinfty}. In addition, if $f_1, f_2 \in L^2 (N)$, then $\langle f_1, f_2 \rangle_{\psi}$ coincides with the inner product $\langle f_1, f_2 \rangle$.

\subsection{Function spaces with mixed norms} \label{sec:mixednorm}
The following function spaces form an analogue of the Peetre-type spaces on $\R^d$ considered in \cite{rauhut2011generalized, koppensteiner2023anisotropic1, koppensteiner2023anisotropic2}, and will form a key ingredient in the wavelet transform characterization of Triebel-Lizorkin spaces. 

\begin{definition} \label{DefPTS}
Let $\ord \in \mathbb{R}$ and let $\PTpar \in (0, \infty)$. For $p \in (0, \infty)$, $q \in (0, \infty]$, the space $\PT(G)$ is defined to be the set of all (equivalence classes of) measurable functions $F: \NA \to \mathbb{C}$ such that
\begin{align*}
\| F \|_{\PT} := \Biggl \|  \biggl \| t^{-\ord} \; \esssup_{z \in \N} \frac{|F(z, t)|}{(1 + t^{-1} |z^{-1}  (\cdot)|)^\PTpar} \biggr \|_{L^q((0, \infty), \frac{dt}{t^{\hdim + 1}})} \Biggr \|_{L^p(\N)} < \infty. 
\end{align*}
For $p = \infty$ and $q \in (0, \infty]$, the space $\PTi(G)$ is defined to consist of all measurable $F : G \to \mathbb{C}$ such that
\begin{align*}
\| F \|_{\PTi} := \esssup_{x \in \N, t > 0} \Biggl (  \dashint_{B_t(x)} \int_0^t \Bigl ( s^{-\ord} \; \esssup_{z \in \N} \frac{|F(z, s)|}{(1 + s^{-1} |z^{-1} y|)^{\PTpar}} \Bigr )^q \, \frac{ds}{s^{\hdim + 1}} \, d\mu_\N(y) \Biggr )^{1/q} < \infty
\end{align*}
if $q < \infty$, and
\begin{align*}
\| F \|_{\PTii} := \esssup_{x \in \N, t > 0} \; \Biggl \| \dashint_{B_t(x)} s^{-\ord} \; \esssup_{z \in \N} \frac{|F(z, s)|}{(1 + s^{-1} |z^{-1} y|)^{\PTpar}} \, d\mu_\N(y) \Biggr \|_{L^\infty((0, t), \frac{ds}{s^{\hdim+1}})} < \infty,
\end{align*}
otherwise.
\end{definition}

The spaces just introduced satisfy the following basic properties, which are standing assumptions of the abstract theory \cite{velthoven2022quasi, feichtinger1989banach} to be employed in the next subsection. Here, we recall that a quasi-norm $\| \cdot \|$ is said to be an $r$-\emph{norm} ($r \in (0,1]$) if $\| F_1 + F_2 \|^r \leq \| F_1\|^r + \| F_2\|^r$.
The proof is very similar to analogous results on Euclidean spaces proven in \cite{ullrich2012continuous,koppensteiner2023anisotropic1, koppensteiner2023anisotropic2}, and hence skipped.

\begin{lemma} \label{lem:OpNormsPTS}
Let $\ord \in \mathbb{R}$, $\PTpar \in (0, \infty)$. For $p, q \in (0, \infty]$, let $r := \min \{1, p, q\}$. Then the space $\PT(\NA)$ is a solid quasi-Banach function space with $r$-norm. Moreover, $\PT(\NA)$ is invariant under left translation $L_{g} F:= F(g^{-1} \cdot)$ and right translation $R_{g} F := F(\cdot g)$, with operator norm estimates
\begin{align*}
    \| L_{(y, t)} \|_{\op} &= t^{\frac{\hdim}{p} - \frac{\hdim}{q} - \ord},  \\
    \| R_{(y, t)} \|_{\op} &\lesssim t^{\ord + \frac{\hdim}{q}} \, \max\{ 1, t^{-a} \}  \, (1 + |y|)^\PTpar 
\end{align*}
 for $(y, t) \in G$ whenever $p < \infty$,
and  \begin{align*}
    \| L_{(y, t)} \|_{\op} &= t^{\frac{\hdim}{p} - \frac{\hdim}{q} - \ord}, \\
    \| R_{(y, t)} \|_{\op} &\lesssim t^{\ord + \frac{\hdim}{q}} \, \max\{ 1, t^{-a} \} \, \max\{ 1, t^{\hdim} \} \, (1 + |y|)^\PTpar 
\end{align*}
whenever $p = \infty$.
\end{lemma}

For the wavelet transform characterization of Besov spaces, we will make use of the mixed normed spaces defined next. 

\begin{definition} \label{def:MTS}
Let $\ord \in \mathbb{R}$ and let $p, q \in (0, \infty]$. Then the space $\MT(\NA)$ is defined to be the set of all (equivalence classes of) measurable functions $F: \NA \to \mathbb{C}$ such that 
\begin{align*}
\| F \|_{\MT} := \bigg \| t^{-\ord} \bigg \|  \esssup_{z \in \N}\frac{|F(z, t)|}{(1 + t^{-1}|z^{-1}(\cdot)|)^\PTpar} \bigg \|_{L^p(\N)} \bigg \|_{L^q((0, \infty), \frac{dt}{t^{\hdim + 1}})} < \infty.
\end{align*}
\end{definition}

As the proof of Lemma \ref{lem:OpNormsPTS}, we skip the proof of the following lemma.

\begin{lemma} \label{lem:OpNormsMTS}
Let $\ord \in \mathbb{R}$, $a \in (0, \infty)$. For $p, q \in (0, \infty]$, let $r := \min \{1,p,q\}$. Then the space $\MT(\NA)$ is a solid quasi-Banach function space with $r$-norm. Moreover, $\MT(\NA)$  is invariant under left translation $L_{g} F:= F(g^{-1} \cdot)$ and right translation $R_{g} F := F(\cdot g)$, with operator norm estimates
\begin{align}
    \| L_{(y, t)} \|_{\op} &= t^{\frac{\hdim}{p} - \frac{\hdim}{q} - \ord} \hspace{5pt}, \label{eq:translation_norms_MTS_1} \\
    \| R_{(y, t)} \|_{\op} &\lesssim t^{\ord + \frac{\hdim}{q}} \, \max\{ 1, t^{-\PTpar} \} \, (1 + |y|)^\PTpar, \label{eq:translation_norms_MTS_2}
\end{align}
for $(y, t) \in G$.
\end{lemma} 

In the remainder, we will prove various results that hold simultaneously for the spaces $\PT$ and $\MT$ defined in Definition \ref{DefPTS} and Definition \ref{def:MTS}, respectively. For this reason, in order to ease notation, we shall also write $\ST$ for either $\PT$ or $\MT$ when no distinction is required.

\subsection{Norm characterizations} The following lemma characterizes the Besov and Triebel-Lizorkin norms in terms of a local mean property of the continuous wavelet transform. Its statement is close to that of (part of) Theorem \ref{thm:cont_char}, but the precise statement of the following characterization is essential for obtaining wavelet decompositions in the next section.

\begin{lemma} \label{lem:wavelet_norm}
Let $\psi \in \SV (N)$ be an admissible vector and let $U := B_{1/2} (e_N) \times (1/2, 2)$. If $\PTpar > \max \{ \frac{\hdim}{p \wedge q}, |\ord| \}$, then for $\ord' := \ord + Q/2 - Q/q$, the norm equivalences
\[
\| f \|_{\TLS} \asymp \big\| V_{\psi} f \big\|_{\PTv{p}{q}{\ord'}} \asymp \bigg\|g \mapsto \sup_{u \in U} \big|V_{\psi} f (g u)\big| \bigg\|_{\PTv{p}{q}{\ord'}}
\]
and 
\[
\| f \|_{\BS} \asymp \big\| V_{\psi} f \big\|_{L^{p,q}_{\PTpar, \ord'}} \asymp \bigg\|g \mapsto \sup_{u \in U} \big|V_{\psi} f (g u)\big| \bigg\|_{L^{p,q}_{\PTpar, \ord'}}
\]
hold for all $f \in \SV'(N)$.
\end{lemma}
\begin{proof}
We only prove the claim for $\TLS(\N)$ with $p,q \in (0,\infty)$. The remaining cases are shown similarly. For simplicity, set $M^{L} V_{\psi} f (g) := \sup_{u \in U} |V_{\psi} f (g u)|$ for $g \in G$. 

We will repeatedly use the identity 
$
V_{\psi} f (g)  = t^{Q/2} (f \ast \psi_t )(y)
$ for $g = (y,t) \in G$.

By the first equivalence in Equation \eqref{eq:cont_char_1},
\begin{align*}
    \| f \|_{\TLS} 
    &\asymp \bigg\| \bigg( \int_0^{\infty} \bigg( t^{-\ord}  (\psi_t^* f)_{\PTpar} \bigg)^q \; \frac{dt}{t} \bigg)^{1/q} \bigg\|_{L^p} \\
    &= \bigg\| \bigg( \int_0^{\infty} \bigg( t^{-\ord} \sup_{y \in N} \frac{|f \ast \psi_t (y)|}{(1 + t^{-1} |y^{-1} \cdot |)^{\PTpar}} \bigg)^q \frac{dt}{t} \bigg)^{1/q} \bigg\|_{L^p} \\
    &= \bigg\| \bigg( \int_0^{\infty} \bigg( t^{-\ord - \hdim/2 + \hdim/q} \sup_{y \in N} \frac{|V_{\psi} f(y, t)|}{(1 + t^{-1} |y^{-1} \cdot |)^{\PTpar}} \bigg)^q \frac{dt}{t^{\hdim + 1}} \bigg)^{1/q} \bigg\|_{L^p} \\
    &= \| V_{\psi} f\|_{\PTv{p}{q}{\ord'}}, 
\end{align*}
and hence $\| f \|_{\TLS} \asymp \| V_{\psi} f \|_{\PTv{p}{q}{\ord'}} \leq \| M^{L} V_{\psi} f \|_{\PTv{p}{q}{\ord'}}$ since $V_{\psi} f \leq M^L V_{\psi} f$  on $G$. 

For the reverse inequality, note that 
\begin{align*}
    \| M^{L} V_{\psi} f \|_{\PTv{p}{q}{\ord'}} &= \bigg\| \bigg( \int_0^{\infty} \bigg( t^{-\ord - Q/2 + Q/q} \sup_{\substack {y \in N \\ (z,s) \in U}} \frac{|V_{\psi} f  (y \delta_t (z), st) |}{(1 + t^{-1} |y^{-1} \cdot |)^{\PTpar}} \bigg)^q \frac{dt}{t^{Q+1}} \bigg)^{1/q} \bigg\|_{L^p} \\
    &=  \bigg\| \bigg( \int_0^{\infty} \bigg( t^{-\ord - Q/2 + Q/q} \sup_{\substack {y \in N \\ (z,s) \in U}} \frac{|V_{\psi} f  (y , st) |}{(1 + t^{-1} | \delta_t (z) y^{-1} \cdot |)^{\PTpar}} \bigg)^q \frac{dt}{t^{Q+1}} \bigg)^{1/q} \bigg\|_{L^p} \\
    &\lesssim \bigg\| \bigg( \int_0^{\infty} \bigg( t^{-\ord - Q/2 + Q/q} \sup_{\substack {y \in N \\ s \in (1/2, 2)}} \frac{|V_{\psi} f  (y , st) |}{(1 + t^{-1} | y^{-1} \cdot |)^{\PTpar}} \bigg)^q \frac{dt}{t^{Q+1}} \bigg)^{1/q} \bigg\|_{L^p}, \numberthis \label{eq:TL_coorbit}
\end{align*}
where the last step used the estimate 
\[\bigl (1 + t^{-1} | \delta_t (z) y^{-1} \cdot |\bigr )^{- \PTpar} \lesssim (1 + t^{-1} | y^{-1} \cdot |)^{-\PTpar} \]
for $z \in B_{1/2} (e_N)$. 
Next, for fixed $t \in (0, \infty)$ and $s \in (1/2, 2)$, 
\[
t^{-Q/2} \sup_{y \in N} \frac{|V_{\psi} f  (y , st) |}{(1 + t^{-1} | y^{-1} \cdot |)^{\PTpar}} \lesssim t^{-Q/2} \sup_{y \in N} \frac{|V_{\psi} f  (y , st) |}{(1 + (st)^{-1} | y^{-1} \cdot |)^{\PTpar}} 
\lesssim \sup_{y \in N} \frac{|f \ast \psi_{st} (y) |}{(1 + (st)^{-1} | y^{-1} \cdot |)^{\PTpar}}.
\]
Therefore, the inequality \eqref{eq:TL_coorbit} can be further estimated as 
\begin{align*}
    \| M^{L} V_{\psi} f \|_{\PTv{p}{q}{\ord'}} &\lesssim \bigg\| \bigg( \int_0^{\infty} \bigg( t^{-\ord  + Q/q} \sup_{\substack {y \in N \\ s \in (t/2, 2t)}} \frac{|f \ast {\psi}_{s} (y) |}{(1 + s^{-1} | y^{-1} \cdot |)^{\PTpar}} \bigg)^q \frac{dt}{t^{Q+1}} \bigg)^{1/q} \bigg\|_{L^p} \\
    &= \bigg\| \bigg( \int_0^{\infty} \bigg( t^{-\ord} \sup_{\substack {y \in N \\ s \in (t/2, 2t)}} \frac{|f \ast \psi_s (y)|}{(1+s^{-1} |y^{-1} \cdot |)^{\PTpar}} \bigg)^q \frac{dt}{t} \bigg)^{1/q} \bigg\|_{L^p} \\
    &= \bigg\| \bigg( \int_{0}^{\infty} t^{-\ord q} [(\psi_t^{**} f)_a]^q \; \frac{dt}{t} \bigg)^{1/q} \bigg\|_{L^p}
    \asymp \| f \|_{\TLS},
    \end{align*}
    where the last step used  the third equivalence in Equation \eqref{eq:cont_char_1}. 
\end{proof}

\section{Wavelet and molecular decompositions} \label{sec:coorbit}
This section is devoted to obtaining wavelet decompositions of the Besov and Triebel-Lizorkin spaces via the abstract theory  \cite{velthoven2022quasi, romero2021dual, feichtinger1989banach}. 

\subsection{Analyzing vectors} \label{sec:analyzing}
Throughout this section, we fix the relatively compact unit neighborhood $U = B_{1/2} (e_N) \times (-1/2, 2)$ in $\NA = \N \rtimes A$. For a function $F \in L^{\infty}_{\loc} (G)$, define the left and right local maximal functions by
\[
M^L F(g) = \esssup_{u \in U} | F(gu)| \quad \text{and} \quad M^R F(g) = \esssup_{u \in U} |F(u g)|, \quad g \in G,
\]
respectively. The two-sided maximal function is then defined by $M F := M^L M^R F = M^R M^L F$.
For $r \in (0, 1]$ and measurable $w : G \to [1,\infty)$, define the space
\[
W(L^r_w) = \bigg\{ F \in L^{\infty}_{\loc} (G) : \int_{\NA} | (MF)(g) w(g) |^r \; d\mu_G (g) < \infty \bigg\}.
\]

For applying the theory of \cite{velthoven2022quasi, feichtinger1989banach} to the Besov and Triebel-Lizorkin spaces studied in this paper, the existence of nonzero functions $\psi \in L^2 (N)$ with wavelet transform $V_{\psi} \psi \in W(L^r_w)$ is required. Such functions are often referred to as \emph{analyzing vectors}. The precise condition on the weight $w$ is as follows.
Here, we recall that we write $\ST$ for either $\PT$ or $\MT$.

\begin{definition} \label{def:control_weight}
Let $\ord \in \mathbb{R}$, $\PTpar \in (0, \infty)$, $p, q \in (0, \infty]$, and set $r := \min \{1, p, q\}$.  A measurable function $w : G \to [1, \infty)$ is called a \emph{strong control weight} for $\ST$ if it satisfies the following conditions:
\begin{enumerate}
    \item[(c1)] $w$ is submultiplicative;
    \item[(c2)] $w(g) = w(g^{-1}) \Delta_G (g^{-1})^{1/r}$ for $g \in G$;
    \item[(c3)] $\| L_{g^{-1}} \|_{\op} \leq w(g)$ for $g \in G$;
    \item[(c4)] $\| R_{g} \|_{\op} \leq w(g)$ for $g \in G$; 
\end{enumerate}
with $\| \cdot \|_{\op}$ denoting the operator norm on $\ST$.
\end{definition}

The following lemma shows the existence of a suitable control weight for the spaces introduced in Section \ref{sec:mixednorm}.

\begin{lemma} \label{lem:standard_control}
Let $p, q \in (0, \infty]$, $\PTpar \in \mathbb{R}$ and $\ord > 0$.
There exists a control weight $w : G \to [1,\infty)$ for $\ST$ satisfying $w(x,s) \lesssim (s^m + s^{-m'})(1+|x|)^k$ for some $m, m', k \in \mathbb{N}$ and all $(x,s) \in \NA$. Any such control weight will be called a \emph{standard control weight} for $\ST$.
\end{lemma}
\begin{proof}
 Let $(x, s) \in \NA$ and recall that $(x, s)^{-1} = (\delta_{s^{-1}}(x^{-1}), s^{-1})$. By Lemma \ref{lem:OpNormsPTS} and Lemma \ref{lem:OpNormsMTS}, there exists a constant $C > 1$ such that
 \[ \| L_{(x,s)^{-1}} \|_{\op} = s^{\ord + \frac{\hdim}{q} - \frac{\hdim}{p}} \]
 and 
 \[ \| R_{(x,s)} \|_{\op} \leq C s^{\ord + \frac{\hdim}{q}} \, \max\{ 1, s^{-\PTpar} \} \, \max\{ 1, s^{\hdim} \}^{\frac{p}{\infty}} \, (1 + |x|)^\PTpar \]
 for the operator norms $\| \cdot \|_{\op}$ of both $\PT(\NA)$ and $\MT(\NA)$. 
 Set 
 \[ v_1 (x,s) = C\max\{ 1, s^{\ord + \frac{\hdim}{q} - \frac{\hdim}{p}} \} \max \{1, s^{\ord + \frac{\hdim}{q}} \} \, \max\{ 1, s^{-\PTpar} \} \, \max\{ 1, s^{\hdim} \}^{\frac{p}{\infty}}  \]
 and 
 \[
 v_2 (x,s) = (1+|x|' + s^{-1} |x|' + s + s^{-1} )^a,
 \]
 where $| \cdot |'$ is a homogeneous quasi-norm satisfying \eqref{quas} for $\gamma = 1$, cf. \cite{hebisch1990smooth}.
 Then both $v_1$ and $v_2$ are measurable, submultiplicative weights, and hence so is their product $v  := v_1 \cdot v_2 $. In addition, the weight $v$ satisfies $\| L_{(x,s)^{-1}} \|_{\op}, \| R_{(x,s)} \|_{\op} \lesssim v(x,s)$. Therefore, defining
 \[
 w(x,s) = C \max \{ v(x,s), s^{\hdim / r} v((x,s)^{-1})) \}
 \]
yields a measurable weight satisfying (c1)-(c4).

 For the remaining claim, note that $v_2(x,s) = v_2((x,s)^{-1})$ and 
 \[
 v_2(x,s) \leq (1+|x|')^a (1+s + s^{-1})^a \lesssim (1+|x|')^a \max\{s, s^{-1} \}^a,
 \]
 so that
 \begin{align*}
 w(x,s) &= C \max \{ v_1(x,s), s^{\hdim / r} v_1((x,s)^{-1})) \} v_2 (x,s) \\
 &\lesssim (1+|x|')^a \max \{ v_1(x,s), s^{\hdim / r} v_1((x,s)^{-1})) \} \max\{s, s^{-1} \}^a,
 \end{align*}
 which easily completes the proof.
\end{proof}

\begin{proposition} \label{prop:better_vector}
Let $r \in (0,1]$. For $f, \psi \in \mathcal{S}_0 (N)$, it holds that $V_{\psi} f \in W(L^r_w)$, 
where $w : G \to [1, \infty)$ is any standard control weight as in Lemma \ref{lem:standard_control}. 
\end{proposition}
\begin{proof}
Using Lemma \ref{lem:pointwise}, we will estimate the (two-sided) local maximal function $M V_{\psi} f$ of $V_\psi f$ by local maximal functions of auxiliary functions $v = v_{n, n',k}$ defined by $v_{n, n',k} (x,s) = (s^{n} + s^{-n'}) ( 1+ |x|)^k$ for $n, n', k \geq 0$. Fix $(x,s) \in \NA$, and note that
\begin{align*}
(M^R v)(x,s) &= \esssup_{z \in B_{1/2}, \; u \in (1/2, 2)} \big( (us)^n + (us)^{-n'} \big) (1 + |z \delta_u (x)|)^{-k} \\
&\lesssim  \esssup_{z \in B_{1/2}, \; u \in (1/2, 2)} \big( s^n + s^{-n'} \big) (1 + | \delta_u (x^{-1}) z^{-1}|)^{-k} \\
&\leq \esssup_{z \in B_{1/2}, \; u \in (1/2, 2)} \big( s^n + s^{-n'} \big) (1 + u| x|)^{-k} (1+|z|)^{k} \\
&\lesssim  \big( s^n + s^{-n'} \big) (1 + |x|)^{-k} = v(x,s).
\end{align*}
Similarly, it follows that 
\begin{align*}
    (M^L v)(x,s) &= \esssup_{y \in B_{1/2}, \; t \in (1/2, 2)} \big( (st)^n + (st)^{-n'} \big) (1 + |x \delta_s (y)|)^{-k} \\
    &\lesssim \esssup_{y \in B_{1/2}} \big( s^n + s^{-n'} \big) (1 + |x |)^{-k} (1 + s|y|)^k \\
    &\lesssim  \big( s^{n} + s^{-n'} \big) (1 + |x |)^{-k} (1 + s)^k.
\end{align*}
On the other hand, note that an application of Lemma \ref{lem:pointwise} yields, for arbitrary $m, m', k \in \N$, that 
$
|V_\psi f (x,s)| \lesssim  v_{\hdim / 2 + m, 0, k} (x,s)$ if $s \leq 1$, and 
$|V_\psi f(x,s)| \lesssim v_{0, \hdim/2 + m', k} (x,s)$ if $s \geq 1$.
Combining the obtained estimates therefore yields
\[
(M V_\psi f)(x,s) \lesssim  s^{\hdim/2 + m}  (1 + |x |)^{-k}, \quad x \in N, \; s \leq 1,
\]
and 
\[
(M V_\psi f)(x,s) \lesssim  s^{-(\hdim/2 + m')+k}  (1 + |x |)^{-k}, \quad x \in N, \; s \geq 1.
\]

In order to estimate the $L_w^r$-norm of $MV_{\psi} f$, choose some exponents $l, l', k' \in \mathbb{N}$ satisfying $w(x,s) \lesssim (s^l + s^{-l'}) (1+|x|)^{k'}$ for all $(x,s) \in \NA$. Then
\begin{align*}
\int_G | (M V_\psi f) (g)  w(g) |^r \; d\mu_G(g) 
&= \int_{0}^{\infty} \int_N  | M V_\psi f (x,s)  w(x,s) |^r \; d\mu_N (x) \frac{ds}{s^{\hdim + 1}} \\
&\lesssim  \int_{0}^1 s^{(\hdim /2 + m - l')r} \int_N  (1+|x|)^{(k' - k)r}  \; d\mu_N (x) \frac{ds}{s^{\hdim + 1}} \\
& \quad \quad  + \int_1^{\infty} s^{(-\hdim /2 - m'+k+l)r} \int_N (1+|x|)^{(k' - k) r}  \; d\mu_N (x) \frac{ds}{s^{\hdim + 1}}.
\end{align*}
Choosing $k \in \mathbb{N}$ sufficiently large so that $\int_N  (1+|x|)^{(k' - k)r}  \; d\mu_N (x) < \infty$ gives 
\begin{align*}
\int_G | (M V_\psi f) (g)  w(g) |^r \; d\mu_G(g) 
\lesssim  \int_{0}^1 s^{(\hdim /2 + m - l')r} \frac{ds}{s^{\hdim + 1}}  + \int_1^{\infty} s^{(-\hdim /2 - m'+k+l)r} \frac{ds}{s^{\hdim + 1}}.
\end{align*}
Lastly, choosing $m, m' \in \mathbb{N}$ sufficiently large yields that
$
\int_G | (M V_\psi f) (g)  w(g) |^r \; d\mu_G(g) < \infty,
$
as required.
\end{proof}

\subsection{Abstract coorbit spaces} 
In this subsection, we identify the Besov and Triebel-Lizorkin spaces studied in this paper with abstract coorbit spaces introduced in \cite{feichtinger1989banach, velthoven2022quasi}. This identification allows us to obtain molecular decompositions by exploiting the theory \cite{velthoven2022quasi, romero2021dual}.

Throughout this section, let $\psi \in \SV (\N)$ be an admissible vector such that $V_{\psi} \psi \in W(L^r_w)$ for a standard control weight $w$ for $\ST$ as constructed in Lemma \ref{lem:standard_control}. Let
\[
\mathcal{H}_{w, \psi} := \big\{ f \in L^2 (N) : V_{\psi} f \in L^1_w (\NA) \big\},
\]
and equip $\mathcal{H}_{w, \psi}$ with the norm $\| f \|_{\mathcal{H}} = \|V_{\psi} f \|_{L^1_w}$. Note that $V_{\psi} \phi \in L^1_w$ for any $\phi \in \SV(\N)$ by Proposition \ref{prop:better_vector}, so that $\SV(\N) \subseteq \mathcal{H}_{w, \psi}$.
The space $\mathcal{H}_{w, \psi}$ is a $\pi$-invariant Banach space that is independent of the choice of admissible vector $\psi \in \SV(\N)$, cf. \cite[Lemma 4.2]{velthoven2022quasi}.  The anti-dual space of $\mathcal{H}_{w, \psi}$ will be denoted by $\mathcal{R}_{w,\psi} := (\mathcal{H}_{w, \psi})^*$  and equipped with the dual pairing $\langle f, h \rangle_{\mathcal{R}, \mathcal{H}} := f(h)$. The associated (extended) wavelet transform of $f \in \mathcal{R}_{w,\psi}$ is defined as $V_{\psi} f = \langle f, \pi(\cdot) \psi \rangle_{\mathcal{R}, \mathcal{H}}$.

The following lemma identifies Besov and Triebel-Lizorkin spaces as (abstract) coorbit spaces studied in \cite{feichtinger1989banach, velthoven2022quasi}. The result is a special case of \cite[Corollary 4.10]{velthoven2022quasi}.

\begin{lemma} \label{lem:coorbit}
Let $\psi \in \SV (N)$ be an admissible vector. Let $\ord \in \mathbb{R}$, $p, q \in (0,\infty]$ and set $\ord' := \ord + Q/2 - Q/q$.  If $\PTpar > \max \{ \frac{\hdim}{p \wedge q}, |\ord| \}$
and $w$ is a standard control weight for $Y^{p,q}_{a, \sigma'}$, then
  \[  \TLS(\N) = \big\{ f \in \mathcal{R}_{w, \psi} : M^L V_{\psi} f \in \PTv{p}{q}{\ord'} \big\} =: \Co_{\psi}(\PTv{p}{q}{\ord'}) \]
  and 
  \[ \BS(\N) = \big\{ f \in \mathcal{R}_{w, \psi} : M^L V_{\psi} f \in L^{p,q}_{\PTpar, \ord'} \big\} =: \Co_{\psi}(L^{p,q}_{\PTpar, \ord'}) \]
  in the sense that $f \mapsto f|_{\SV}$ is a well-defined bijection from $\Co_{\psi}(\PTv{p}{q}{\ord'})$ into $\TLS(\N)$ (resp. from $\Co_{\psi}(L^{p,q}_{\PTpar, \ord'})$ into $\BS(\N)$). Furthermore, given the unique extension $\widetilde{f} \in \Co_{\psi}(\PTv{p}{q}{\ord'})$ of $f \in \TLS(\N)$ (resp. $\widetilde{f} \in \Co_{\psi}(L^{p,q}_{\PTpar, \ord'})$ of $f \in \BS(\N)$), it holds that
  \[
  \langle \widetilde{f}, \phi \rangle_{\mathcal{R}, \mathcal{H}} = \langle f, \phi \rangle_{\psi} 
  \]
  for all $\phi \in \mathcal{H}_{w, \psi}$, where $\langle \cdot, \cdot \rangle_{\psi}$ denotes the extended pairing of Definition \ref{def:extendedpairing}.
\end{lemma}
\begin{proof}
To treat Besov and Triebel-Lizorkin spaces simultaneously, we write $Y^{p,q}_{a, \sigma}$ for either $\PT$ or $\MT$. 
By Lemma \ref{lem:wavelet_norm}, it follows $\| f \|_{\TLS} \asymp \| M^L V_{\psi} f \|_{\PTv{p}{q}{\ord'}}$ and $\| f \|_{\BS} \asymp \|M^L V_{\psi} f \|_{L^{p,q}_{a, \ord'}}$ for all $f \in \SV'(\N)$. As such, the claims $\TLS(\N) = \Co_{\psi}(\PTv{p}{q}{\ord'})$ and $\BS(\N) = \Co_{\psi} (L^{p,q}_{a, \ord'})$ follow once it is shown that
\begin{align} \label{eq:abstractcoorbit}
\Co_{\psi} (Y^{p,q}_{a, \sigma'}) := \big\{ f \in \mathcal{R}_{w, \psi} : M^L V_{\psi} f \in Y^{p,q}_{a, \sigma'} \big\} = \big\{ f \in \SV'(\N) : M^L V_{\psi} f \in Y^{p,q}_{a, \sigma'} \big\}.
\end{align}
For this, we only need to verify the hypotheses of \cite[Corollary 4.10]{velthoven2022quasi}. 
That is, we need to show that the pair $(Y^{p,q}_{a, \ord'}, w)$ is $L^r_w$-compatible ($r := \min\{1,p,q\}$) in the sense of \cite[Definition 3.5]{velthoven2022quasi}, that $V_{\psi} \psi \in W(L^r_w)$ and $\SV(\N) \hookrightarrow \mathcal{H}_{w, \psi}$, and that the reproducing formula
\begin{align} \label{eq:repro_extended2}
 \int_G V_{\psi} f_1 (g) \overline{V_{\psi} f_2 (g)} \; d\mu_G (g) = \langle f_1, f_2 \rangle
\end{align}
holds for all $f_1 \in \SV' (\N)$ and $f_2 \in \SV (\N)$.

Since $w : \NA \to (0,\infty)$ is a strong control weight for $Y^{p,q}_{a, \ord'}$ by Lemma \ref{lem:standard_control}, a combination of Lemma \ref{lem:OpNormsPTS} and \cite[Corollary 3.9]{velthoven2022quasi} imply that the pair $(Y^{p,q}_{a, \ord'}, w)$ is $L^r_w$-compatible in the sense of \cite[Definition 3.5]{velthoven2022quasi}. The fact that $V_{\psi} \psi \in W(L^r_w)$ and $\SV (\N) \hookrightarrow \mathcal{H}_{w, \psi}$ follow directly by (the proof of) Proposition \ref{prop:better_vector}. Lastly, the formula \eqref{eq:repro_extended2} is satisfied by Lemma \ref{lem:repro_sinfty}. In combination, this shows that \cite[Corollary 4.10]{velthoven2022quasi} is applicable, which yields that Equation \eqref{eq:abstractcoorbit} holds, in the sense that $f \mapsto f|_{\SV}$ is a bijection from $\Co_{\psi}(\PTv{p}{q}{\ord'})$ into $\TLS(\N)$ and from $\Co_{\psi} (L^{p,q}_{a, \ord'})$ into $\BS(\N)$. This proves the first assertion of the lemma.

For the remaining part, let $\widetilde{f} \in \Co_{\psi}(Y^{p,q}_{a, \ord'}) \subseteq \mathcal{R}_{w, \psi}$  be the unique extension of an element $f \in \TLS(\N)$ or $f \in \BS(\N)$. Then $\langle \widetilde{f}, \pi(\cdot) \psi \rangle_{\mathcal{R}, \mathcal{H}} = \langle f, \pi(\cdot) \psi \rangle_{\SV', \SV}$, and hence it follows by \cite[Lemma 4.6(iii)]{velthoven2022quasi} that $\langle \widetilde{f}, \phi \rangle_{\mathcal{R}, \mathcal{H}} = \langle V_{\psi} f, V_{\psi} \phi \rangle_{L^{\infty}_{1/w}, L^1_w}$, and thus $ \langle \widetilde{f}, \phi \rangle_{\mathcal{R}, \mathcal{H}} = \langle f, \phi \rangle_{\psi}$ by definition of $\langle \cdot, \cdot \rangle_{\psi}$.
\end{proof}

\subsection{Molecular decompositions}
This subsection is devoted to obtaining wavelet and molecular decompositions of Besov and Triebel-Lizorkin spaces. The central definition is as follows.

\begin{definition}
    Let $\Lambda \subseteq G$ be discrete and let $\psi \in \SV(\N)$ be an admissible vector. 

    For $r \in (0,1]$ and a standard control weight $w$, a system $(\phi_{\lambda})_{\lambda \in \Lambda}$ of vectors $\phi_{\lambda} \in L^2 (\N)$ is said to be an \emph{$L^r_w$-molecular system} if there exists a function $\Phi \in W(L^r_w)$ such that
    \[
  | \big( V_{\psi} \phi_{\lambda} \big)(g) | \leq \Phi(\lambda^{-1} g)
    \]
    for all $g \in G$ and $\lambda \in \Lambda$.
\end{definition}

Note that any system $(\pi(\lambda) \phi)_{\lambda \in \Lambda}$ for $\phi \in \SV(\N)$ is an $L^r_w$-molecular system since
\[
|( V_{\psi} \pi(\lambda) \phi )(g) | = |\langle \phi, \pi(\lambda^{-1} g) \psi \rangle| = |V_{\psi} \phi|(\lambda^{-1} g)
\]
and $V_{\psi} \phi \in W(L^r_w)$ by Lemma \ref{prop:better_vector}. In general, however, a molecular system $(\phi_{\lambda})_{\lambda \in \Lambda}$ need not be of the simple form $(\pi(\lambda) \phi)_{\lambda \in \Lambda}$, but the wavelet transforms of the vectors $h_{\lambda}$ do satisfy analogous uniform size estimates. The notion of a molecular system is independent of the choice of admissible vector, cf. \cite[Lemma 6.3]{velthoven2022quasi}.

The following result is the main theorem of this section, cf.  \cite[Theorem 6.14]{velthoven2022quasi}.

\begin{theorem} \label{thm:molecular_frame}
Let $\psi \in \SV(\N)$ be an admissible vector. 
Let $\ord \in \mathbb{R}$, $p, q \in (0,\infty]$ and $r := \min \{1,p,q\}$. If $\PTpar > \max \{ \frac{\hdim}{p \wedge q}, |\ord| \}$ and $w : G \to (0,\infty)$ is a standard control weight for $\ST$, then
 there exists a relatively compact unit neighborhood $V \subseteq G$ with the following property:

For any discrete set $\Lambda \subseteq G$ satisfying
\begin{align} \label{eq:relsep}
\sup_{g \in G} \# (\Lambda \cap g U) < \infty \quad \text{and} \quad G = \bigcup_{\lambda \in \Lambda} \lambda V,
\end{align}
there exists an $L^r_w$-molecular system $(\phi_{\lambda})_{\lambda \in \Lambda}$ such that any $f \in \TLS (N)$ (resp. $f \in \BS (N)$) admits the expansion
\begin{align} \label{eq:molecular}
f = \sum_{\lambda \in \Lambda} \langle f, \phi_{\lambda} \rangle_{\psi} \; \pi(\lambda) \psi = \sum_{\lambda \in \Lambda} \langle f, \pi(\lambda) \psi \rangle_{\psi} \; \phi_{\lambda}
\end{align}
with convergence of the series in the weak-$*$-topology of $\SV'(\N)$. 
\end{theorem}

\begin{proof}
By Lemma \ref{lem:coorbit}, it follows that $\TLS(\N) = \Co_{\psi}(\PTv{p}{q}{\ord'})$ and $\BS(\N) = \Co_{\psi} (L^{p,q}_{a, \ord'})$ for $a> \max\{\hdim / p \wedge q, |\sigma|\}$ and $\sigma' = \sigma +\hdim/2 - \hdim/q$. 
Denote by $Y^{p,q}_{a, \ord'}$ either $\PTv{p}{q}{\ord'}$ or $L^{p,q}_{a, \ord'}$.
Then an application of  \cite[Theorem 6.14]{velthoven2022quasi} yields a compact unit neighborhood $V \subseteq G$ such that, for any discrete set $\Lambda \subseteq G$ satisfying \eqref{eq:relsep}, there exists a $L^r_w$-molecular system $(\phi_{\lambda})_{\lambda \in \Lambda}$ such that any $f' \in \Co_{\psi}(Y^{p,q}_{a, \ord'})$ can be represented as
\begin{align} \label{eq:molecular_abstract}
f' = \sum_{\lambda \in \Lambda} \langle f', \phi_{\lambda} \rangle_{\mathcal{R}, \mathcal{H}} \; \pi(\lambda) \psi = \sum_{\lambda \in \Lambda} \langle f', \pi(\lambda) \psi \rangle_{\mathcal{R}, \mathcal{H}} \; \phi_{\lambda}
\end{align}
with unconditional convergence of the series in the weak-$*$-topology of $\mathcal{R}_{w, \sigma}$. By Lemma \ref{lem:coorbit}, any $f \in \TLS(\N)$ and $f \in \BS(\N)$ admit a unique extension to an element $\widetilde{f} \in \Co_{\psi}(Y^{p,q}_{a, \ord'})$ for which the dual pairing satisfies $\langle \widetilde{f}, \phi \rangle_{\mathcal{R}, \mathcal{H}} = \langle f, \phi \rangle_{\psi}$. Therefore, applying Equation \eqref{eq:molecular_abstract} to $\widetilde{f}$ and restricting the domains of both sides of Equation \eqref{eq:molecular_abstract} to $\SV(\N)$ yields that Equation \eqref{eq:molecular} holds for all $f \in \TLS(\N)$ and $f \in \BS(\N)$, with unconditional convergence in the weak-$*$-topology of $\SV'(\N)$.
\end{proof}

We record the following useful consequence of Theorem \ref{thm:molecular_frame}. 

\begin{corollary} \label{cor:dense}
   Let $\ord \in \mathbb{R}$. The space $\SV(N)$ is norm dense in $\TLS(\N)$ and $\BS(\N)$ if $p, q \in (0, \infty)$ and weak-$*$-dense, otherwise.
\end{corollary}
\begin{proof}
  Let $\psi \in \SV(N)$ be an admissible vector and let $f \in \TLS(N) = \Co_{\psi}(\PTv{p}{q}{\ord'})$ or $f \in \BS(\N) = \Co_{\psi} (L^{p,q}_{a, \ord'})$ (cf. Lemma \ref{lem:coorbit}), and write $Y^{p,q}_{a, \ord'}$ for either $\PTv{p}{q}{\ord'}$ or $L^{p,q}_{a, \ord'}$. By Theorem \ref{thm:molecular_frame}, it follows that $f = \sum_{\lambda \in \Lambda} c_{\lambda} (f) \pi(\lambda) \psi$ for coefficients $c_{\lambda} (f) := \langle f, \phi_{\lambda} \rangle_{\psi}$, where $(\phi_{\lambda})_{\lambda \in \Lambda}$ is a system of $L^p_w$-molecules. Choose an enumeration $(\lambda_n)_{n \in \mathbb{N}}$ of the index set $\Lambda$ and define $f_k = \sum_{n = 1}^k c_{\lambda_n} (f) \pi(\lambda_n) \psi \in \SV(N)$ for $k \in \mathbb{N}$. Then $V_{\psi} f_k = \sum_{n = 1}^k c_{\lambda_n} (f) L_{\lambda_n} V_{\psi} \psi$, and thus $V_{\psi} f_k \in \Co_{\psi} (Y^{p,q}_{a, \ord'})$ by Lemma \ref{lem:coorbit}. Moreover, Lemma \ref{lem:repro_sinfty} and Lebesgue's dominated convergence theorem yield 
  $
  \lim_{k \to \infty} \langle f_k , \varphi \rangle = \langle f, \varphi \rangle,
  $
  so that $f_k \to f$ in the weak-$*$-topology. 

  Lastly, the fact that $(\phi_{\lambda})_{\lambda \in \Lambda}$ is a molecular system implies by \cite[Proposition 6.11]{velthoven2022quasi} that $(c_{\lambda} (f))_{\lambda \in \Lambda} \in Y_d (\Lambda)$, where $Y_d (\Lambda)$ denotes the sequence space associated to $Y^{p,q}_{a, \ord'}$ (cf. \cite[Section 2.4]{velthoven2022quasi}). Therefore, if $p, q < \infty$, the series defining $f$ converges in the norm of $\Co_{\psi}(\PTv{p}{q}{\ord'})$ by \cite[Proposition 6.11]{velthoven2022quasi}, which implies that $f_k \to f$ converges in norm, and thus $\SV(N)$ is norm dense.
\end{proof}

\subsection{Application: boundedness of operators}
Using the molecular decomposition proven in Theorem \ref{thm:molecular_frame}, we easily obtain the following criterion for the boundedness of operators on Besov and Triebel-Lizorkin spaces. See \cite{gilbert2002smooth, grochenig2009molecules} for similar results on Euclidean spaces. 

\begin{theorem} \label{thm:continuity_BTL}
 Let $\ord \in \mathbb{R}$, $p, q \in (0,\infty]$ and $r := \min \{1,p,q\}$. Moreover, let $\PTpar > \max \{ \frac{\hdim}{p \wedge q}, |\ord| \}$ and $w : G \to (0,\infty)$ be a standard control weight for $\ST$. For an admissible vector $\psi \in \SV(\N)$, let $\Lambda \subseteq \NA$ be a discrete set such that $(\pi(\lambda) \psi)_{\lambda \in \Lambda}$ is a frame for $L^2 (\N)$ with a dual frame forming an $L^r_w$-molecular system (cf. Theorem~\ref{thm:molecular_frame}).

\begin{enumerate}
    [\rm (i)]\item If $T: L^2(\N) \to L^2(\N)$ is continuous and such that $(T \pi(\lambda) \psi)_{\lambda \in \Lambda}$ forms an $L^r_w$-molecule, then $T$ extends to continuous maps on $\TLS(\N)$ and $\BS(\N)$.
    \item In particular, this holds true for any operator $T:\SV(\N)\rightarrow \SV(\N)$ defined by
    \[
    T f = f \ast k_T, \quad f \in \SC(\N),
    \]
    for a distribution $k_T \in \mathcal{S}'(\N)$ of homogeneous of degree $- \hdim$ that is smooth on $\N \setminus \{e_N\}$.
\end{enumerate}
\end{theorem}

\begin{proof} 
(i) Due to the bijection between Besov and Triebel-Lizorkin spaces and their corresponding coorbit spaces shown in Lemma~\ref{lem:coorbit}, the claimed continuity follows directly from \cite[Theorem 7.1]{velthoven2022quasi} and \cite[Remark 7.2]{velthoven2022quasi}.
\\~\\
(ii)  
Under the assumptions, an application of \cite[Theorem 3.2.30]{FR} yields that $T$ extends to a continuous operator on $L^2(\N)$. In addition, since the kernel $k_T$ of $T$ is homogeneous of degree $-\hdim$, the operator $T$ is homogeneous of degree $0$ (cf.~\cite[Lemma 3.2.7]{FR}), and thus it commutes with dilations. Clearly, $T$ is also translation-invariant. 

Let $(\pi(\lambda) \psi)_{\lambda \in \Lambda}$ be a frame as in the statement.
 We observe that
\begin{align*}
	(V_{\psi} T\pi(\lambda) \psi )(g) &= \langle T \pi(\lambda) \psi, \pi(g) \psi \rangle = \langle \pi(\lambda) T \psi, \pi(g) \psi \rangle = \langle T \psi, \pi(\lambda^{-1}g) \psi \rangle \\
    &= (V_\psi T \psi)(\lambda^{-1}g)
\end{align*}
for all $g \in G$ and $\lambda \in \Lambda$. Since $T \psi \in \SV(\N)$ by assumption on $T$, the wavelet transform $V_\psi T \psi$ lies in $W(L^r_w)$ by Lemma \ref{prop:better_vector}. In combination, this shows that the family $(T \pi(\lambda) \psi)_{\lambda \in \Lambda}$ forms an $L^r_w$-molecular system, and thus (i) applies. This completes the proof of the theorem.
\end{proof}

\begin{example} \label{ex:Riesz_transform_graded}
Let $\N$ be a graded Lie group of homogeneous degree $\hdim$ and let $\RLO$ be a homogeneous positive Rockland operator on $\N$ of degree $\hdeg \in \NN$. Consider
the left-invariant Riesz transforms $T_\alpha := X^\alpha \RLO^{-\frac{[\alpha]}{\hdeg}}$ for arbitrary nonvanishing $\alpha \in \NN_0^\dimN$.
Using the identity
\[
\lambda^{-\frac{[\alpha]}{\nu}} = \frac{1}{\Gamma(\frac{[\alpha]}{\nu})}\int_0^\infty t^{\frac{[\alpha]}{\nu}}e^{-t\lambda}\; \frac{dt}{t}, \quad \lambda >0,
\]
we can write 
\begin{equation*} 
T_\alpha f= X^\alpha \RLO^{-\frac{[\alpha]}{\hdeg}} f= \frac{1}{\Gamma(\frac{[\alpha]}{\nu})}\int_0^\infty  t^{\frac{[\alpha]}{\nu}}X^\alpha e^{-t\RLO}f\; \frac{dt}{t} 
\end{equation*}
for all $f\in \SV(\N)$, where we use that $\SV(\N) \subseteq \Dom(\RLO^{-\frac{[\alpha]}{\hdeg}})$, cf.  Lemma~\ref{lem:domain_of_Rps}.
Hence, the distributional convolution kernel $k_{T_\alpha}$ of $T_\alpha$ has the expression
\begin{equation} \label{eq:expression_of_kernel_of_T_alpha}
k_{T_\alpha} = \frac{1}{\Gamma(\frac{[\alpha]}{\nu})}\int_0^\infty  t^{\frac{[\alpha]}{\nu}}X^\alpha \big( h(\cdot, t)\big)\; \frac{dt}{t},
\end{equation}
where $h : \N \times (0,\infty) \to \mathbb{C}$ denotes the heat kernel associated with $\RLO$, see, e.g., \cite[Theorem 4.2.7]{FR}. Using the fact that $h(x,t) = t^{-\frac{\hdim}{\nu}}h (t^{-\frac{1}{\nu}}x,1)$,
we may rewrite \eqref{eq:expression_of_kernel_of_T_alpha} as 
\begin{align} \label{eq:expression_of_kernel_of_T_alpha_2}
k_{T_\alpha} = \frac{1}{\Gamma(\frac{[\alpha]}{\nu})}\int_0^\infty t^{-\frac{\hdim}{\nu}} (X^\alpha \phi)(t^{-\frac{1}{\nu}}\cdot)\; \frac{dt}{t} &=
\frac{\nu}{\Gamma(\frac{[\alpha]}{\nu})}\int_0^\infty s^{-\hdim} (X^\alpha \phi)(s^{-1}\cdot)\; \frac{ds}{s} \nonumber \\
&= \frac{\nu}{\Gamma(\frac{[\alpha]}{\nu})}\int_0^\infty (X^\alpha \phi)_s\; \frac{ds}{s},
\end{align}
where $\phi := h(\cdot, 1)$ is 
a Schwartz function on $\N$ (cf. \cite[Corollary 4.2.17]{FR}). 
Since $X^\alpha \phi \in \SC(\N)$ and $\int_\N (X^\alpha \phi)(x)d\mu_\N (x)=0$ by integration by parts, it follows from 
\cite[Theorem 1.65]{FS} that the last integral in \eqref{eq:expression_of_kernel_of_T_alpha_2} converges in $\mathcal{S}'(\N)$ to a distribution which is smooth away from the group identity $e_\N$ and homogeneous of degree $-\hdim$.
Moreover, a straight-forward argument shows that $T_\alpha$ maps $\SV(\N)$ continuously into itself (cf., e.g., Lemma~\ref{lem:domain_of_Rps}~(i)). Hence, each $T_\alpha$ satisfies the conditions of Theorem~\ref{thm:continuity_BTL}~(ii) and therefore it is bounded on $\BS(\N)$ and $\TLS(\N)$ for the full range of the parameters $p,q \in (0,\infty]$ and $\sigma \in \mathbb{R}$.
\end{example}

\begin{example} \label{ex:Riesz_transform_hom}
For an arbitrary homogeneous group $\N$ and an operator $\P$ defined by \eqref{eq:def_P}, up to a choice of homogeneous quasi-norm on $\N$, we consider the operators $T_\alpha := X^\alpha \P^{-[\alpha]}$, $\alpha \in \mathbb{N}_0^\dimN$.
For an arbitrary but fixed $\alpha \in \mathbb{N}_0^\dimN$ we choose a real-valued function $m \in C_c^\infty (\mathbb{R}^+)$ such that
$C_m := \int_0^\infty t^{[\alpha]}m(t) \;\frac{dt}{t} \neq 0$. Then a simple change of variables gives
\begin{align*}
\lambda^{-[\alpha]}=C_m^{-1} \int_0^\infty t^{[\alpha]} m(t\lambda) \;\frac{dt}{t}, \quad \lambda \in \mathbb{R}^+.
\end{align*} 
By functional calculus,
\[
\P^{-[\alpha]} f=  C_m^{-1} \int_0^\infty t^{[\alpha]} m(t \P)f \;\frac{dt}{t}
\]
for $f \in \Dom(\P^{-[\alpha]})$.
If $k_m$ denotes the convolution kernel of $m(\P)$, then the fact that $k_m \in \SV(\N)$ (cf. the proof of Proposition~\ref{prop:construction_crk}~(ii)) can be used to show that $\SV(\N) \subseteq \Dom(\P^{-[\alpha]})$ and that $\P^{-[\alpha]}$ maps $\SV(\N)$ continuously into itself. (See Lemma~\ref{lem:domain_of_Rps}~(i) for similar arguments in the case of Rockland operators.)
Hence, the operator $X^\alpha \P^{-[\alpha]}$ is continuous on $\SV(\N)$ and can be written as
\[
X^\alpha \P^{-[\alpha]} f =  C_m^{-1} \int_0^\infty t^{[\alpha]}X^\alpha m(t \P)f \;\frac{dt}{t} = C_m^{-1} \int_0^\infty t^{[\alpha]}f * X^\alpha (k_m)_t \;\frac{dt}{t}
\]
for all $f \in \SV(\N)$. Since $X^\alpha k_m \in \SV(\N)$, it follows from \cite[Theorem 1.65]{FS} that the corresponding integral representation of its distributional convolution kernel $k_{T_\alpha}$,
\begin{align*} %\label{eq:kernel_of_Riesz_P}
k_{T_\alpha}= C_m^{-1}\int_0^\infty 
 t^{[\alpha]} X^\alpha (k_m)_t \;\frac{dt}{t} 
 = C_m^{-1}\int_0^\infty
\big( X^\alpha k_m\big)_t \;\frac{dt}{t},
\end{align*}
converges in $\mathcal{S}'(\N)$ to a distribution which is smooth away from $e_\N$ and homogeneous of degree $-\hdim$. Thus, $T_\alpha$ satisfies the conditions of Theorem 7.8 (ii) and therefore it is bounded on $\BS(\N)$ and $\TLS(\N)$ for the full range of the parameters $p,q \in (0,\infty]$ and $\sigma \in \mathbb{R}$. 
\end{example}

\section{Identification  with some specific function spaces} \label{sec:identification}
In this section, we identify the Besov and Triebel-Lizorkin spaces defined in this paper with various function spaces studied in the literature before. Among others, we discuss the identification with Hardy spaces on homogeneous groups \cite{FS}, Sobolev spaces on stratified and graded groups \cite{FR2, Folland1975subelliptic} and Lipschitz spaces on stratified groups \cite{Folland1979Lipschitz}.

\subsection{Lebesgue and Hardy spaces on homogeneous groups}
The Hardy spaces considered in this section were introduced and systematically studied by Folland and Stein \cite{FS}. 
Among the various definitions and characterizations,  the definition of Hardy spaces in terms of commutative approximative identities is most convenient for our purposes. 

 A function $\eta \in \mathcal{S}(\N)$ is said to be a \emph{commutative approximate identity} if $\int_\N \eta(x)\; d\mu_N(x) =1$  and 
\[ \eta_s \ast \eta_t =\eta_t \ast \eta_s \] 
hold for all $s, t>0$. The existence of commutative approximate identities on general homogeneous groups was proven in \cite{Dziubanski1992, glowacki1986stable}. Given a commutative approximate identity, 
the associated (radial) maximal function of a tempered distribution $f \in \mathcal{S}'(N)$ is defined by
\begin{align*}
M_\eta f (x) =\sup_{t >0} |f \ast \eta_t (x)|, \quad x \in \N.
\end{align*}
The Hardy space $H^p(N)$, $p \in (0,\infty)$, is then defined as the space of all tempered distributions $f \in \mathcal{S}'(N)$ such that $M_{\eta} f \in L^p (\N)$ and equipped with the quasi-norm $\| f \|_{H^p} := \| M_{\eta} f \|_{L^p}$. This definition of the Hardy space is independent of the choice of commutative approximate identity, with equivalent (quasi-)norms for different choices. Moreover, $H^p (\N)$ is isometric to the Lebesgue space $L^p(\N)$ for $p \in (1,\infty)$.
For more details and results, cf. \cite{FS}.

Hardy spaces can be identified with Triebel-Lizorkin spaces as studied in this paper in the following manner.

\begin{proposition} \label{prop:Identification of Lebesgue and Hardy}
Let $f \in \mathcal{S}'(\N)$ and $p \in (0,\infty)$. 
\begin{enumerate}[\rm (i)]
    \item If $f \in H^p(\N)$, then $[f] =f + \mathcal{P} \in \TLptwo(\N)$, and     
    \[
    \big\|[f]\big\|_{\TLptwo} \lesssim \|f\|_{H^p}.
    \]
    \item If $f \in \mathcal{S}'(\N)$ such that  
    $[f]=f + \mathcal{P} \in \TLptwo(\N)$, then there exists a polynomial $P$ such that  $f-P \in H^p(\N)$, and  
    \[
    \|f-P\|_{H^p}
    \lesssim \big\|[f]\big\|_{\TLptwo}.
    \]
\end{enumerate}
\end{proposition}

\begin{proof} The proof is based on quite classical arguments using vector-valued singular integrals, see, e.g., \cite{Bui, FS, Sato} for similar results. 
Let $\mathcal{H}$ denote the Hilbert space $L^2((0,\infty),dt/t)$. For $p \in (0,\infty)$, let $L^p_\mathcal{H}(\N)$ be the space of $\mathcal{H}$-valued measurable functions $F$ on $\N$ for which
\begin{align*}
\|F\|_{L^p_\mathcal{H}} = \Big\|x \mapsto \big\|t \mapsto F(x)(t) \big\|_\mathcal{H}  \Big\|_{L^p} 
=\left( \int_{\N} \left(\int_0^\infty  |F(x)(t)|^{2} \; \frac{dt}{t}\right)^{p/2}d\mu_\N (x)\right)^{1/p}  
<\infty.
\end{align*}
Similarly, let $H^p_{\mathcal{H}}(\N)$ be the $\mathcal{H}$-valued Hardy spaces over $N$, that is, the space of $\mathcal{H}$-valued tempered distributions $F$
for which
\begin{align*}
\|F\|_{H^p_{\mathcal{H}}}& := 
\left\|x\mapsto \sup_{s >0} \big\| t \mapsto (F\ast \eta_s )(x)(t)\big\|_{\mathcal{H}} \right\|_{L^p}\\
&= \left\|x\mapsto \sup_{s >0} \left( \int_0^\infty\big|(F \ast \eta_s)(x)(t)\big|^2 \; \frac{dt}{t}\right)^{1/2}\right\|_{L^p} < \infty
\end{align*}
for an arbitrary but fixed commutative approximate identity $\eta \in \SC(\N)$.
\\~\\
(i)  Let us fix a function $\crk \in \SV(N)$ satisfying the continuous Calder\'{o}n condition~\eqref{eq:continuous_calderon}. We show that $H^p(N)$ is continuously embedded into $\TLptwo(\N)$. Define 
the continuous linear mapping $K:\mathcal{S}(\N) \rightarrow \mathcal{H}$ by
\begin{align*}
( K, f) (t) =
\int_\N f(x) \phi_t(x)d\mu_\N(x), \quad f \in \mathcal{S}(\N).
\end{align*}
Then $K$ is an $\mathcal{H}$-valued 
tempered distribution, cf. the proof of \cite[Theorem 7.7]{FS}. Moreover,  the proof of \cite[Theorem 7.7]{FS} shows that the associated convolution-type singular integral operator $T_K$, given by
\begin{align*}
T_K f (x)(t) = ( f \ast K )(x) (t) =  f \ast \phi_t (x), \quad  f \in \mathcal{S}(\N),
\end{align*}
extends to a bounded map from $L^2(\N)$ to $L^2_\mathcal{H}(\N)$ and that its kernel $K$ satisfies, for any $\alpha \in \mathbb{N}_0^n$,
\begin{align} \label{regular}
\left\|\widetilde{X}^\alpha  K(x)\right\|_{\mathcal{H}} \leq C_\alpha |x|^{-\hdim -[\alpha]}, \quad x \in \N \setminus \{e\}. 
\end{align}
Therefore, by \cite[Theorem 6.20]{FS}, it follows that $T_K$ extends to a bounded operator from $H^p(\N)$ to $H^p_\mathcal{H}(\N)$. This further implies that $T_K$ is bounded from $H^p(\N)$ to $L^p_\mathcal{H}(\N)$ by a limit argument as in the proof of \cite[Theorem 7.7]{FS}, thus
\begin{align*}
\left\| \left(\int_0^\infty |f \ast \phi_t |^2 \; \frac{dt}{t} \right)^{1/2} \right\|_{L^p} =\|T_K f\|_{L^p_\mathcal{H}}\lesssim \|f\|_{H^p}.
\end{align*}
In view of Theorem \ref{thm:cont_char}, this gives $[f] \in \TLptwo(\N)$ with $\| [f] \|_{\dot{F}^0_{p,2}} \lesssim \|f\|_{H^p}$.
\\~\\
(ii) For any $\psi \in \SV(\N)$ we define the continuous linear functional $\widetilde{K} : \mathcal{S}(\N; \mathcal{H}) \rightarrow \mathbb{C}$ by
\begin{align*}
\langle \widetilde{K}, \varphi\rangle =
 \int_\N \int_0^\infty  \psi_t (x) \big(\varphi(x)(t)\big) \; \frac{dt}{t}d\mu_\N(x), \quad \varphi \in \mathcal{S}(\N;\mathcal{H}).
\end{align*} 
From the proof of \cite[Theorem 7.7]{FS} and taking into account that $\mathscr{B}(\mathcal{H},\mathbb{C}) \cong \mathcal{H}$, we see that $\widetilde{K}$ is a $\mathscr{B}(\mathcal{H},\mathbb{C})$-valued tempered distribution on $\N$ and, similarly 
as \eqref{regular} we have, for any $\alpha \in \mathbb{N}_0^n$,
\begin{align*}
\left\|\widetilde{X}^\alpha  \widetilde{K}(x)\right\|_{\mathscr{B}(\mathcal{H},\mathbb{C})} \leq C_\alpha |x|^{-\hdim -[\alpha]}, \quad x \in \N \setminus \{e\}.
\end{align*}
Therefore, by \cite[Theorem 6.20]{FS}, the operator $T_{\widetilde{K}}$ is bounded from $H^p_\mathcal{H}(\N)$ into $H^p(\N)$ with
\begin{align} \label{vec val bdd}
\|T_{\widetilde{K}}g\|_{H^p} \lesssim \|g\|_{H^p_\mathcal{H}}.
\end{align}
We will show that \eqref{vec val bdd} yields the desired estimate for an appropriate choice of $g \in H^p_\mathcal{H}(\N)$.

Consider the operator $\P$ defined by \eqref{eq:def_P}. Choose some real-valued function $m_1 \in C^\infty_c(\R^+)$ that satisfies the Tauberian condition, i.e., for every $\lambda > 0$ there exists some $t > 0$ such that $m_1(t \lambda) \neq 0$. Then $m := m_1 \cdot m_1\in C^\infty_c(\R^+)$ also satisfies the Tauberian condition. Hence, by the Calder\'{o}n representation theorem (see, e.g., \cite[Thm.~1.1]{Hei}), there exists a real-valued function $m_2 \in C^\infty_c(\R^+)$ such that
\begin{align*}
    \int_0^\infty m(t \lambda) m_2(t \lambda) \; \frac{dt}{t} = 1 \quad \mbox{ for all } \lambda \in \R^+.
\end{align*}
So by Proposition~\ref{prop:construction_crk}, the convolution kernels $\phi$,  $\psi \in \SV(\N)$ of $m_1(\P), m_2(\P)$ satisfy the reproducing formula
\begin{align} \label{cal rep ind}
    f = \int_0^\infty f * \phi_t * \phi_t * \psi_t \; \frac{dt}{t}
\end{align}
for all $f \in \SV'(\N)$ with the right-hand side converging in $\SV'(\N)$.

Let now $f \in \mathcal{S}'(\N)$ such that $[f] = f + \mathcal{P} \in \SV'(\N)$ satisfy $\| [f] \|_{\TLptwo} < \infty$ and define the vector-valued function $g$ by $g(x)(t) = f \ast \phi_t \ast {\phi}_t(x)$. Then
\begin{align} \label{eq:vector_valued_g}
 T_{\widetilde{K}}g = \int_0^\infty f * \phi_t * \phi_t * \psi_t \; \frac{dt}{t}.
\end{align}
Using the simple estimate
\begin{align*}
\sup_{s >0} \left( \int_0^\infty \bigl |f \ast \phi_t \ast {\phi}_t \ast
\eta_s \bigr |^2 \; \frac{dt}{t}\right)^{1/2} \leq \left( \int_0^\infty \sup_{s >0} \bigl |f \ast \phi_t \ast {\phi}_t \ast 
\eta_s \bigr |^2 \; \frac{dt}{t}\right)^{1/2},
\end{align*}
we observe that
\begin{align*}
    \|g\|_{H^p_\mathcal{H}} &= \left\|\sup_{s >0} \left( \int_0^\infty \bigl |f \ast \phi_t \ast {\phi}_t \ast \eta_s \bigr |^2 \; \frac{dt}{t}\right)^{1/2}  \right\|_{L^p} \\
    &\lesssim \left\|\left( \int_0^\infty \sup_{s >0} \bigl |f \ast \phi_t \ast {\phi}_t \ast \eta_s \bigr |^2 \; \frac{dt}{t}\right)^{1/2}  \right\|_{L^p}.
\end{align*}
By \cite[Lemma 5.4]{Sato}, given $a >0$, it follows that
\begin{align*}
\sup_{s >0} \bigl |f \ast \phi_t \ast {\phi}_t \ast \eta_s(x) \bigr |
\lesssim_{\phi, a} (\crk_t^* f)_\PTpar(x), \quad t >0, \ x \in \N.
\end{align*}
Hence, using \eqref{eq:cont_char_1},
\begin{align*}
    \|g\|_{H^p_\mathcal{H}} \lesssim \left\|\left( \int_0^\infty \big[ (\crk_t^* f) _\PTpar \big]^2 \; \frac{dt}{t}\right)^{1/2}  \right\|_{L^p} \lesssim \big\|[f]\big\|_{\dot{F}^0_{p,2}}.
\end{align*}
Combining this with \eqref{eq:vector_valued_g} and \eqref{vec val bdd} gives
\begin{align} \label{eq:Hardy_norm_of_0_infty}
\biggl \| \int_0^{\infty} f \ast \phi_t \ast \phi_t \ast \psi_t\; \frac{dt}{t} \biggr \|_{H^p}
= \| T_{\widetilde{K}}g \|_{H^p} \lesssim \| g \|_{H^p_{\mathcal{H}}} 
   \lesssim \big\|[f]\big\|_{\dot{F}^0_{p,2}}.
\end{align}
In particular, this implies that 
\begin{align} \label{eq:varepsilon_L_Hardy}
\int_\varepsilon^L f \ast \phi_t \ast \phi_t \ast \psi_t\; \frac{dt}{t} 
\end{align}
converges in $H^p(\N)$ to some element $\widetilde{f}$ as $\varepsilon \rightarrow 0$ and $L\rightarrow \infty$. Since $H^p(\N)$ is embedded continuously into $\mathcal{S}'(\N)$ (cf. \cite[Propositon 2.15]{FS}), \eqref{eq:varepsilon_L_Hardy} also converges in $\mathcal{S}'(N)$ to $\widetilde{f}$ as $\varepsilon \rightarrow 0$ and $L\rightarrow \infty$. On the other hand, 
the Calder\'{o}n reproducing formula \eqref{cal rep ind} says that \eqref{eq:varepsilon_L_Hardy} converges in $\mathcal{S}'(\N)/\mathcal{P}$ to $f$. Hence $\widetilde{f} = f-P$ for some polynomial $P$. Therefore, in view of \eqref{eq:Hardy_norm_of_0_infty}, we have $f -P \in H^p(\N)$ and $\|f-P\|_{H^p} \lesssim \big\|[f]\big\|_{\dot{F}^0_{p,2}}$.
\end{proof}

\subsection{Functions of bounded mean oscillation on homogeneous groups}

Following \cite{FS},  the space $\BMO(\N)$  of functions of bounded mean oscillation (BMO) on $\N$ consists of those functions $f \in L^1_{\loc} (\N)$ such that
\begin{align*}
    \| f \|_{\BMO} := \sup_{x \in \N, t > 0} \hspace{5pt} \dashint_{B_t(x)} |f(y) - m_{B_t(x)}f| \, d\mu_\N(y) < \infty,
\end{align*}
where
$
    m_{B_t(x)} f= \dashint_{B_t(x)} f(y) \, d\mu_\N(y)
$
for $x \in \N$ and $t>0$.

We start by recalling some basic facts about $\BMO(\N)$. First, the space $\BMO(\N)$ can be identified with the dual space of the Hardy space $H^1(\N)$, cf. \cite{FS, BownikFolland2007}. 
Second, if $f \in \BMO(\N)$, then, for any $\varepsilon >0$,
\begin{align*}
\int_{\N} \frac{|f(x)|}{(1+|x|)^{\hdim + \varepsilon}} \; d\mu_N (x) < \infty;
\end{align*}
see \cite[Proposition 5.9]{FS}.
As such, any $f \in \BMO(\N)$ can be regarded as a tempered distribution. Actually, the proof of \cite[Proposition 5.9]{FS} gives the following estimate:
\begin{align} \label{eq:local p-integrable}
\int_{\N} \frac{|f(x) - m_{B_1(e_N)}f|}{(1 + |x|)^{\hdim +\varepsilon}} d\mu_N (x) \lesssim \|f\|_{\BMO}.
\end{align}
In addition,  \cite[Corollary 5.8]{FS} shows that the norm equivalence
\begin{align} \label{eq:equivalent bmo norm}
 \|f\|_{\BMO} \asymp \sup_{x \in \N, t > 0} \hspace{5pt} \left(\dashint_{B_t(x)} |f(y) - m_{B_t(x)}f|^p \, d\mu_\N(y)\right)^{1/p}
\end{align}
holds for all $1<p <\infty$.

In the following we will prove that $\BMO(\N)$ can be identified with $\F^0_{\infty, 2}(\N)$  modulo polynomials.

\begin{proposition} \label{prop:identification_BMO}
\begin{enumerate} 
   [\rm (i)] \item If $f \in \BMO(\N)$, then $[f] =f + \mathcal{P} \in \F^0_{\infty,2}(\N)$, and
    \begin{align*} 
    \bigl \| [f] \bigr \|_{\F^0_{\infty,2}} \lesssim \| f \|_{\BMO}.
    \end{align*}
    \item If $f \in \mathcal{S}'(\N)$ such that $[f] = f + \mathcal{P} \in \F^0_{\infty,2}(\N)$, then there exists a polynomial $P$ such that  $f-P \in \BMO(\N)$, and
     \begin{align*}
     \| f-P \|_{\BMO} \lesssim \bigl \|[f] \bigr \|_{\F^0_{\infty,2}}.
    \end{align*}
\end{enumerate}
\end{proposition}

\begin{proof}
Our line of arguments will closely follow the proof of \cite[Ch.~IV~\S~4~Theorem.~3]{S} on Euclidean spaces. 
Throughout, let $\phi \in \SV(\N)$ be a function satisfying the continuous Calder\'on condition $f = \int_0^{\infty} f \ast \phi_t \ast \phi_t \; \frac{dt}{t}$ for $f \in \SV'(\N)$.
\\~\\
(i) Let $f \in \BMO(\N)$ and let $B := B_t(x)$ and $B^* := B_{2\gamma t}(x)$ for some fixed $x \in \N$ and $t \in (0, \infty)$. Then write $f = f_1 + f_2 + f_3$, where $f_1 := (f - m_{B^*}f)\chi_{B^*}$, $f_2 := (f - m_{B^*}f)\chi_{N \setminus B^*}$ and $f_3 := m_{B^*}f$. Since $\crk \in \SV(\N)$, we have 
\begin{align*}
f \ast \crk_\tau  = f_1 \ast \crk_\tau + f_2 \ast \crk_\tau.
\end{align*}
Using the $L^2(\N)$-continuity of the operator
\begin{equation*}
f_1 \mapsto \left(\int_0^\infty |f_1 \ast \phi_\tau|^2\; \frac{d\tau}{\tau}  \right)^{1/2}
\end{equation*}
cf. \cite[Theorem 7.7]{FS}, together with the norm equivalence \eqref{eq:equivalent bmo norm}, we get
\begin{align*}
    \int_{B_t(x)} \int_0^t |f_1 \ast \phi_\tau (y) |^2 \; \frac{d\tau}{\tau}d\mu_\N (y) &\leq \bigg\|\left(\int_0^\infty |f_1 \ast \phi_\tau|^2 \; \frac{d\tau}{\tau}  \right)^{1/2} \bigg\|_{L^2}^2 \lesssim \|f_1\|_{L^2(\N)}^2 \\
    &= \int_{B^*} |f(y) -m_{B^*} f|^2 \, d\mu_\N(y) \lesssim \mu_\N (B) \|f\|_{\BMO}^2
\end{align*}
for $x \in \N$ and $t > 0$.
Hence, by the norm equivalence \eqref{eq:cont_char_2_q}, this implies that
\begin{align}  \label{eq:BMO_estimate_f1}
    \bigl \| [f_1] \bigr \|_{\F^0_{\infty,2}} \asymp \sup_{x \in \N, t > 0} \biggl( \dashint_{B_t(x)} \int_0^t \bigl |f_1 * \crk_{\tau} (y) \bigr|^2 \; \frac{d\tau}{\tau}  d\mu_\N(y) \biggr)^{1/2} \lesssim \| f \|_{\BMO}.
\end{align}

To treat $f_2$, we first notice that because of $\crk \in \SC(\N)$,
\begin{align*}
    |\crk_{\tau}(z^{-1}y)| \lesssim \tau^{-\hdim} (1 + \tau^{-1}|z^{-1}y|)^{-(\hdim + 1)}
\end{align*}
for all $y \in \N$, $z \in \N$ and $\tau > 0$. Combining this with  the fact that 
\begin{align*}
    \tau + |z^{-1}y| \asymp 2 \gamma t + |z^{-1}x|
\end{align*}
for all $y \in B_t (x)$, $z \in \N \setminus B_{2\gamma t} (x)$ and $\tau \in (0,t)$, we get 
\begin{align*} 
    |f_2 \ast \crk_\tau (y)|
    &\lesssim \int_{\N \setminus B^*} \frac{|f(z) - m_{B^*}f| \tau^{-\hdim}}{(1 + \tau^{-1}|z^{-1}y|)^{\hdim + 1}} \; d\mu_\N(z) \\
    &
    \lesssim \frac{\tau}{t} \int_{\N \setminus B^*} \frac{|f(z) - m_{B^*}f| t^{-\hdim}}{(1 + (2 \gamma t)^{-1}|z^{-1}x|)^{\hdim + 1}} \, d\mu_\N(z) \\
    &\leq \frac{\tau}{t} \int_{\N} \frac{|f\bigl ( x \delta_{2 \gamma t}(z) \bigr ) - m_{B^*}f|}{(1 + |z|)^{\hdim + 1}} \; d\mu_\N(z),\numberthis \label{eq:f2 phi tau}
\end{align*}
where we have applied the change of variable $z \mapsto \delta_{(2 \gamma t)^{-1}}(x^{-1}z)$. So, taking into account that for $\widetilde{f}(z) := f\bigl ( x \delta_{2 \gamma t}(z) \bigr )$ the mean over the unit ball at the origin is given by
\begin{align*}
    m_{B_1(e_N)}\widetilde{f} &= \frac{1}{\mu_\N \bigl ( B_1(e_N) \bigr )} \int_{B_1(e_N)} \widetilde{f}(z) \, d\mu_\N(z) = \frac{(2 \gamma t)^{-\hdim}}{\mu_\N \bigl ( B_1(e_N) \bigr )} \int_{B_{2 \gamma t}(x)} f(z) \, d\mu_\N(z) = m_{B^*}f,
\end{align*}
we get 
\begin{align*}
    |f_2 \ast \crk_\tau (y)| \lesssim \frac{\tau}{t} \int_{\N } \frac{|f\bigl ( x \delta_{2 \gamma t}(z) \bigr ) - m_{B^*}f|}{(1 + |z|)^{\hdim + 1}} \; d\mu_\N(z) = \frac{\tau}{t} \int_{\N} \frac{|\widetilde{f}(z) - m_{B_1(e)}\widetilde{f}|}{(1 + |z|)^{\hdim + 1}} \, d\mu_\N(z).
\end{align*}
Since the $\BMO(\N)$-norm is invariant under translations and dilations, we deduce from \eqref{eq:local p-integrable} that
\begin{align*}
   |f_2 \ast \crk_\tau (y)| \lesssim \frac{\tau}{t} \int_{\N} \frac{|\widetilde{f}(z) - m_{B_1(e)}\widetilde{f}|}{(1 + |z|)^{\hdim + 1}} \, d\mu_\N(z) \lesssim \frac{\tau}{t} \| f \|_{\BMO}.
\end{align*}
Using the norm equivalence \eqref{eq:cont_char_2_q}, we thus obtain
\begin{align*} 
    \bigl \| [f_2] \bigr \|_{\F^0_{\infty,2}}^2 &\asymp \sup_{x \in \N, t > 0}  \dashint_{B_t(x)} \int_0^t \bigl |f_2 * \crk_{\tau} (y) \bigr|^2 \frac{d\tau}{\tau} \; d\mu_\N(y) \\
    &\lesssim \|f\|_{\BMO}^2 \sup_{t > 0} \int_0^t \frac{\tau^2}{t^2} \frac{d\tau}{\tau} \asymp \|f\|_{\BMO}^2. \numberthis \label{eq:BMO_estimate_f2}
\end{align*}
Combining \eqref{eq:BMO_estimate_f1} and \eqref{eq:BMO_estimate_f2} gives
\begin{align*}
    \bigl \| [f] \bigr \|_{\F^0_{\infty,2}} \leq \bigl \| [f_1] \bigr \|_{\F^0_{\infty,2}} + \bigl \| [f_2] \bigr \|_{\F^0_{\infty,2}} \lesssim \|f\|_{\BMO},
\end{align*}
which completes the proof of (i).
\\~\\
(ii) Let $f \in \mathcal{S}'(\N)$ such that $[f] = f + \mathcal{P} \in \dot{F}^0_{\infty,2}(\N)$. Let $g$ be an element of the space $\SV(\N)$, which is dense in $H^1(\N)$ by a combination of Corollary~\ref{cor:dense} and Proposition~\ref{prop:Identification of Lebesgue and Hardy}. Since $f * g^\vee = [f] * g^\vee \in C^\infty(\N)$, we can use the Calder\'{o}n reproducing formula \eqref{eq:continuous_calderon} to express
\begin{align*}
    ([f], g) &= f * g^\vee (e_N) = \int_0^\infty f * \crk_t * \crk_t * g^\vee (e_N)\frac{dt}{t} \\
 &= \int_0^\infty \int_\N f * \crk_t(x) \crk_t * g^\vee(x^{-1}) \, d\mu_\N(x) \, \frac{dt}{t} \\
    &= \int_0^\infty \int_\N f * \crk_t(x)  g * \crk^\vee_t(x) \, d\mu_\N(x) \, \frac{dt}{t} \\
    &= \int_0^\infty \int_\N f * \crk_t(x)  g * \crk_t(x) \, d\mu_\N(x) \, \frac{dt}{t}. \numberthis \label{eq:aux_repr_id_BMO_ii}
\end{align*}

In several steps we now derive the continuity estimates that prove (ii). Of crucial use will be the Lusin function
\begin{align*}
    S_\crk g(x) := \biggl ( \int_0^\infty \int_{|x^{-1}y| < t} |g * \crk_t(y)|^2 \, d\mu_\N(y) \, \frac{dt}{t^{\hdim + 1}} \biggr )^{1/2},
\end{align*}
and the continuity of the operator $g \mapsto S_{\phi} g$ from $H^1(\N)$ into $L^1(\N)$, cf. \cite[Theorem 7.8]{FS}.
Let us fix some further auxiliary notions that will prove useful in this context. For arbitrary $x \in \N$ and $s > 0$, let
\begin{align*}
    \Gamma^s(x) := \{ (y, t) \in \N \times (0, \infty) : |x^{-1}y| < t, t < s \}
\end{align*}
be the cone at $x$ truncated at height $s$, and set
\begin{align*}
    S_\crk^s g(x) := \biggl ( \iint_{\Gamma^s(x)} |g * \crk_t(y)|^2 \, d\mu_\N(y) \, \frac{dt}{t^{\hdim + 1}} \biggr )^{1/2}.
\end{align*}
We set
\begin{align*}
    T_\crk f(x) := \sup_{t > 0} \biggl ( \dashint_{B_t(x)} \int_0^t \bigl | f * \crk_\tau(y) \bigr |^2 \frac{d\tau}{\tau} \; d\mu_\N(y) \biggr )^{1/2},
\end{align*}
noting that $\| T_\crk f \|_{L^\infty} = \| [f] \|_{\F^0_{\infty, 2}} < \infty$, and define the ``stopping time'' $s(x)$ for $x \in \N$ by
\begin{align*}
    s(x) := \sup \bigl \{ s > 0 : S_\crk^s f(x) < A \, T_\crk f(x) \bigr \}
\end{align*}
for a constant $A = A(\hdim) > 0$ that shall be fixed soon.

Now, for arbitrary but fixed $y \in \N, t > 0$ we observe that
\begin{align*}
    \bigcup_{z \in B_t(y)} \Gamma^t(z) \subseteq B_{3 \gamma t}(y) \times (0, 3 \gamma t),
\end{align*}
so an application of Fubini's theorem gives
\begin{align*}
    \int_{B_t(y)} S_\crk^t f(x)^2 \, d\mu_\N(x) \leq \mu_\N(B_1(e_N)) \int_{B_{3 \gamma t}(y)} \int_0^{3 \gamma t} |f * \crk_t(y)|^2 \, \frac{dt}{t} \, d\mu_\N(y),
\end{align*}
and hence 
\begin{align*}
    \dashint_{B_t(y)} S_\crk^t f(x)^2 \, d\mu_\N(x)
    &\leq \frac{(3 \gamma)^\hdim \mu_\N(B_1(e_N))}{\mu_\N(B_{3 \gamma t}(y))} \int_{B_{3 \gamma t}(y)} \int_0^{3 \gamma t} |f * \crk_t(y)|^2 \, \frac{dt}{t} \, d\mu_\N(y) \\
    & \leq (3 \gamma)^\hdim \mu_\N(B_1(e_N)) \inf_{x \in B_t(y)} T_\crk f(x).
\end{align*}
If we now choose $A > \mu_\N(B_1(e))$, then the above implies that for
\begin{align*}
    c = c(A) =  \mu_\N(B_1(e_N)) \Bigl ( 1 - (3 \gamma)^\hdim \frac{\mu_\N(B_1(e_N))}{A^2} \Bigr )
\end{align*}
we have
\begin{align*}
    c t^\hdim \leq \mu_\N \bigl ( \{ x \in B_t(y) : s(x) \geq t \} \bigr ) \leq \mu_\N \bigl ( B_t(y) \bigr ) = \mu_\N(B_1(e_N)) t^\hdim.
\end{align*}

Using the above estimate and Fubini's theorem to rewrite the order of integration, this allows us to estimate
\begin{align} \label{eq:aux_est_BMO_ii}
    \int_0^\infty \int_\N H(y, t) t^\hdim \, d\mu_\N(y)  dt \leq c^{-1} \int_\N \bigg( \iint_{\Gamma^{s(x)}(x)} H(y, t) \, d\mu_\N(y)  dt \bigg) \; d\mu_\N(x)
\end{align}
for any nonnegative measurable function $H: \N \times (0, \infty) \to \mathbb{C}$. For the particular choice $H(y, t) := |f * \crk_t(x) g * \crk_t(x)| t^{- (\hdim + 1)}$, the identity \eqref{eq:aux_repr_id_BMO_ii} and estimate \eqref{eq:aux_est_BMO_ii}, together with an application of the Cauchy-Schwarz inequality, give
\begin{align*}
    |([f], g)| &\leq  \int_0^\infty \int_\N |f * \crk_t(x)  g * \crk_t(x)| \, d\mu_\N(x) \, \frac{dt}{t}  \\
    &\leq c^{-1} \int_\N S_\crk^{s(x)} f(x) S_\crk^{s(x)} g(x) \, d\mu_\N(x) \\
    &\leq c^{-1} A \int_\N T_\crk f(x) S_\crk g(x) \, d\mu_\N(x) \\
    &\leq c^{-1} A \| T_\crk f \|_{L^\infty} \| S_\crk g \|_{L^1} \\
    &\lesssim \| [f] \|_{\F^0_{\infty, 2}}  \| g \|_{H^1}.
\end{align*}
By the denseness of $\SV(\N)$ in $H^1(\N)$, we conclude that $[f]$ extends to a continuous linear functional on $H^1(\N)$ and because of $H^1(\N)^* \cong \BMO(\N)$ this functional can be identified with an element of $f' \in \BMO(\N)$. Since $[f] \in \TD/\mathcal{P}$, this implies the existence of a uniquely determined polynomial $P \in \mathcal{P}$ such that $f' = f - P \in \BMO(\N)$ and
\begin{align*}
    \| f - P \|_{\BMO} \lesssim \| [f] \|_{\F^{0}_{\infty, 2}}.
\end{align*}
This shows (ii) and thus completes the proof of the proposition.
\end{proof}

\subsection{Homogeneous Sobolev spaces on graded  groups}
In this subsection, we discuss the (homogeneous) Sobolev spaces on stratified and graded groups as introduced by Folland \cite{Folland1975subelliptic} and Fischer and Ruzhansky \cite{FR2}, respectively. 
Let $\N$ be a graded Lie group and let $\RLO$ be a positive Rockland operator on $\N$ of homogeneous degree $\nu$.  When $\N$ is stratified and $\RLO$ is a sub-Laplacian on $\N$, the inhomogeneous Sobolev spaces were introduced and studied by Folland \cite{Folland1975subelliptic}, in terms of the fractional powers of the sub-Laplacian. The results of \cite{Folland1975subelliptic} were extended to the setting of general graded Lie groups by Fischer and Ruzhansky \cite{FR2}.                  

We recall the definition of homogeneous Sobolev spaces on the graded Lie group $\N$. For $p \in [1,\infty)$, let $\RLO_p$ be the infinitesimal operator of the heat semigroup $\{e^{-t\RLO}\}_{t >0}$ on the Banach spaces $L^p(\N)$. For each $\ord \in \mathbb{R}$, let $\RLO_p^\ord$ be the fractional power of $\RLO_p$, which is defined as the power of a Komatsu-non-negative operator; cf. \cite[Section 4.3]{FR}. For $p =2$, the fractional power $\RLO_2^
\ord$ coincides with the fractional power $\RLO^\ord$ 
defined by the functional calculus
\begin{align}\label{eq:fractional_power_functional_calculus}
\RLO^\ord := \int_0^\infty \lambda^\ord dE_\RLO(\lambda), 
\end{align}
where $E_\RLO$ is the spectral measure associated with $\RLO$. In \cite{FR2,FR}, given $p \in (1,\infty)$ and $\ord \in \mathbb{R}$, the homogeneous Sobolev space $\dot{L}^p_{\ord, \RLO} (\N)$ is defined as the completion of $\mathcal{S}(\N) \cap \Dom (\RLO_p^{\frac{\ord}{\hdeg}})$ with respect to the norm
\begin{align*}
\|f\|_{\dot{L}^p_{\ord, \RLO}} :=\bigl \| \RLO_p^{\frac{\ord}{\hdeg}}f \bigr \|_{L^p} , \quad f \in \mathcal{S}(\N) \cap \Dom(\RLO_p^{\frac{\ord}{\hdeg}}).
\end{align*}
Then $\mathcal{S}(\N) \cap \Dom(\RLO_p^{\frac{\ord}{\hdeg}}) \subseteq \dot{L}^p_{\ord, \RLO}(\N) \subseteq \mathcal{S}'(\N)$ and $\dot{L}^p_{\ord, \RLO}(\N)$ is a Banach space with respect to the norm $\| \cdot \|_{\dot{L}^p_\ord}$, cf. \cite[Proposition 4.4.13]{FR}.

Our main purpose in this subsection is to show the following identification.

\begin{proposition} \label{Prop:identification of homogeneous Sobolev}
Let $\N$ be a graded Lie group and $\RLO$ a Rockland operator on $\N$. Then, for $p \in (1,\infty)$ and $\ord \in \mathbb{R}$,
the space $\dot{L}^p_{\ord, \RLO} (\N)$ is isomorphic to $\F^\ord_{p,2}(\N)$.   
Consequently, the space $\dot{L}^p_{\ord, \RLO} (\N)$ is independent of the choice of Rockland operator $\RLO$.
\end{proposition}

We mention that the fact that the Sobolev space is independent of the choice of Rockland operator was already shown earlier in \cite{FR, FR2}. The norm equivalence provided by Proposition \ref{Prop:identification of homogeneous Sobolev} yields an alternative proof of this fact.

To prove this proposition, we will use the following lemma that allows us to identify the relevant reservoirs.

\begin{lemma} \label{lem:domain_of_Rps}
Let $\N$ be a graded Lie group with homogeneous dimension $\hdim$, and let $\RLO$ be a positive Rockland operator on $\N$ of homegeneous degree $\hdeg$. Let $\ord \in \R$ and let $\RLO^\ord$ be defined by the functional calculus \eqref{eq:fractional_power_functional_calculus}.
\begin{enumerate}[\rm (i)]
    \item For any $\ord \in \mathbb{R}$, we have $\SV(\N) \subseteq \Dom(\RLO^\ord)$ and  $\RLO^\ord(\SV(\N)) \subseteq \SV(\N)$. Moreover, the mapping $\RLO^\ord :\SV(\N)\rightarrow \SV(\N)$ is continuous and bijective.
    \item For any $\ord \in \mathbb{R}$, we have $\SV(\N) \subseteq \bigcap_{1\leq p <\infty} \Dom(\RLO_p^\ord)$. Moreover, for each $p \in [1,\infty)$ and $f \in \SV(\N)$, we have $\RLO_p^\ord f =\RLO^\ord f$.
\end{enumerate}
\end{lemma}

\begin{proof}
(i) Choose a real-valued function $m \in C_c^\infty (\mathbb{R}^+)$ such that
\begin{align*}
C_m := \int_0^\infty t^{-\ord}m(t) \;\frac{dt}{t} \neq 0.     
\end{align*}
Then by a simple change of variables we have
\begin{align*}
\lambda^\ord= C_m^{-1}\hdeg\int_0^\infty t^{-\hdeg \ord}m(t^\hdeg \lambda) \;\frac{dt}{t} 
\end{align*}
for all $\lambda \in \mathbb{R}^+$.
Hence, for any $f \in \Dom(\RLO^\ord)$, the identity
\begin{align} \label{eq:functional_calculus_omega}
\RLO^\ord f
= C_m^{-1}\hdeg\int_0^\infty t^{-\hdeg \ord} m(t^{\hdeg}\RLO)f \; \frac{dt}{t}
\end{align}
holds in $L^2(\N)$.
Since $\RLO^\ord$ is a closed operator, an $L^2$-function $f$ belongs to $\Dom(\RLO^\ord)$ if and only if 
\begin{align} \label{eq:spectral_convergence}
\int_0^\infty t^{-\hdeg \ord} m(t^{\hdeg}\RLO) f\; \frac{dt}{t}
\end{align} 
converges in $L^2(\N)$.

For proving the claims, we let $f \in \SV(\N)$ and will show that  \eqref{eq:spectral_convergence} converges in $\SV(\N)$. To see this, we denote by $\phi$ the convolution kernel of $m(\RLO)$. Then we have the following facts.
First, it follows from Hulanicki's theorem (see \cite{Hulanicki})  that $\phi \in 
\mathcal{S} (\N)$. Second, since $m$ vanishes near the origin, using integration by parts (cf. the proof of \cite[Corollary 1]{geller2006continuous}), it is easy to show that 
\begin{align*}
\int_{\N} \phi (x)P(x) \; d\mu_N(x) =0 
\end{align*}
for all polynomials $P \in \mathcal{P}$, so that $\phi \in \SV(\N)$. Third, the $\nu$-homogeneity of  $\RLO$ implies that the convolution kernel of $m(t^\nu \RLO)$ is $\phi_t$ for any $t \in (0, \infty)$.
Therefore, 
\begin{align} \label{eq:spectral_convolution}
\int_\varepsilon^R t^{-\hdeg \ord} m(t^{\nu}\RLO) f\; \frac{dt}{t} = \int_\varepsilon^R t^{-\hdeg \ord} f \ast \phi_t \; \frac{dt}{t}.
\end{align}
By Lemma \ref{aoe} we have, for any $M,L>0$,
\begin{align} \label{eq:X_f_omega_t} 
\big|X^\alpha \big(t^{-\hdeg \ord}f \ast \phi_t\big)(x)\big| &=t^{-\hdeg \ord} t^{-[\alpha]} \big|f \ast (X^\alpha \phi)_t(x)\big| \nonumber\\
&\lesssim \|f\|_{(k)} t^{-\hdeg \ord} t^{-[\alpha]} (t\wedge t^{-1})^M \frac{(t^{-1}\wedge 1)^\hdim}{(1 + (t^{-1}\wedge 1)|x|)^L} \\
& \leq \|f\|_{(k)}(t\wedge t^{-1})^{M -\hdeg |\ord|-[\alpha] -L}(1 +|x|)^{-L}, \nonumber 
\end{align}
where $k$ is a positive integer depending on $\alpha$, $M$ and $L$.
Given $\ell \in \mathbb{N}$, by choosing $L =\ell$ and $M > \hdeg|\ord| + 2 \ell$ in \eqref{eq:X_f_omega_t}, it follows that 
%for all $0< \varepsilon_1\leq \varepsilon_2\leq 1\leq R_1 \leq R_2 <\infty$,
\begin{equation} \label{eq:Converges_in_Schwartz}
\begin{split}
\int_0^\infty\Bigl \| t^{-\hdeg \ord} m(t^{\nu}\RLO) f\Bigr \|_{(\ell)} \; \frac{dt}{t} 
   &= \int_0^\infty\sup_{[\alpha] \leq \ell, x\in \N}  (1+|x|)^\ell \big| X^\alpha \big(t^{-\hdeg \ord} f \ast \phi_t\big)(x)\big| \; \frac{dt}{t} \\
    &\lesssim \|f\|_{(k)}\int_0^\infty (t \wedge t^{-1})^{M -\hdeg |\ord| -2\ell} \; \frac{dt}{t} \\
    & \lesssim \|f\|_{(k)},
\end{split} 
\end{equation} 
where $k$ is a positive integer depending on $\ell$ and $M$. 
Since $\SV(\N)$ is complete, it follows that \eqref{eq:spectral_convergence} converges in $\SV(\N)$ as $\varepsilon \to 0$ and $R \to \infty$,
thus in particular in $L^2(\N)$, to some $g \in \SV(\N)$. Thus, $f \in \Dom(\RLO^\ord)$ and $\RLO^\ord f = C_m^{-1}\nu g \in \SV(\N)$, cf. the identity \eqref{eq:functional_calculus_omega}.
This shows that $\RLO^\ord : \SV(\N) \to \SV(\N)$ is well-defined. For the continuity, note that from \eqref{eq:Converges_in_Schwartz} it also follows that if $f_j \rightarrow 0$ in $\SV(\N)$ as $j \rightarrow \infty$, then 
\begin{align*}
\RLO^\ord f_j = C_m^{-1}\nu \int_0^\infty t^{-\hdeg \ord} f_j \ast \phi_t \; \frac{dt}{t} \rightarrow 0 \quad \text{in } \SV(\N)
\end{align*}
as $j \rightarrow \infty$.
Therefore, $\RLO^\ord$ is a continuous mapping from $\SV(\N)$ into $\SV(\N)$.  

It remains to show that $\RLO^\ord:\SV(\N) \rightarrow \SV(\N)$ is bijective. First observe that if $f \in \SV(\N)$ is such that $\RLO^\ord f =0$, then $f = \RLO^{-\ord} \RLO^\ord f =0$, which shows the injectivity. Second, given any $h \in \SV(\N)$, set $f: =\RLO^{-\ord}h \in \SV(\N)$. Then $\RLO^\ord f = h$, hence $\RLO^\ord:\SV(\N) \rightarrow \SV(\N)$ is also surjective.
\\~\\
(ii) Fix an arbitrary $p \in [1,\infty)$ and $\ord \in \mathbb{R}$. 
If $\ord \geq 0$, then it follows from \cite[Theorem 4.3.6~(1f)]{FR} that $\mathcal{S}(\N) \subseteq \Dom(\RLO_p^\ord)$, and thus $\SV(\N) \subseteq \Dom(\RLO_p^\ord)$. If $\ord <0$, then we let $k$ be the unique nonnegative integer such that $-k -1 < \ord \leq -k$. 
By  \cite[Theorem 4.3.6~(1b)]{FR}, $\RLO_p^{\ord+k}\RLO_p^{-k} \subseteq \RLO_p^\ord$ in the sense of operator graph. Hence, in order to prove $\SV(\N) \subset \Dom(\RLO_p^\ord)$, it suffices to show that 
\begin{align} \label{eq:domain_Rps_1}
f \in \Dom(\RLO_p^{-k}) \quad \text{for all } f \in \SV(\N)
\end{align}
and 
\begin{align} \label{eq:domain_Rps_2}
\RLO_p^{-k}f \in \Dom(\RLO_p^{\ord+k}) \quad \text{for all } f \in \SV(\N).
\end{align}
To this end, let $f \in \SV(\N)$. By assertion (i), we have $f \in \Dom(\RLO^{-k})$ and 
$g := \RLO^{-k}f \in \SV(\N)$. Consequently, $f = \RLO^k g$. Since $k$ is nonnegative integer,  it follows from 
\cite[Theorem 4.3.6~(1a)]{FR} 
that $\RLO_p^k g= \RLO^kg$. Hence $f =\RLO_p^kg$, which implies that $f \in \Dom(\RLO_p^{-k})$.
This shows \eqref{eq:domain_Rps_1}. 
Since $\RLO_p^{-k}f =g \in \SV(\N)$ and $\ord + k \geq 0$, it follows from 
 \cite[Theorem 3.6~(1f)]{FR2} that $\RLO_p^{-k}f \in \Dom(\RLO_p^{\ord+k})$, i.e., also \eqref{eq:domain_Rps_2} is true. In combination, this shows that $\SV(\N)\subseteq \Dom(\RLO_p^\ord)$ for arbitrary $p \in [1,\infty)$.

For the final claim, we note that \cite[Theorem 4.3.6~(1g)]{FR2} yields $\RLO_p^\ord f = \RLO_2^\ord f$ as $\RLO_2^\ord = \RLO^\ord_2$ and $f \in \Dom(\RLO_p^\ord) \cap \Dom(\RLO_2^\ord)$.
\end{proof}

We are now in a position to prove Proposition \ref{Prop:identification of homogeneous Sobolev}.
 
\begin{proof}[Proof of Proposition \ref{Prop:identification of homogeneous Sobolev}]
Throughout the proof, we will simply write $\dot{L}^p_{\ord}(\N)$ for $\dot{L}^p_{\ord, \RLO} (\N)$. 
Since $\| f \|_{\dot{L}_{\ord}^p} = \| \RLO_p^{\frac{\ord}{\nu} } \|_{L^p}$ for $f \in \SC(\N) \cap \Dom(\RLO_p^{\frac{\ord}{\nu}} )$ and $\SC(\N) \cap \Dom(\RLO_p^{\frac{\ord}{\nu}} )$ is dense in $\dot{L}^p_{\ord}(\N)$, it follows that $\RLO_p^{\frac{\ord}{\nu}}$  extends to a continuous map from $\dot{L}^p_\ord(\N)$  into $L^p(\N)$, and thus $\| \RLO_p^{\frac{\ord}{\nu}} f \|_{L^p} = \| f \|_{\dot{L}^p_\ord}$ for all $f \in \dot{L}^p_{\ord}(\N)$. On the other hand,  for any $f \in \SV(\N)$, letting $g =\RLO_p^{-\frac{\ord}{\nu}}f$ (i.e., $f =\RLO^{\frac{\ord}{\nu
}}g$) we have $\|\RLO_p^{-\frac{\ord}{\mu}}f \|_{\dot{L}^p_{\sigma}} = \|g\|_{\dot{L}^p_\ord} =\|\RLO_p^{\frac{\sigma}{\nu}}g\|_{L^p} = \|f\|_{L^p}$, and hence $\RLO_p^{-\frac{\ord}{\nu}}$ extends to a continuous map  from  $L^p(\N)$ into $\dot{L}^p_\ord(\N)$. 
In combination, this shows that  the mapping $\RLO_p^{\frac{\ord}{\nu}}$, which is  initially defined on $\SC(\N) \cap \Dom(\RLO_p^{\frac{\ord}{\nu}})$, extends to an isomorphism from $\dot{L}^p_\ord(\N)$ onto $L^p(\N)$. 

Furthermore, we have proved in Proposition~\ref{prop:Identification of Lebesgue and Hardy} that $\F^0_{p,2}(\N) = L^p(\N)$ with equivalent norms. Hence, it suffices to show that  $\RLO^{\frac{\ord}{\hdeg}}_p$  extends to an isomorphism from  $\F^\ord_{p,2}(\N)$ onto $\F^0_{p,2}(\N)$ and 
\begin{align} \label{eq:bijective_continuous}
\|\RLO_p^{\frac{\ord}{\hdeg}}f\|_{\F_{p,2}^0} \asymp \|f\|_{\F_{p,2}^{\ord}}
\end{align}
for all $f \in \F^{\ord}_{p,2}(\N)$.
Since $\SV(\N)$ is dense in $\F^\ord_{p,2}(\N)$ and $\F^0_{p,2}(\N)$ (cf. Corollary \ref{cor:dense}) and the mapping $\RLO_p^{\frac{\ord}{\hdeg}}:\SV(\N) \rightarrow \SV(\N)$ is continuous and bijective (cf. Lemma~\ref{lem:domain_of_Rps}), it suffices to show that \eqref{eq:bijective_continuous} holds
for all $f \in \SV(\N)$.

To show the latter, let $m_1 \in C_c^\infty(\mathbb{R}^+)$ be such that 
\begin{align} \label{eq:condition_on_Phi}
\supp m_1 \subseteq [2^{-\hdeg}, 2^\hdeg] \ \ \ \text{and} \ \ \ |m_1(\lambda)| \geq c >0 \ \text{for} \ 
\lambda \in [(3/5)^\hdeg, (5/3)^\hdeg].
\end{align}
These properties of $m_1$ guarantee the existence of $m_2 \in C_c^\infty(\mathbb{R}^+)$
such that 
\begin{align*}
\sum_{j \in \mathbb{Z}} m_1 (2^{-\hdeg j}\lambda) m_2 (2^{-\nu j }\lambda) =1 \quad \text{for all } \lambda \in \mathbb{R}^+.
\end{align*}
By (the proof of) Proposition 
\ref{prop:construction_crk}, it follows that for all $f \in \SV'(\N)$,
\begin{align*}
f =  \sum_{j \in \mathbb{Z}} f \ast \crk_{2^{-j}} \ast \drk_{2^{-j}} \quad \mbox{ in } \SV'(\N),
\end{align*}
where $\crk, \drk \in \SV(\N)$ denote the convolution kernels of $m_1(\RLO)$ and $m_2(\RLO)$, respectively. This shows that the function $\crk$ satisfies the discrete Calder\'{o}n condition \eqref{eq:discrete_calderon}. 
Set $m_3(\lambda):= \lambda^{\frac{\ord}{\hdeg}} m_1(\lambda)$, $\lambda \in \mathbb{R}_+$. Obviously, $m_3 \in C_c^\infty(\mathbb{R}^+)$ and $m_3$ also satisfies the properties \eqref{eq:condition_on_Phi}. Hence $\varphi$, the convolution kernel of $m_3(\RLO)$, also satisfies the discrete Calder\'{o}n condition \eqref{eq:discrete_calderon}. 
 By Lemma \ref{lem:domain_of_Rps}(ii) and Corollary~\ref{cor:spectral_multiplier}, we have, for any $f \in \SV(\N)$,
\begin{align*} 
\|\RLO_p^{\frac{\ord}{\hdeg}} f\|_{\F_{p,2}^0} &=\|\RLO^{\frac{\ord}{\hdeg}} f\|_{\F_{p,2}^0} \\
& \asymp \biggl \| \biggl ( \sum_{j \in \mathbb{Z}} \bigl |  (\RLO^{\frac{\ord}{\hdeg}} f) \ast \crk_{2^{-j}} \bigr |^2 \biggr )^{1/2}
\biggr \|_{L^p}  = \biggl \| \biggl ( \sum_{j \in \mathbb{Z}} \bigl |   m_1(2^{j \nu} \RLO)  (\RLO^{\frac{\ord}{\hdeg}} f)  \bigr |^2 \biggr )^{1/2}
\biggr \|_{L^p} \\
&=   \biggl \| \biggl ( \sum_{j \in \mathbb{Z}} \bigl | 2^{js} m_3(2^{-j \nu} \RLO)  f \bigr |^2 \biggr )^{1/2}
\biggr \|_{L^p}
=  \biggl \| \biggl ( \sum_{j \in \mathbb{Z}} \bigl | 2^{js} f \ast \varphi_{2^{-j}}\bigr |^2 \biggr )^{1/2}
\biggr \|_{L^p}\\
&\asymp\|f\|_{\F^\ord_{p,2}(\N)}.
\end{align*}
This completes the proof.
\end{proof}

\subsection{Lipschitz spaces on stratified groups}  Inhomogeneous Lipschitz spaces on stratified groups were first introduced by Folland \cite{Folland1975subelliptic}. Characterizations of these spaces by Poisson integrals and Littlewood-Paley functions were given in \cite{Folland1979Lipschitz} and \cite{hu2019littlewood}, respectively. 

Homogeneous Lipschitz spaces on any stratified Lie group $\N$ can be defined as follows. For $0 \leq \ord \leq 1$, the homogeneous Lipschitz space $\dot{\Lambda}^\ord(\N)$ is defined to be the space of all $f \in C(\N)$ such that 
\begin{align*}
 \displaystyle \|f\|_{\dot{\Lambda}^\ord} := \begin{cases}
  \displaystyle\sup_{x\in \N} \sup_{y \in \N \backslash \{e\}} \frac{|f(xy)-f(x)|}{|y|^\ord}, \quad & 0< \ord <1,\\
 \displaystyle \sup_{x\in \N} \sup_{y \in \N \backslash \{e\} }\frac{|f(xy) + f(xy^{-1})-2f(x)|}{|y|}, & \ord = 1,
 \end{cases}  
\end{align*}
is finite. 

 To define the space $\dot{
\Lambda
 }^\ord(\N)$ for higher $\ord$ we need to introduce some notation. Let $(\mathfrak{n}_i)_{i \in \mathbb{N}}$ be a stratification of the Lie algebra $\mathfrak{n}$, i.e., an \emph{$\mathbb{N}$-grading} such that $\mathfrak{n}_1$ generates $\mathfrak{n}$. Denoting $n_1 = \dim (\mathfrak{n}_1)$, we fix a basis $\{X_1,\cdots, X_{n_1}\}$ of $\mathfrak{n}_1$ and set
\begin{align*}
\mathcal{I}(n_1) = \bigcup_{k \in \NN_0} \{1,\cdots, n_1\}^k
\end{align*}
to be the set of multi-indices $I$ of arbitrary but finite length with values in $\{1, \ldots, n_1\}$.
For any $k \in \NN_0$ and $I = \{i_1,\ldots, i_k\} \in \{1,\ldots, n_1\}^k \subset \mathcal{I}(n_1)$, we set $|I| =k$ and 
\begin{align*}
    X_I =X_{i_1}\cdots X_{i_k},
\end{align*}
with the convention $X_I ={\rm id}$ when $|I|=0$. Now, for $\ord =k + \ord'$ with $k \in \NN$ and $0 \leq \ord' \leq 1$, the homogeneous Lipschitz space $\dot{\Lambda}^\ord(\N)$ is defined to be the space of all $f \in C^k(\N)$ such that
\begin{align*}
    \|f\|_{\dot{\Lambda}^\ord} := \sum_{I \in \mathcal{I}(n_1), |I| = k}
    \|X_If\|_{\dot{\Lambda}^{\ord'}} <\infty.
\end{align*}

A Littlewood-Paley characterization of inhomogeneous Lipschitz spaces on stratified Lie groups was proven in \cite{hu2019littlewood}. One may use the argument in \cite{hu2019littlewood}, with minor modifications, to also prove a Littlewood-Paley characterization of the homogeneous spaces $\dot{\Lambda}^s(\N)$. Thus one obtains the following identification:
\begin{proposition} \label{prop:LP_characterization_of_hom_Lip}
Let $\N$ be a stratified Lie group and let $s \in \mathbb{R}$. 
\begin{enumerate}[\rm (i)]
    \item If $f \in \dot{\Lambda}^\ord(\N)$, then $f \in \B^{\ord}_{\infty,\infty}(\N)$, and
    \[   \|f\|_{\B^{\ord}_{\infty,\infty}} \lesssim  \|f\|_{\dot{\Lambda}^\ord}.
   \]
   \item If  $[f]=f + \mathcal{P} \in \B^\ord_{\infty,\infty}(\N)$, then there exists a polynomial $P$ such that  $f-P \in \dot{\Lambda}^\ord(\N)$, and  
    \[
    \|f-P\|_{\dot{\Lambda}^\ord}
    \lesssim \big\|[f]\big\|_{\B^\ord_{\infty,\infty}}.
    \]
\end{enumerate}
\end{proposition}

\section*{Acknowledgements}

G.H. is supported by NSF of China (Grant No. 12461018). D. R., M. R. and J. v. V. gratefully acknowledge support from the Research Foundation - Flanders (FWO): Odysseus 1 grant G.0H94.18N and Senior Research Grant G022821N. For J. v. V., this research was funded in whole or in part by the Austrian Science Fund (FWF): 10.55776/J4555 and 10.55776/PAT2545623. J. v. V. also thanks Tommaso Bruno for helpful comments on a draft of this paper.

\end{document}